\documentclass{article}

\usepackage{amsthm}
\usepackage[margin=1.2in]{geometry}
\usepackage{bm}
\usepackage{amsmath,amssymb}
\usepackage{xcolor}
\usepackage{subcaption}
\usepackage{pdflscape}
\usepackage{siunitx}
\usepackage{multirow}
\usepackage[ruled,vlined,linesnumbered]{algorithm2e}
\usepackage{graphicx}
\usepackage{booktabs,makecell}
\usepackage{adjustbox}
\usepackage{diagbox}
\usepackage{float}

\theoremstyle{plain}
\newtheorem{theorem}{Theorem}[section]
\newtheorem{lemma}[theorem]{Lemma}
\newtheorem{corollary}[theorem]{Corollary}
\newtheorem{proposition}[theorem]{Proposition}

\theoremstyle{definition}

\theoremstyle{remark}
\newtheorem{remark}[theorem]{Remark}

\everymath{\displaystyle}

\newcommand{\CG}{\texttt{CG}~}

\begin{document}

\title{Auxiliary Splines Space Preconditioning for B-Splines Finite Elements: The case of $\bm{H}(\bm{curl},\Omega)$ and $\bm{H}(div,\Omega)$ elliptic problems}

\author{A. El Akri \\ Lab. MSDA, Mohammed VIPolytechnic University \\ Ben Guerir,
Morocco \\ e-mail: abdeladim.elakri@um6p.ma \and K. Jbilou \\ Lab. MSDA, Mohammed VI Polytechnic University\\ Ben Guerir, Morocco, and \\ Lab. LMPA, University of Littoral Côte d’Opale, France \\ e-mail: jbilou@univ-littoral.fr, \and A. Ratnani \\ Lab. MSDA, Mohammed VI Polytechnic University\\ Ben Guerir, Morocco \\ e-mail: ahmed.ratnani@um6p.ma}


\date{\today}


\maketitle

\begin{abstract}
This paper presents a study of large linear systems resulting from the regular $B$-splines finite element discretization of the $\bm{curl}-\bm{curl}$ and $\bm{grad}-div$ elliptic problems on unit square/cube domains. We consider systems subject to both homogeneous essential and natural boundary conditions. Our objective is to develop a preconditioning strategy that is optimal and robust, based on the Auxiliary Space Preconditioning method proposed by Hiptmair et al. \cite{hiptmair2007nodal}. Our approach is demonstrated to be robust with respect to mesh size, and we also show how it can be combined with the Generalized Locally Toeplitz (GLT) sequences analysis presented in \cite{mazza2019isogeometric} to derive an algorithm that is optimal and stable with respect to spline degree. Numerical tests are conducted to illustrate the effectiveness of our approach.
\end{abstract}

\section{Introduction}
The {\em Isogeometric Analysis} (IgA) is a mathematical approach that combines {\em Finite Element Methods} (FEMs) with {\em Computer-Aided Design} (CAD) to design and analyze the numerical approximation of {\em Partial Differential Equations} (PDEs). Like FEM, IgA formulates problems through variational methods and specifies a finite-dimensional subspace for the solution. However, IgA employs the same functions used to describe the underlying domain, typically $B$-spline or Non-Uniform Rational $B$-spline (NURBS) functions commonly used in CAD. This approach offers several advantages over FEM, including exact geometry, which eliminates geometric approximation errors, and the use of $B$-spline functions, which makes higher $C^p$-continuous interpolation more practical than standard Lagrange and Hermite polynomials used in FEM.

The field of IgA has rapidly developed in recent years, with significant contributions since the pioneering work by Hughes in 2005 \cite{hughes2005isogeometric}. This approach has been applied in various areas, including electromagnetism \cite{buffa2011isogeometric1,buffa2010isogeometric,ratnani2012arbitrary}, incompressible fluid dynamics \cite{buffa2011isogeometric2}, fluid-structure interaction \cite{bazilevs2012isogeometric,hsu2011blood}, structural and contact mechanics \cite{konyukhov2012geometrically,shojaee2012free}, plasmas physics problems \cite{ratnani2012isogeometric}, and kinetic systems \cite{abiteboul2011solving,back2011axisymmetric,crouseilles2012isogeometric}, among others. For a comprehensive overview, readers can refer to the review paper by Da Veiga et al. \cite{da2014mathematical} and Cottrell et al. \cite{cottrell2009isogeometric}.

Despite the large success of the method,  it is important to note that dealing with higher-order IgA finite elements can be challenging. Specifically,  using higher-order $B$-spline functions can generate huge, sparse, ill-conditioned matrices. Although the discrete systems produced by IgA methods are typically better conditioned than those produced by standard finite elements, their condition numbers cannot be uniformly bounded with respect to the discretization parameter $h$, and can even grow rapidly as $h$ approaches zero. For example, this is the case for the Full-Wave problem with high wave numbers \cite{mazza2019spectral}. (See also \cite{gahalaut2014condition} for explicit bounds of the spectral condition number in the case of the Poisson equation). As a result, direct solvers are not suitable for IgA discrete systems, and even standard iterative methods may fail. Preconditioning is therefore necessary to obtain convergence in a reasonable amount of time.

The literature offers several techniques to address the problem of preconditioning IgA discrete systems. These include overlapping Schwarz preconditioners \cite{da2013isogeometric,da2012overlapping}, non-overlapping decomposition methods \cite{beirao2013bddc,buffa2013bpx,da2014isogeometric,da2017adaptive}, FETI-type preconditioners \cite{bosy2020domain,kleiss2012ieti,pavarino2016isogeometric}, multilevel algorithms \cite{cho2020optimal,de2021two,gahalaut2013algebraic}, multigrid methods \cite{donatelli2015robust,gahalaut2013multigrid}, and preconditioning based on the solution of Sylvester equations \cite{sangalli2016isogeometric}.

A review of the current state of the art indicates a growing interest in developing efficient and rapid IgA preconditioning techniques in recent years. However, most research has focused on scalar elliptic problems, with limited technical generalizations to linear elasticity systems. To the best of our knowledge, only a few papers \cite{mazza2019isogeometric,mazza2019spectral} have studied $\bm{H}(\bm{curl})$ and $\bm{H}(div)$ problems.  In these works, the construction of solvers exploits a detailed spectral analysis of the involved matrices based on the theory of the Generalized Locally Toeplitz (GLT) sequences. However, results of practical interest can be precisely developed only if addressed to specific models. In contrast, this paper presents a more general and systematic approach, providing abstract techniques that can be applied to a broader range of problems.

We shall consider two model problems; the $\bm{curl}-\bm{curl}$ problem: finds a vector field $\bm{u} \, : \, \overline{\Omega} \rightarrow \mathbb{R}^3$ such that
\begin{equation}\label{eq:curl-curl}
\bm{curl}\,  \bm{curl} \, \bm{u} + \tau \bm{u} = \bm{f}, \quad \text{in } \Omega,
\end{equation} 
and the $\bm{grad}-div$ problem: finds $\bm{u} \, : \, \overline{\Omega} \rightarrow \mathbb{R}^3$ such that
\begin{equation}\label{eq:div-div}
- \bm{gard}\, div\, \bm{u} + \tau \bm{u} = \bm{f}, \quad \text{in } \Omega,
\end{equation}
and 
subject to both homogeneous natural and essential boundary conditions and where $\bm{f} \,:\, \Omega  \rightarrow \mathbb{R}^3$ is a vector field and $\tau$ is a small positive parameter. The variational forms of these problems can be written in unified form as follows: Finds $\bm{u} \in \mathcal{H}(\mathcal{D}, \Omega)$ such that 
\begin{equation}\label{introduction-variational-formulation}
(D\bm{u}, D\bm{v})_{\left( L^2(\Omega) \right)^3} + \tau (\bm{u},\bm{v})_{\left( L^2(\Omega) \right)^3} = (\bm{f}, \bm{v})_{\left( L^2(\Omega) \right)^3}, \quad \forall v \in \mathcal{H}(\bm{D}, \Omega),
\end{equation}
where $D \in \{\bm{curl}, div\}$ and the space $\\mathcal{H}(\mathcal{D}, \Omega)$ fulfill the boundary conditions on the case of essential boundary conditions (see the next section for a precise definition).

Preconditioning for these types of problems is particularly challenging. This is because, as pointed out in \cite{hiptmair2007nodal}, the operator $\mathcal{D}$ has a large null space. Unlike the scalar Laplacian operator, which has a null space of dimension one, the kernel of $\mathcal{D}$ has infinite dimension. Another challenge when discretizing \eqref{introduction-variational-formulation} is the loss of coercivity as $\tau \rightarrow 0$. While the continuous problem is well-posed, discrete stability can only be achieved with very fine meshes, which leads to a rapidly growing spectral condition number as $\tau$ approaches $0$. As a result, the preconditioning approach must not only consider the structure of the space $\mathcal{H}(\mathcal{D}, \Omega)$ but also be robust with respect to the parameter $\tau$.

Over the last decades, a promising technique, called Auxiliary Space Preconditioning (ASP) method \cite{chen2015auxiliary,hiptmair2006auxiliary,hiptmair2007nodal,kolev2008auxiliary,kolev2009parallel,nepomnyaschikh1991decomposition,xu1992iterative}, has leads to a general abstract framework for the derivation of stable preconditioners in the case of conforming finite element discretizations. The basic idea of ASP is to transfer the original problem to an auxiliary space where it is easier to solve, then transfer the solution back to the original space and correct the error between the auxiliary space and the full space by applying a smoothing scheme. The choice of the auxiliary space is typically based on a stable decomposition of the space $\mathcal{H}(\mathcal{D}, \Omega)$, known as a regular decomposition \cite{birman1987l2,buffa2002traces,costabel1990remark,hiptmair2007nodal,pasciak2002overlapping,zhao2004analysis}. However, the main challenge in developing the method lies in adapting these regular space decompositions to the discrete level.

The ASP method has already been successfully applied to various preconditioning problems for large-scale finite element systems. In this paper, we extend the method to the isogeometric context, building upon the work presented in \cite{hiptmair2007nodal}. As a first step, we assume that $\Omega = (0, 1)^d$, where $d=2$ or $3$.

The paper is organized as follows. Section \ref{preliminaries} presents the notations, definitions, and preliminary results relevant to our analysis. We introduce the abstract theory of ASP method and briefly recall the notations for $B$-splines spaces and related de Rham sequence. In Section \ref{sec:asp}, we present the main theoretical result of the paper, which is a uniform discrete regular decomposition. We then use this regular decomposition to design robust and efficient ASP preconditioners. Section \ref{sec:numerical-results} provides several numerical examples for both $2$-$d$ and $3$-$d$ cases to illustrate the performance of our preconditioners. Finally, Section \ref{conclusions} concludes the paper.

\begin{remark}
Although the results presented in the paper are applicable to both $2$-$d$ and $3$-$d$ problems, the focus of the analysis is on the $3$-$d$ setting. However, the results for the $2$-$d$ case can be easily derived from those of the $3$-$d$ problems.
\end{remark}

\section{Preliminaries}\label{preliminaries}
In this section, we establish the notation and recall some preliminary results which will be used later in the paper. Firstly, we provide the basic definitions and properties of Sobolev spaces, and we introduce a regular decomposition of space $\mathcal{H}(\mathcal{D}, \Omega)$. This decomposition is critical for our analysis of the discrete $\bm{curl}-\bm{curl}$ and $\bm{grad}-div$ problems. Additionally, we summarize the key aspects of the abstract theory of the Auxiliary Space Preconditioning method. Finally, we introduce the Isogeometric discrete spaces and their relevant properties.

\subsection{Functional Spaces: Notation and Results}\label{subsec:functional-spaces}
In this paper, we will work with Sobolev spaces. We will provide standard notations, but for a more detailed presentation, we refer the reader to \cite{adams2003sobolev,girault2012finite,monk2003finite}. For the unit cube (or square) domain $\Omega$, we denote by $L^2(\Omega)$ the Hilbert space of Lebesgue square-integrable functions on $\Omega$, equipped with the standard $L^2(\Omega)$ norm. Given a positive integer $r$, we denote by $H^r(\Omega)$ the Sobolev space of order $r$ on $\Omega$, which is the space of functions in $L^2(\Omega)$ with $r$th-order derivatives, in the sense of distributions, also in $L^2(\Omega)$, endowed with the standard norm $\|\cdot\|_{H^r(\Omega)}$. By definition, we let $H^0(\Omega)=L^2(\Omega)$. We denote by $H^r_0(\Omega)$ the subspaces of functions with Dirichlet boundary conditions. Note that by definition, we have
$$
L^2_0(\Omega) =\left\{ u \in L^2(\Omega) :\; \int_\Omega u = 0\right\}.
$$
We use boldface letter notation for vectorial spaces, i.e., $\bm{L}^2(\Omega)=\left(L^2(\Omega)\right)^3$, $\bm{H}^r(\Omega)=\left(H^r(\Omega)\right)^3$ and $\bm{H}^r_0(\Omega)=\left(H^r_0(\Omega)\right)^3$.

We also consider the following spaces
\begin{align*}
& \bm{H}( \bm{curl},\, \Omega) =\left\{ \bm{u} \in \bm{L}^2(\Omega),\: \, \bm{curl} \, \bm{u} \in \bm{L}^2(\Omega) \right\},\\
& \bm{H}(div,\, \Omega) =\left\{ \bm{u} \in \bm{L}^2(\Omega),\: \, div \, \bm{u} \in L^2(\Omega) \right\},
\end{align*} 
equipped with their default inner products
\begin{align*}
& (\bm{u}, \bm{v})_{\bm{H}( \bm{curl}, \, \Omega)} = (\bm{u}, \bm{v})_{\bm{L}^2(\Omega)} + (\bm{curl}\,\bm{u}, \bm{curl}\,\bm{v})_{\bm{L}^2(\Omega)},\\
& (\bm{u}, \bm{v})_{\bm{H}( div, \, \Omega)} = (\bm{u}, \bm{v})_{\bm{L}^2(\Omega)} + (div\,\bm{u}, div\,\bm{v})_{L^2(\Omega)}.
\end{align*}
The corresponding norms are denoted by $\|\bm{u}\|_{\bm{H}( \bm{curl}, \, \Omega)}$ and $\|\bm{u}\|_{\bm{H}(div,\, \Omega)}$, respectively.

To deal with the essential boundary conditions, we introduce the spaces
\begin{align*}
& \bm{H}_0( \bm{curl},\, \Omega) = \left\{  \bm{u} \in \bm{H}( \bm{curl},\, \Omega), \; \bm{u}\vert_{\partial \Omega} \times \bm{n} =0\right\},\\
& \bm{H}_0(div,\, \Omega)  = \left\{  \bm{u} \in \bm{H}(div,\, \Omega), \;\bm{u}\vert_{\partial \Omega}\cdot \bm{n} =0\right\},
\end{align*}
where $\bm{n}$ is the unit outward normal of $\partial \Omega$. As subspaces of $\bm{H}( \bm{curl},\, \Omega)$ and $\bm{H}( div,\, \Omega)$,   spaces $\bm{H}_0( \bm{curl},\, \Omega)$ and $\bm{H}_0( div,\, \Omega)$ are endowed with $(\cdot, \cdot)_{\bm{H}( \bm{curl}, \, \Omega)}$ and $(\cdot, \cdot)_{\bm{H}( div, \, \Omega)}$, respectively,  as their default inner products. We write $\|\cdot \|_{\bm{H}_0( \bm{curl}, \, \Omega)}$ and $\| \cdot \|_{\bm{H}_0(div,\, \Omega)}$ for the corresponding norms. It is worth mentioning however that the semi-norms $\|\bm{curl} (\cdot)\|_{ \bm{L}^2(\Omega)}$ and $\|div(\cdot)\|_{ L^2(\Omega)}$ are norms which are equivalent to $\|\cdot \|_{\bm{H}_0( \bm{curl}, \, \Omega)}$ and $\| \cdot \|_{\bm{H}_0(div,\, \Omega)}$ in spaces $\bm{H}_0( \bm{curl},\, \Omega)$ and $\bm{H}_0(div,\, \Omega)$, respectively.

Next we provide regular decomposition for spaces $\bm{H}( \bm{curl}, \, \Omega)$, $\bm{H}(div,\, \Omega)$, \\ $\bm{H}_0( \bm{curl},\, \Omega)$ and $\bm{H}_0(div,\, \Omega)$. For this purpose, following ideas of \cite{hiptmair2007nodal}, we introduce a generic notation $\mathcal{H}(\mathcal{D}, \Omega)$ to indicate any of the four spaces listed above. Here, $\mathcal{D}$ denotes either $\bm{curl}$ or $div$. We also use $\mathcal{D}^-$ and $\mathcal{D}^+$ to represent the differential operators characterizing the null and range spaces of $\mathcal{D}$, respectively. The corresponding Sobolev spaces are denoted by $\mathcal{H}(\mathcal{D}^-, \Omega)$ and $\mathcal{H}(\mathcal{D}^+, \Omega)$. Table \ref{table1} summarizes these notations. \\

With these notations, we have the following result:

\begin{table}[!]
\begin{subtable}{0.3\textwidth} 
    \centering
\begin{tabular}{l|l|l}
$\mathcal{D}$ & $\mathcal{D}^-$ & $\mathcal{D}^+$ \\ \hline
$\bm{curl}$ & $\bm{grad}$ & $div$ \\
$div$ & $\bm{curl}$  & $0$
\end{tabular}
\end{subtable}%
\begin{subtable}{0.3\textwidth} 
    \centering
\begin{tabular}{l|l|l|l}
$\mathcal{H}(\mathcal{D}, \Omega)$ &  $\mathcal{H}(\mathcal{D}^-, \Omega)$ & $\mathcal{H}(\mathcal{D}^+, \Omega)$ & $\bm{X}(\Omega)$ \\ \hline
 $\bm{H}(\bm{curl}, \Omega)$ &  $H^1(\Omega)$ & $\bm{H}(div, \Omega)$ & $\left( H^1(\Omega) \right)^3$\\ 
$\bm{H}_0(\bm{curl}, \Omega)$ & $H^1_0(\Omega)$ &  $\bm{H}_0(div, \Omega)$ & $\left( H^1_0(\Omega) \right)^3$\\
$\bm{H}(div, \Omega)$  &  $\bm{H}(\bm{curl}, \Omega)$ & $L^2(\Omega)$ & $\left( H^1(\Omega) \right)^3$ \\
$\bm{H}_0(div, \Omega)$ &  $\bm{H}_0(\bm{curl}, \Omega)$ & $L^2_0(\Omega)$ & $\left( H^1_0(\Omega) \right)^3$
\end{tabular}
\end{subtable}\hspace{5.cm}
\caption{Definition of operators $\mathcal{D}$, $\mathcal{D}^-$ and $\mathcal{D}^+$, along with their associated Sobolev spaces and of the regular space $\bm{X}(\Omega)$}
\label{table1}
\end{table}

\begin{proposition}
The de Rham complex
\begin{align*}
\begin{array}{cccc}
\mathcal{H}(D^-, \Omega) & \xrightarrow{\quad D^- \quad} & \mathcal{H}(D, \Omega) & \xrightarrow{\quad D \quad} \mathcal{H}(D^+, \Omega)
\end{array}
\end{align*}
is exact, meaning that $D^- \mathcal{H}(D^-, \Omega) = \operatorname{ker}(D)$.
\end{proposition}

Before stating the regular decomposition result of space $\mathcal{H}(\mathcal{D},\Omega)$, let us make the following notation: we denote by $\bm{X}(\Omega)$ one of the two spaces $\left( H^1(\Omega)\right)^3$ or $\left( H^1_0(\Omega)\right)^3$ according to the context (see Table \eqref{table1}). The following theorem is essential, see for instance  \cite{amrouche1998vector,hiptmair2007nodal, pasciak2002overlapping,zhao2004analysis}.

\begin{theorem}[Regular decomposition of $\mathcal{H}(\mathcal{D},\Omega)$]\label{th:cont-regular-decomposition}
For each $\bm{u} \in \mathcal{H}(\mathcal{D},\Omega)$, there exist $\bm{\varphi} \in \bm{X}(\Omega)$ and $\bm{\phi} \in \mathcal{H}(\mathcal{D}^-,\Omega)$ such that
$$
\bm{u}=\bm{\varphi} + \mathcal{D}^- \bm{\phi},
$$
with estimates
\begin{equation}
\|\bm{\varphi}\|_{\bm{L}^2(\Omega)} \leq \|\bm{u}\|_{\bm{L}^2(\Omega)},
\end{equation}
and
\begin{equation}
\|\bm{\varphi}\|_{\bm{X}(\Omega)} \leq C \|D \bm{u} \|_{\bm{L}^2(\Omega)},
\end{equation}
for some positive constant $C$.
\end{theorem}

In Section \ref{sec:asp} (see Theorem \ref{prop:stable-Hitpmair-Xu-decomposition}), we shall show a discrete version of Theorem \ref{th:cont-regular-decomposition}.
 
\subsection{Auxiliary Space Preconditioning (ASP) Method}\label{subsec:Auxiliary-Space-Preconditioning-Method}
This subsection provides a brief overview of the Auxiliary Space Preconditioning (ASP) method. For a more detailed discussion, the reader is refereed to \cite{chen2015auxiliary,hiptmair2006auxiliary,hiptmair2007nodal,kolev2008auxiliary,kolev2009parallel,nepomnyaschikh1991decomposition,xu1992iterative} and the references therein.

Let $V$ be a Hilbert space with an inner product $a: V \times V \rightarrow \mathbb{R}$. The ASP method involves three main components: auxiliary spaces, transfer operators, and a smoother. The auxiliary spaces, denoted as $W_i$ for $i=1, \cdots, I$, are equipped with inner products $a_i: W_i \times W_i \rightarrow \mathbb{R}$. The transfer operators are linear operators $\pi_i: W_i \rightarrow V$ that map the auxiliary spaces to $V$. The smoother is an inner product $s: V \times V \rightarrow \mathbb{R}$ that is distinct from $a$ and is often provided by a relaxation method such as the Jacobi or symmetric Gauss-Seidel schemes.

Given these components, the ASP preconditioner is constructed as
\begin{equation*}
B = S^{-1} + \sum_{i=1}^I \pi_i \circ {A}_i^{-1} \circ \pi_i^*,
\end{equation*}
where $S$ and $A_i$ are linear operators corresponding to the inner products $s$ and $a_i$, respectively, and $\circ$ denotes composition of linear operators. The adjoint operator of $\pi_i$ is denoted as $\pi_i^*$.

Under appropriate assumptions, we prove that $B$ is a valid preconditioner for $A$. Specifically, we have the following result (see \cite[Theorem 2.2]{hiptmair2007nodal}):

\begin{theorem}\label{th:ASP-lemma}
Assume that there are some nonnegative constants $\beta_j$ and $\gamma$ such that
\begin{itemize}
\item[(i)] The continuity of $\pi_j$ with respect to the graph norms:
\begin{equation*}
a\left(\pi_j (w_j),\pi_j (w_j)\right) \leq \beta_j \sum_{i=1}^I a_i(w_j,w_j), \quad \forall w_j \in W_j.
\end{equation*}  
\item[(ii)]  The continuity of $s^{-1}$:
\begin{equation*}
a(v,v)\leq \gamma \,s(v,v), \quad \forall v \in V.
\end{equation*}

\item[(iii)] Existence of a stable decomposition of $V$: for each $v \in V$, there exist $\widetilde{v} \in V$ and $w_i \in W_i$ such that 
\begin{equation*}
v=\widetilde{v}+ \sum_{i=1}^I \pi_i w_i,
\end{equation*}
with estimate 
\begin{equation*}
s(\widetilde{v},\widetilde{v}) + \sum_{i=1}^I a_i(w_i,w_i) \leq \eta \, a(v,v),
\end{equation*}
for some nonnegative (small) constant $\eta$. 
\end{itemize}
Then we have the following estimate for the {\em spectral condition number} of the preconditioned operator
\begin{equation*}
\kappa(BA) \leq \eta \left(\gamma+ \sum_{i=1}^I \beta_i \right).
\end{equation*}
\end{theorem}

The above result highlights the central importance of stable regular decompositions in constructing an efficient auxiliary space preconditioner. In this work, we focus on the discrete case, which requires adapting the regular decomposition of Theorem \ref{th:cont-regular-decomposition} to the discrete level. As a first step, we introduce the discrete spaces in the next section.

\subsection{IsoGeometric Spaces}\label{subsec-IGA}
In this section, we introduce a discrete counterpart of the functional space $\mathcal{H}(\mathcal{D}, \Omega)$ in the context of Isogeometric Analysis \cite{bazilevs2006isogeometric,buffa2011isogeometric,cottrell2009isogeometric,da2014mathematical,hughes2005isogeometric}. We begin by recalling some basic properties of $B$-spline functions and then proceed to construct the IgA discretization of $\bm{curl}$ and $div$ operators. For an introduction to the subject, we refer the reader to standard textbooks on the topic \cite{cohen2001geometric,farin1999nurbs,farin2002curves,gu2008computational,piegl1996nurbs,prautzsch2002bezier,rogers2001introduction,schumaker2007spline}.

Let $T=(t_1,t_2,\ldots,t_{m})$ be a knot vector, which is a non-decreasing sequence of real numbers. The $i$-th $B$-spline of order $p \in \mathbb{N}$ is defined recursively using the {\em Cox-de Boor formula} as follows:
\begin{equation*}
B_{i,0}(t) = 
\begin{cases}
1 \quad &\text{if } t_i \leq t < t_{i+1},\\
0 \quad &\text{otherwise}
\end{cases}
\end{equation*}
\begin{equation*}
B_{i,p}(t)=\frac{t - t_i}{t_{i+p}-t_i} B_{i,p-1}(t) + \frac{t_{i+p+1} - t}{t_{i+p+1}-t_{i+1}} B_{i+1,p-1}(t), 
\end{equation*}
for $i=1, \ldots, n=m-p-1$, where a fraction with zero denominator is assumed to be zero. Following \cite{buffa2011isogeometric}, we introduce also the vector $U=(u_1,\ldots,u_N)$ of breakpoints where $N$ is the number of knots without repetition and the regularity vector $\bm{\alpha}=(\alpha_1, \ldots, \alpha_N) \in \mathbb{N}^N$ in such a way that for each $i \in \{ 1,\ldots,N\}$, the $B$-spline function $B_{i,p}$ is $\alpha_i$ continuously derivable at the breakpoint $u_i$. Note that $\alpha_i=p-r_i$ where $r_i$ is the multiplicity of the break point $u_i$. Throughout the paper, we will only consider  {\em non-periodic} knot vectors
\begin{equation*}
T=(\underbrace{0,\ldots,0}_{p+1}, t_{p+2}, \ldots, t_{m-p-1}, \underbrace{1,\ldots,1}_{p+1}),
\end{equation*}  
{and we suppose that $1 \leq r_i \leq p$. In this way we guarantee that $0 \leq \alpha_i \leq p$ where the minimal regularity $\alpha_i=0$ corresponds to a continuity at knot $u_i$.} This allow us to introduce  the uni-variate spline spaces
\begin{equation*}
\mathcal{S}^p_{\bm{\alpha}}= span\left\{ B_{i,p} \,:\, i=1,\ldots,n\right\}, \quad \mathcal{S}^p_{\bm{\alpha},0}= span\left\{ B_{i,p} \,:\, i=2,\ldots,n-1\right\}.
\end{equation*} 
Note that all the elements of space $\mathcal{S}^p_{\bm{\alpha},0}$ vanish at the boundary of  $(0,1)$ (by definition). Hence, the space is suited for dealing with homogeneous Dirichlet boundary conditions.

These definitions can be generalized to the multivariate case $\Omega=(0,1)^3$ by {\em tensorization}: With a tridirectional knot vector $\bm{T}=T_1 \times T_2 \times T_3$ at hand, where
\begin{equation*}
T_i=(\underbrace{0,\ldots,0}_{p_i+1}, t_{i,p_i+2}, \ldots, t_{i,m_i-p_i-1}, \underbrace{1,\ldots,1}_{p_i+1}),\quad m_i, p_i \in \mathbb{N}, \; i=1,2,3,
\end{equation*} 
is an open univariate knot vector, we define the {\em tensor-product spline} space by
\begin{equation*}
\bm{\mathcal{S}}^{p_1,p_2,p_3}_{\bm{\alpha}_1,\bm{\alpha}_2,\bm{\alpha}_3}= \mathcal{S}^{p_1}_{\bm{\alpha}_1} \otimes \mathcal{S}^{p_2}_{\bm{\alpha}_2} \otimes \mathcal{S}^{p_3}_{\bm{\alpha}_3},
\end{equation*} 
where $\bm{\alpha}_i$ is the regularity vector related to knot $T_i$, with $i=1,2,3$. However, we shall also assume our mesh to be {\em locally quasi-uniform}, meaning, there exists a constant $\theta \geq 1$ such that for all $i\in \{1,2,3\}$ we have 
\begin{equation*}
\frac{1}{\theta} \leq \frac{h_{i,j_i}}{h_{i,j_i+1}} \leq \theta, \quad j_i=1, \ldots, N_i-2,
\end{equation*}
where $N_i$ is the number of $T_i$-knots without repetition and $h_{i,j_i}=t_{i,j_i+1}-t_{i,k_{j_i}}$, with $k_{j_i}=\max \{l\,:\,t_l < t_{i,j_i+1}\}$.

With these notations, the $3$-$d$ approximations spaces are given  by (see, e.g. \cite{buffa2011isogeometric,da2014mathematical})
$$
\left\{\begin{array}{l}
V_{h}(\textbf{grad},\Omega) =\bm{\mathcal{S}}^{p_1,p_2,p_3}_{\bm{\alpha}_1,\bm{\alpha_2},\bm{\alpha_3}}, \vspace{0.25cm}\\
\bm{V}_{h}(\bm{curl},\Omega) = \bm{\mathcal{S}}^{p_1-1,p_2,p_3}_{\bm{\alpha}_1-1,\bm{\alpha_2},\bm{\alpha_3}} \times \bm{\mathcal{S}}^{p_1,p_2-1,p_3}_{\bm{\alpha}_1,\bm{\alpha_2}-1,\bm{\alpha_3}} \times \bm{\mathcal{S}}^{p_1,p_2,p_3-1}_{\bm{\alpha}_1,\bm{\alpha_2},\bm{\alpha_3}-1}, \vspace{0.25cm}\\
\bm{V}_{h}(div,\Omega) = \bm{\mathcal{S}}^{p_1,p_2-1,p_3-1}_{\bm{\alpha}_1,\bm{\alpha_2}-1,\bm{\alpha_3}-1} \times \bm{\mathcal{S}}^{p_1-1,p_2,p_3-1}_{\bm{\alpha}_1-1,\bm{\alpha_2},\bm{\alpha_3}-1} \times \bm{\mathcal{S}}^{p_1-1,p_2-1,p_3}_{\bm{\alpha}_1-1,\bm{\alpha_2}-1,\bm{\alpha_3}}, \vspace{0.25cm}\\
V_{h}(L^2,\Omega) = \bm{\mathcal{S}}^{p_1-1,p_2-1,p_3-1}_{\bm{\alpha}_1-1,\bm{\alpha_2}-1,\bm{\alpha_3}-1},
\end{array}  \right.
$$
where $h$ refers to the global mesh size, i.e $h=\max_{\substack{1 \leq j_i \leq N_i-1 \\ i=1,2,3}}h_{i,j_i}$. Let 
$$
\left\{\begin{array}{l}
V_{h,0}(\textbf{grad},\Omega) = V_{h}(\textbf{grad},\Omega) \cap H^1_0(\Omega) \vspace{0.25cm}\\
\bm{V}_{h,0}(\bm{curl},\Omega) = \bm{V}_{h}(\bm{curl},\Omega) \cap \bm{H}_0(\bm{curl}, \, \Omega), \vspace{0.25cm}\\
\bm{V}_{h,0}(div,\Omega) = \bm{V}_{h}(div,\Omega) \cap \bm{H}_0(div, \, \Omega), \vspace{0.25cm}\\
V_{h,0}(L^2, \Omega) = V_h(L^2, \Omega) \cap L_0^2(\Omega),
\end{array}  \right.
$$
for spaces with essential boundary conditions.

\begin{remark}
Since we work on the parametric domain $(0, 1)^3$, we have 
$$
\left\{\begin{array}{l}
V_{h, 0}(\textbf{grad},\Omega) =\mathcal{S}^{p_1}_{\bm{\alpha}_1, 0} \otimes \mathcal{S}^{p_2}_{\bm{\alpha}_2, 0} \otimes \mathcal{S}^{p_3}_{\bm{\alpha}_3, 0}, \vspace{0.25cm}\\
\bm{V}_{h, 0}(\bm{curl},\Omega)= 
\begin{pmatrix}
\mathcal{S}^{p_1-1}_{\bm{\alpha}_1-1} \otimes \mathcal{S}^{p_2}_{\bm{\alpha}_2, 0} \otimes \mathcal{S}^{p_3}_{\bm{\alpha}_3, 0}\\
 \mathcal{S}^{p_1}_{\bm{\alpha}_1, 0} \otimes \mathcal{S}^{p_2-1}_{\bm{\alpha}_2-1} \otimes \mathcal{S}^{p_3}_{\bm{\alpha}_3, 0}\\
\mathcal{S}^{p_1}_{\bm{\alpha}_1, 0} \otimes \mathcal{S}^{p_2}_{\bm{\alpha}_2, 0} \otimes \mathcal{S}^{p_3-1}_{\bm{\alpha}_3-1}
\end{pmatrix} \vspace{0.25cm}\\
\bm{V}_{h, 0}(div,\Omega) = 
\begin{pmatrix}
\mathcal{S}^{p_1}_{\bm{\alpha}_1,0} \otimes \mathcal{S}^{p_2-1}_{\bm{\alpha}_2-1} \otimes \mathcal{S}^{p_3-1}_{\bm{\alpha}_3-1}\\
\mathcal{S}^{p_1-1}_{\bm{\alpha}_1-1} \otimes \mathcal{S}^{p_2}_{\bm{\alpha}_2,0} \otimes \mathcal{S}^{p_3-1}_{\bm{\alpha}_3-1}\\
\mathcal{S}^{p_1-1}_{\bm{\alpha}_1-1} \otimes \mathcal{S}^{p_2-1}_{\bm{\alpha}_2-1} \otimes \mathcal{S}^{p_3}_{\bm{\alpha}_3,0}
\end{pmatrix}.
\end{array}  \right.
$$

\end{remark}

Now {\em de Rham diagrams} can be constructed. Among the important properties, one can build specific projectors, what is called {\em quasi interpolation operators}, that make these diagrams commute. We shall start with the univariate case, then extend it by tensor product. For this purpose, we take any locally stable projector $\mathcal{P}_{h} \,:\, H^1(0,1) \longrightarrow \mathcal{S}^p_{\bm{\alpha}}$, for instance see \cite{buffa2011isogeometric,schumaker2007spline} for theoretical studies, then we define the corresponding 
histopolation operator by 
\begin{equation*}
\mathcal{Q}_{h} \phi = \frac{d}{d x} \mathcal{P}_{h} \left( \int_0^x \phi(t) dt \right), \quad \phi \in L^2(0,1).
\end{equation*} 
Following the notations above, the quasi interpolation operators are given by  
$$
\left\{\begin{array}{c}
\Pi^{\textbf{grad}}_h=\mathcal{P}_h \otimes \mathcal{P}_h \otimes \mathcal{P}_h, \vspace{0.25cm}\\
\Pi^{\bm{curl}}_h=
\begin{pmatrix}
\mathcal{Q}_h \otimes \mathcal{P}_h \otimes \mathcal{P}_h\\
\mathcal{P}_h \otimes \mathcal{Q}_h \otimes \mathcal{P}_h\\
\mathcal{P}_h \otimes \mathcal{P}_h \otimes \mathcal{Q}_h
\end{pmatrix},\vspace{0.25cm}\\
\Pi^{div}_h=
\begin{pmatrix}
\mathcal{P}_h \otimes \mathcal{Q}_h \otimes \mathcal{Q}_h\\
\mathcal{Q}_h \otimes \mathcal{P}_h \otimes \mathcal{Q}_h\\
\mathcal{Q}_h \otimes \mathcal{Q}_h \otimes \mathcal{P}_h,
\end{pmatrix}.\vspace{0.25cm}\\
\Pi^{L^2}_h=\mathcal{Q}_h \otimes \mathcal{Q}_h \otimes \mathcal{Q}_h.
\end{array}  \right.
$$
(here the notation $\otimes$ express the composition of operators on each coordinate). 

The case with boundary conditions follows the same rationals. In fact, in this case one simply replace $\mathcal{P}_{h}$ by a  locally stable projector preserving boundary conditions $\mathcal{P}_{h,0} \,:\, H^1_0(0,1) \longrightarrow \mathcal{S}^p_{\bm{\alpha},0}$ (see \cite{buffa2011isogeometric}) and modifies  the projector $\mathcal{Q}_{h}$ as follows
\begin{equation*}
\mathcal{Q}_{h,0} \phi = \frac{d}{d x} \mathcal{P}_{h,0} \left( \int_0^x \phi(t) dt \right), \quad \phi \in L^2_0(0,1). 
\end{equation*}  
Let then

$$
\left\{\begin{array}{c}
\Pi^{\textbf{grad}}_{h, 0}=\mathcal{P}_{h, 0} \otimes \mathcal{P}_{h, 0} \otimes \mathcal{P}_{h, 0}, \vspace{0.25cm}\\
\Pi^{\bm{curl}}_{h, 0}=
\begin{pmatrix}
\mathcal{Q}_{h, 0} \otimes \mathcal{P}_{h, 0} \otimes \mathcal{P}_{h, 0}\\
\mathcal{P}_{h, 0} \otimes \mathcal{Q}_{h, 0} \otimes \mathcal{P}_{h, 0}\\
\mathcal{P}_{h, 0} \otimes \mathcal{P}_{h, 0} \otimes \mathcal{Q}_{h, 0}
\end{pmatrix},\vspace{0.25cm}\\
\Pi^{div}_{h, 0}=
\begin{pmatrix}
\mathcal{P}_{h, 0} \otimes \mathcal{Q}_{h, 0} \otimes \mathcal{Q}_{h, 0}\\
\mathcal{Q}_{h, 0} \otimes \mathcal{P}_{h, 0} \otimes \mathcal{Q}_{h, 0}\\
\mathcal{Q}_{h, 0} \otimes \mathcal{Q}_{h, 0} \otimes \mathcal{P}_{h, 0},
\end{pmatrix}.\vspace{0.25cm}\\
\Pi^{L^2}_{h, 0}=\mathcal{Q}_{h, 0} \otimes \mathcal{Q}_{h, 0} \otimes \mathcal{Q}_{h, 0}.
\end{array}  \right.
$$

Next, we provide some approximation error results, for this purpose, it is more suitable to use an unified  presentation. Thus, as in subsection  \ref{subsec:functional-spaces} we write $D$ for either $\bm{curl}$ or $div$, and in the case with essential boundary conditions we will drop the index $0$ (see Table \eqref{table2}). We have then (see \cite[Proposition 4.5]{buffa2011isogeometric})

\begin{table}[H]
\centering
\begin{tabular}{l|l|l|l}
$V_h(\mathcal{D}, \Omega)$ &  $V_h(\mathcal{D}^-, \Omega)$ & $V_h(\mathcal{D}^+, \Omega)$ & $\bm{X}_h(\Omega)$ \\ \hline
$\bm{V}_h(\bm{curl}, \Omega)$ &  $V_h(\bm{grad}, \Omega)$ & $\bm{V}_h(\bm{curl}, \Omega)$ & $\left( V_h(\bm{grad}, \Omega) \right)^3$\\ 
$\bm{V}_{h,0}(\bm{curl}, \Omega)$ & $V_{h,0}(\bm{grad}, \Omega)$ &  $\bm{V}_{h,0}(\bm{curl}, \Omega)$ & $\left( V_{h,0}(\bm{grad}, \Omega) \right)^3$\\
$\bm{V}_h(div, \Omega)$  &  $\bm{V}_h(\bm{curl}, \Omega)$ & $V_{h}(L^2, \Omega)$ & $\left( V_h(\bm{grad}, \Omega) \right)^3$ \\
$\bm{V}_{h,0}(div, \Omega)$ &  $\bm{V}_{h,0}(\bm{curl}, \Omega)$ & $V_{h,0}(L^2, \Omega)$ & $\left( V_{h,0}(\bm{grad}, \Omega) \right)^3$
\end{tabular}
\caption{Spaces $V_h(\mathcal{D}, \Omega)$, $V_h(\mathcal{D}^-, \Omega)$, $V_h(\mathcal{D}^+, \Omega)$ and $\bm{X}_h(\Omega)$}
\label{table2}
\end{table}

\begin{proposition}
The diagram shown below is exact and commutes:
\begin{align}\label{pr-eq:de Rham-diagram}
    \begin{array}{ccccc}
   \mathcal{H}(\mathcal{D}^-, \Omega) & \xrightarrow{\quad \mathcal{D}^- \quad } & \mathcal{H}(\mathcal{D}, \Omega) & \xrightarrow{\quad \mathcal{D} \quad } & \mathcal{H}(\mathcal{D}^+, \Omega)\\
  \Pi_{h}^{\mathcal{D}^-} \Bigg\downarrow &  & \Pi_{h}^{\mathcal{D}} \Bigg\downarrow &  & \Pi_{h}^{\mathcal{D}^+} \Bigg\downarrow \\
  V_{h}(\mathcal{D}^-,\Omega) & \xrightarrow{\quad \mathcal{D}^- \quad } & V_{h}(\mathcal{D},\Omega) & \xrightarrow{\quad \mathcal{D} \quad } & V_{h}(\mathcal{D}^+,\Omega)
  \end{array}
\end{align}

Finally, we shall need the following approximation result (see \cite[Theorem 5.3]{buffa2011isogeometric})
\begin{theorem}\label{error_estimates}
Suppose $l$ and $r$ are integers satisfying $0 \leq l \leq r \leq \underline{p}$ and $l \leq \underline{\alpha}$, where $\underline{p}$ is the minimum of $p_1$, $p_2$, and $p_3$, and $\underline{\alpha}$ is the minimum of $\bm{\alpha}_1$, $\bm{\alpha}_2$, and $\bm{\alpha}_3$. Then, the following inequalities hold true
\begin{align*}
& \left \|\varphi - \Pi_{h}^{\mathcal{D}}\varphi \right \|_{H^l(\Omega)} \leq C h^{r-l}\|\varphi\|_{H^s(\Omega)}, \quad & &\forall \varphi \in H^r(\Omega),\\
& \left \|\bm{u} - \Pi_{h}^{\bm{grad}}\bm{u} \right \|_{\bm{H}^l(\Omega)} \leq C h^{r-l}\|\bm{u}\|_{\bm{H}^s(\Omega)}, & &\forall \bm{u} \in \bm{H}^r(\Omega),\\
& \left \|\varphi - \Pi_{h,0}^{\mathcal{D}}\varphi \right \|_{H^l(\Omega)} \leq C h^{r-l}\|\varphi\|_{H^s(\Omega)}, & &\forall \varphi \in \mathcal{H}_0(\mathcal{D}, \Omega) \cap H^r(\Omega),\\
& \left \|\bm{u} - \Pi_{h,0}^{\bm{grad}}\bm{u} \right \|_{\bm{H}^l(\Omega)} \leq C h^{r-l}\|\bm{u}\|_{\bm{H}^s(\Omega)}, & &\forall \bm{u} \in \bm{H}^1_0(\Omega) \cap \bm{H}^r(\Omega).
\end{align*}
Here, $C$ is a positive constant that does not depend on $h$.
\end{theorem}

\end{proposition}

\section{Auxiliary Space Preconditioners}\label{sec:asp}
The aim of this section is to develop a suitable auxiliary space preconditioner for the $\bm{curl}-\bm{curl}$ and $\bm{grad}-div$ problems. As mentioned earlier, the main challenge is to drive a discrete version of the regular decomposition presented in Theorem \ref{th:cont-regular-decomposition}; known as the Hitmair-Xu decomposition. The section is divided into two subsections. In Subsection \ref{subsec:asp}, we focus on the discrete Hitmair-Xu decomposition. The outcome of this subsection is later employed in Subsection \ref{sec:construction-of-ASP} to construct the ASP preconditioner.

Throughout this section, we use the notation $A \lesssim B$ to indicate the existence of a constant $C > 0$, independent of $h$ and $\tau$, such that $A \leq CB$. If $A \lesssim B$ and $B \lesssim A$, we write $A \approx B$.
 
\subsection{Discrete Decompositions}\label{subsec:asp}

We need the following preliminary results in order to prove Hitpmair-Xu decomposition stated in Proposition \ref{prop:Hitpmair-Xu decomposition}.

\begin{lemma}\label{thm:sec3-lemma1}
For every $\bm{\varphi} \in \bm{X}(\Omega)$ such that $\mathcal{D} \bm{\varphi} \in V_{h}(\mathcal{D}^+,\Omega)$, we have
\begin{itemize}
\item[(i)] $\Pi_{h}^{\mathcal{D}} \bm{\varphi}$ is well-defined.

\item[(ii)] $\mathcal{D} \bm{\varphi} = \mathcal{D} \left( \Pi_{h}^{\mathcal{D}} \bm{\varphi}\right)$.

\item[(iii)] $\left\| \bm{\varphi}-\Pi_{h}^{\mathcal{D}} \bm{\varphi} \right\|_{\bm{L}^2(\Omega)} \lesssim h \|\bm{\varphi}\|_{\bm{X}(\Omega)}$.
\end{itemize}
\end{lemma}

\begin{proof}
First insertion is a consequence of the fact that $\bm{X}(\Omega) \subset \mathcal{H}(\mathcal{D},\Omega)$. Concerning (ii), using the commutativity of Diagram \eqref{pr-eq:de Rham-diagram}, we obtain
$$
\mathcal{D}\left( \Pi_{h}^{\mathcal{D}} \bm{\varphi}\right)= \Pi_{h}^{\mathcal{D}^+} \left( \mathcal{D} \bm{\varphi} \right).
$$
We now use $\mathcal{D} \bm{\varphi}\in \bm{V}_{h}(\mathcal{D}^+,\Omega)$ to obtain (ii). Estimate (iii) follows from Theorem \ref{error_estimates}.
\end{proof}

\begin{lemma}\label{lem-semi-discrete-decomposition}
For each $\bm{u}_h \in \bm{V}_{h}(\mathcal{D},\Omega)$, there exist $\bm{\varphi} \in \bm{X}(\Omega)$ and $\bm{\phi}_h \in V_{h}(\mathcal{D}^-,\Omega)$ such that
\begin{equation}\label{eq:semi-discrete-decomposition}
\bm{u}_h = \Pi_{h}^{\mathcal{D}} \bm{\varphi} + \mathcal{D}^- \bm{\phi}_h,
\end{equation}
with estimates
\begin{equation}\label{eq:discrete-H-curl-estimate}
\|\bm{\varphi}\|_{\bm{X}(\Omega)} \lesssim \|\mathcal{D}  \bm{u}_h\|_{\bm{L}^2(\Omega)}, \quad \|\bm{\varphi}\|_{\bm{L}^2(\Omega)} \leq \|  \bm{u}_h\|_{\bm{L}^2(\Omega)},
\end{equation}
\begin{equation}\label{eq:stability-of-semi-discrete-decomposition}
\text{and} \quad \|\Pi_{h}^{\mathcal{D}} \bm{\varphi}\|_{\bm{L}^2(\Omega)} + \|\mathcal{D}^- \bm{\phi}_h \|_{\bm{L}^2(\Omega)}\lesssim \|\bm{u}_h\|_{\bm{L}^2(\Omega)}.
\end{equation}
\end{lemma}

\begin{proof}
Let $\bm{u}_h \in \bm{V}_{h}(\mathcal{D},\Omega)$. According to Theorem \ref{th:cont-regular-decomposition}, there exists $\bm{\varphi} \in \bm{X}(\Omega)$ such that

\begin{equation}\label{eq:semi-discrete-regular-decomposition}
\left\{ \begin{array}{l}
\mathcal{D} \bm{\varphi} = \mathcal{D} \bm{u}_h  \in \bm{V}_{h}(\mathcal{D}^+,\Omega), \vspace{0.2cm}\\
\|\bm{\varphi}\|_{\bm{X}(\Omega)} \lesssim \|\mathcal{D} \bm{u}_h \|_{\bm{L}^2(\Omega)}, \vspace{0.2cm}\\
\|\bm{\varphi}\|_{\bm{L}^2(\Omega)} \leq \|\bm{u}_h\|_{\bm{L}^2(\Omega)}.
\end{array} \right.
\end{equation}
We now apply Lemma  \ref{thm:sec3-lemma1} and obtain
$$
\mathcal{D} \bm{u}_h = \mathcal{D} \bm{\varphi} = \mathcal{D}\left( \Pi_{h}^{\mathcal{D}} \bm{\varphi}\right),
$$
hence, 
$$\bm{u}_h - \Pi_{h}^{\mathcal{D}} \bm{\varphi} \in \textbf{ker}\left(\mathcal{D} \mid_{\bm{V}_{h}(\mathcal{D},\Omega)} \right)=\mathcal{D}^- \left( V_{h}(\mathcal{D}^-,\Omega)\right).$$ 
Therefore, there exists $\bm{\phi}_h \in V_{h}(\mathcal{D}^-,\Omega)$  such that $\bm{u}_h - \Pi_{h}^{\mathcal{D}} \bm{\varphi}=\mathcal{D}^- \bm{\phi}_h$,  which yields \eqref{eq:semi-discrete-decomposition}.  Property \eqref{eq:discrete-H-curl-estimate} follows from estimates in \eqref{eq:semi-discrete-regular-decomposition}.

We now show \eqref{eq:stability-of-semi-discrete-decomposition}. We write
\begin{eqnarray*}
\|\Pi_{h}^{\mathcal{D}} \bm{\varphi}\|_{\bm{L}^2(\Omega)} & \leq & \|\Pi_{h}^{\mathcal{D}} \bm{\varphi} -\bm{\varphi}\|_{\bm{L}^2(\Omega)}+\|\bm{\varphi}\|_{\bm{L}^2(\Omega)}\\
& \lesssim & h \|\bm{\varphi}\|_{\bm{X}(\Omega)} +  \|\bm{\varphi}\|_{\bm{L}^2(\Omega)}\\
& \lesssim &  h \|\mathcal{D} \bm{u}_h\|_{\bm{L}^2(\Omega)} + \|\bm{u}_h\|_{\bm{L}^2(\Omega)},
\end{eqnarray*}
where in the last estimate we have used first and second estimates in \eqref{eq:semi-discrete-regular-decomposition}. Moreover, using the inverse inequality
\begin{equation}\label{eq:inverse-inequlity-curl}
\|\mathcal{D} \bm{u}_h \|_{\bm{L}^2(\Omega)} \lesssim h^{-1} \|\bm{u}_h\|_{\bm{L}^2(\Omega)},
\end{equation}
we get 
$$
\|\Pi_{h}^{\mathcal{D}} \bm{\varphi}\|_{\bm{L}^2(\Omega)} \lesssim \|\bm{u}_h\|_{\bm{L}^2(\Omega)}.
$$
On the other hand, we have 
\begin{eqnarray*}
\|\mathcal{D}^- \bm{\phi}_h\|_{\bm{L}^2(\Omega)}  &=& \| \bm{u}_h - \Pi_{h}^{\mathcal{D}} \bm{\varphi}\|_{\bm{L}^2(\Omega)}\\
& \lesssim & \| \bm{u}_h \|_{\bm{L}^2(\Omega)} + \| \Pi_{h}^{\mathcal{D}} \bm{\varphi}\|_{\bm{L}^2(\Omega)} \lesssim \| \bm{u}_h \|_{\bm{L}^2(\Omega)},
\end{eqnarray*}
and  inequality \eqref{eq:stability-of-semi-discrete-decomposition} is proved.
\end{proof}

Let $\bm{X}_h(\Omega)$ denotes one of the two discrete spaces $\left(V_h(\bm{grad}, \Omega)\right)^3$ or $\left(V_{h,0}(\bm{grad}, \Omega)\right)^3$, depending if we work with Dirichlet or Neumann boundary condition type (see Table \ref{table2}). 

\begin{lemma}\label{lem-stable-approximation-of-vect-grad}
Every $\bm{\varphi} \in \bm{X}(\Omega)$ admits a stable approximation $\bm{\varphi}_h \in \bm{X}_h(\Omega)$ satisfying
$$
\|\bm{\varphi}_h\|_{\bm{L}^2(\Omega)} \lesssim \|\bm{\varphi}\|_{\bm{L}^2(\Omega)}, \quad \text{and} \quad 
h^{-1} \|\bm{\varphi} -\bm{\varphi}_h \|_{\bm{L}^2(\Omega)} + \|\bm{\varphi}_h\|_{\bm{X}(\Omega)} \lesssim \|\bm{\varphi}\|_{\bm{X}(\Omega)}.
$$  
\end{lemma}

\begin{proof}
Let $\bm{\varphi}:=\left(\varphi^1,\varphi^2,\varphi^3\right) \in \bm{X}(\Omega)$ and define 
$$\bm{\varphi}_h= \left(\Pi_{h}^{\bm{grad}} \varphi^1,\Pi_{h}^{\bm{grad}} \varphi^2,\Pi_{h}^{\bm{grad}} \varphi^3\right) \in \bm{X}_h(\Omega).$$
According to Theorem \ref{error_estimates}, we have:
\begin{equation}\label{eq1}
\|\varphi^k - \Pi_{h}^{\bm{grad}} \varphi^k\|_{L^2(\Omega)} \lesssim   \|\varphi^k\|_{L^2(\Omega)}, \quad k=1,2,3,
\end{equation}
\begin{equation}\label{eq2}
\|\varphi^k - \Pi_{h}^{\bm{grad}} \varphi^k\|_{L^2(\Omega)} \lesssim  h \|\varphi^k\|_{H^1(\Omega)}, \quad k=1,2,3, 
\end{equation}
and 
\begin{equation}\label{eq3}
\|\varphi^k - \Pi_{h}^{\bm{grad}} \varphi^k\|_{H^1(\Omega)} \lesssim \|\varphi^k\|_{H^1(\Omega)}, \quad k=1,2,3. 
\end{equation}
Using \eqref{eq1}, we get 
\begin{eqnarray*}
&& \|\bm{\varphi} -\bm{\varphi}_{h} \|_{\bm{L}^2(\Omega)}^2 \nonumber \\
&& = \|\varphi^1 - \Pi_{h}^{\bm{grad}} \varphi^1\|_{L^2(\Omega)}^2+\|\varphi^2 - \Pi_{h}^{\bm{grad}} \varphi^2\|_{L^2(\Omega)}^2+\|\varphi^3 - \Pi_{h}^{\bm{grad}} \varphi^3\|_{L^2(\Omega)}^2 \nonumber \\
&& \lesssim   \|\varphi^1\|_{L^2(\Omega)}^2 + \|\varphi^2\|_{L^2(\Omega)}^2 + \|\varphi^3\|_{L^2(\Omega)}^2 = \|\bm{\varphi}\|_{\bm{L}^2(\Omega)}^2.
\end{eqnarray*}
From which we deduce that 
\begin{equation*}
\|\bm{\varphi}_h\|_{\bm{L}^2(\Omega)} \leq \|\bm{\varphi} -\bm{\varphi}_h \|_{\bm{L}^2(\Omega)} + \|\bm{\varphi}\|_{\bm{L}^2(\Omega)} \lesssim  \|\bm{\varphi}\|_{\bm{L}^2(\Omega)},
\end{equation*}
which proves the first inequality. 

Similarly, using \eqref{eq2} and \eqref{eq3} we drive
\begin{eqnarray*}
\|\bm{\varphi} -\bm{\varphi}_h \|_{\bm{L}^2(\Omega)}\lesssim   h \left(\|\varphi^1\|_{H^1(\Omega)} + \|\varphi^2\|_{H^1(\Omega)} + \|\varphi^3\|_{H^1(\Omega)} \right) \lesssim \|\bm{\varphi}\|_{\bm{X}(\Omega)},
\end{eqnarray*}
and 
\begin{eqnarray*}
\|\bm{\varphi}_h\|_{\bm{X}(\Omega)} \leq  \|\bm{\varphi} -\bm{\varphi}_h \|_{\bm{X}(\Omega)} + \|\bm{\varphi}\|_{\bm{X}(\Omega)} \lesssim \|\bm{\varphi}\|_{\bm{X}(\Omega)},
\end{eqnarray*}
which conclude the proof of the lemma.
\end{proof}

We have the following regular discrete decomposition.

\begin{proposition}\label{prop:Hitpmair-Xu decomposition}
Let $\tau>0$. Every $\bm{u}_h \in \bm{V}_{h}(\mathcal{D},\Omega)$ has a decomposition 
\begin{equation}\label{eq:Hitpmair-Xu-decomposition}
\bm{u}_h=\bm{w}_h+\Pi_{h}^{\mathcal{D}} \bm{\varphi}_h + \mathcal{D}^- \bm{\phi}_h,
\end{equation}
where $\bm{w}_h \in \bm{V}_{h}(\mathcal{D},\Omega)$, $\bm{\varphi}_h \in  \bm{X}_h(\Omega)$ and $\bm{\phi}_h \in V_{h}(\mathcal{D}^-,\Omega)$ with estimate
\begin{equation}\label{eq:Hitpmair-Xu-decomposition-estimate}
(h^{-2} + \tau) \, \|\bm{w}_h\|_{\bm{L}^2(\Omega)}^2 + \|\bm{\varphi}_h\|_{\bm{X}(\Omega)}^2 + \tau \|\bm{\varphi}_h\|_{\bm{L}^2(\Omega)}^2 + \tau \|\mathcal{D}^- \bm{\phi}_h \|_{\bm{L}^2(\Omega)}^2 \lesssim \|\bm{u}_h\|_{A_{\mathcal{D}}}^2,
\end{equation}
with notation
$$
\|\bm{u}_h\|_{A_{\mathcal{D}}}^2= \|\mathcal{D} \bm{u}_h\|_{\bm{L}^2(\Omega)}^2 + \tau \|\bm{u}_h\|_{\bm{L}^2(\Omega)}^2.
$$
\end{proposition} 

\begin{proof}
Let $\bm{u}_h \in \bm{V}_{h}(\mathcal{D},\Omega)$. Using lemma \eqref{lem-semi-discrete-decomposition} we can find 
$\bm{\varphi} \in \bm{X}(\Omega)$ and $\bm{\phi}_h \in V_{h}(\mathcal{D}^-,\Omega)$ with properties
\begin{equation}\label{pr-eq:semi-discrete-decomposition}
\left\{\begin{array}{l}
\bm{u}_h = \Pi_{h}^{\mathcal{D}} \bm{\varphi} + \mathcal{D}^- \phi_h \vspace{0.2cm}\\
\|\bm{\varphi}\|_{\bm{X}(\Omega)} \lesssim \|\mathcal{D} \bm{u}_h \|_{\bm{L}^2(\Omega)} \vspace{0.25cm}\\
\|\bm{\varphi}\|_{\bm{L}^2(\Omega)} \leq \| \bm{u}_h \|_{\bm{L}^2(\Omega)} \vspace{0.25cm}\\
\|\Pi_{h}^{\mathcal{D}} \bm{\varphi}\|_{\bm{L}^2(\Omega)} + \|\mathcal{D}^- \bm{\phi}_h \|_{\bm{L}^2(\Omega)}\lesssim \|\bm{u}_h\|_{\bm{L}^2(\Omega)},
\end{array} \right.
\end{equation}
and let $\bm{\varphi}_h \in \bm{X}_h(\Omega)$ be the stable approximation of $\bm{\varphi}$, given by Lemma \ref{lem-stable-approximation-of-vect-grad}. We define
$$
\bm{w}_h=\Pi_{h}^{\mathcal{D}} (\bm{\varphi}-\bm{\varphi}_h).
$$
In this way, using the decomposition in \eqref{pr-eq:semi-discrete-decomposition}, we obtain
\begin{eqnarray*}
\bm{u}_h = \Pi_{h}^{\mathcal{D}} \bm{\varphi} + \mathcal{D}^- \bm{\phi}_h &=& \Pi_{h}^{\mathcal{D}} (\bm{\varphi}-\bm{\varphi}_h)+ \Pi_{h}^{\mathcal{D}} \bm{\varphi}_h+ \mathcal{D}^- \bm{\phi}_h \\
&=& \bm{w}_h+ \Pi_{h}^{\mathcal{D}} \bm{\varphi}_h+ \mathcal{D}^- \bm{\phi}_h,
\end{eqnarray*}
and decomposition \eqref{eq:Hitpmair-Xu-decomposition} is proved. In order to show \eqref{eq:Hitpmair-Xu-decomposition-estimate}, we need to perform careful estimates. Indeed, we have
\begin{eqnarray}\label{pr-eq:semi-discrete-decomposition-1}
h^{-1} \|\bm{w}_h\|_{\bm{L}^2(\Omega)} &=& h^{-1} \|\Pi_{h}^{\mathcal{D}} (\bm{\varphi}-\bm{\varphi}_h)\|_{\bm{L}^2(\Omega)}  \nonumber \\ 
& \lesssim & h^{-1}  \|\bm{\varphi}-\bm{\varphi}_h)\|_{\bm{L}^2(\Omega)}
 \lesssim  \|\bm{\varphi}\|_{\bm{X}(\Omega)} \lesssim \|\mathcal{D} \bm{u}_h\|_{\bm{L}^2(\Omega)},
\end{eqnarray}
where in the last estimate we have used the first inequality in \eqref{pr-eq:semi-discrete-decomposition}. Moreover, using the inverse inequality \eqref{eq:inverse-inequlity-curl} we get
\begin{equation}
\|\bm{w}_h\|_{\bm{L}^2(\Omega)} \lesssim h \|\mathcal{D} \bm{u}_h \|_{\bm{L}^2(\Omega)} \lesssim   \|\bm{u}_h\|_{\bm{L}^2(\Omega)}.
\end{equation}
Concerning the component $\bm{\varphi}_h$, we use the first and the second inequalities in \eqref{pr-eq:semi-discrete-decomposition} to obtain
\begin{equation}
\|\bm{\varphi}_h\|_{\bm{X}(\Omega)} \lesssim \|\bm{\varphi}\|_{\bm{X}(\Omega)} \lesssim \|\mathcal{D} \bm{u}_h \|_{\bm{L}^2(\Omega)},
\end{equation}
and
\begin{equation}\label{pr-eq:semi-discrete-decomposition-4}
\|\bm{\varphi}_h\|_{\bm{L}^2(\Omega)} \lesssim \|\bm{\varphi}\|_{\bm{L}^2(\Omega)}    \lesssim \| \bm{u}_h \|_{\bm{L}^2(\Omega)}.
\end{equation}
Combining \eqref{pr-eq:semi-discrete-decomposition-1}--\eqref{pr-eq:semi-discrete-decomposition-4} together with third estimate in \eqref{pr-eq:semi-discrete-decomposition} we obtain the desired estimate  \eqref{eq:Hitpmair-Xu-decomposition-estimate}. This complete the proof of the proposition. 
\end{proof}

Proposition \ref{prop:Hitpmair-Xu decomposition} forms the basis for applying the auxiliary space theory described in Section \ref{subsec:Auxiliary-Space-Preconditioning-Method}. It offers a strategy for selecting suitable auxiliary spaces and projections, as discussed in the next subsection, and provides clear instructions for choosing a smoothing operator, which must satisfy the following condition: 
$$s(\bm{w}_h, \bm{w}_h) \approx (h^{-2} + \tau) \, \|\bm{w}_h\|_{\bm{L}^2(\Omega)}^2.$$ Typically, a smoother such as Jacobi or Gauss-Seidel is used, which depends on the choice of the bases of the discrete spaces, similar to the multigrid method.

Next, we will improve estimate \eqref{eq:Hitpmair-Xu-decomposition-estimate} to apply a Jacobi smoothing method, by first constructing suitable bases for the discrete spaces. We adopt the following set of basis functions, as proposed in \cite{buffa2014isogeometric, da2014mathematical, ratnani2012arbitrary}:

$$
\mathcal{B}(\bm{grad}) = \Big\{ B_{i_1, p_1} \otimes B_{i_2, p_2} \otimes B_{i_3, p_3}: \quad 1 \leq i_l \leq n_l, \quad l=1, 2, 3 \Big\}, 
$$
\begin{align*}
\mathcal{B}(\bm{curl})  = \left\{ (D_{i_1, p_1-1} \otimes B_{i_2, p_2} \otimes B_{i_3, p_3}) \bm{e}_1 : \; 1 \leq i_1 \leq n_1-1, \, 1 \leq i_l \leq n_l, \, l=2, 3 \right\} & \\
  \cup \left\{ (B_{j_1, p_1} \otimes D_{j_2, p_2-1} \otimes B_{j_3, p_3})\bm{e}_2: \; 1 \leq j_2 \leq n_2-1, \, 1 \leq j_l \leq n_l, \, l=1, 3 \right\} &\\
 \cup \left\{ (B_{k_1, p_1} \otimes B_{k_2, p_2} \otimes D_{k_3, p_3-1})\bm{e}_3 \; 1 \leq k_3 \leq n_3-1, \, 1 \leq k_l \leq n_l, \, l=1, 2 \right\} &, \vspace{0.3cm}
\end{align*}
\begin{align*}
\mathcal{B}(div) = \left\{(B_{i_1, p_1} \otimes D_{i_2, p_21} \otimes D_{i_3, p_3-1})\bm{e}_1: \; 1 \leq i_1 \leq n_1, \, 1 \leq i_l \leq n_l-1, \, l=2, 3 \right\} & \\
\cup \left\{(D_{j_1, p_1-1} \otimes B_{j_2, p_2} \otimes D_{j_3, p_3-1})\bm{e}_2: \; 1 \leq j_2 \leq n_2, \, 1 \leq j_l \leq n_l-1, \, l=1, 3 \right\} & \\
\cup \left\{ (D_{k_1, p_1-1} \otimes D_{k_2, p_2-1} \otimes B_{k_3, p_3})\bm{e}_3: \; 1 \leq k_3 \leq n_3, \, 1 \leq k_l \leq n_l-1, \, l=1, 2 \right\} &,
\end{align*}
where $\{\bm{e}_l\}_{l=1, 2, 3}$ is the canonical basis of $\mathbb{R}^3$ and $D_{i,p-1}$ stands for Curry–Schoenberg spline basis (see, e.g., \cite{ratnani2012arbitrary})
$$
D_{i, p-1}(t) = \frac{p}{t_{i+p+1}-t_{i+1}} B_{i+1, p-1}(t), \quad t \in [0, 1]. 
$$

In the case with boundary conditions, we introduce
$$
\mathcal{B}_0(\bm{grad}) = \Big\{ B_{i_1, p_1} \otimes B_{i_2, p_2} \otimes B_{i_3, p_3}: \quad 2 \leq i_l \leq n_l-1, \quad l=1, 2, 3 \Big\}, 
$$
\begin{align*}
\mathcal{B}_0(\bm{curl})  = \left\{ (D_{i_1, p_1-1} \otimes B_{i_2, p_2} \otimes B_{i_3, p_3}) \bm{e}_1 : \; 1 \leq i_1 \leq n_1-1, \, 2 \leq i_l \leq n_l-1, \, l=2, 3 \right\} & \\
  \cup \left\{ (B_{j_1, p_1} \otimes D_{j_2, p_2-1} \otimes B_{j_3, p_3})\bm{e}_2: \; 1 \leq j_2 \leq n_2-1, \, 2 \leq j_l \leq n_l-1, \, l=1, 3 \right\} &\\
 \cup \left\{ (B_{k_1, p_1} \otimes B_{k_2, p_2} \otimes D_{k_3, p_3-1})\bm{e}_3 \; 1 \leq k_3 \leq n_3-1, \, 2 \leq k_l \leq n_l-1, \, l=1, 2 \right\} &, \vspace{0.3cm}
\end{align*}
\begin{align*}
\mathcal{B}_0(div) = \left\{(B_{i_1, p_1} \otimes D_{i_2, p_21} \otimes D_{i_3, p_3-1})\bm{e}_1: \; 2 \leq i_1 \leq n_1-1, \, 1 \leq i_l \leq n_l-1, \, l=2, 3 \right\} & \\
\cup \left\{(D_{j_1, p_1-1} \otimes B_{j_2, p_2} \otimes D_{j_3, p_3-1})\bm{e}_2: \; 2 \leq j_2 \leq n_2-1, \, 1 \leq j_l \leq n_l-1, \, l=1, 3 \right\} & \\
\cup \left\{ (D_{k_1, p_1-1} \otimes D_{k_2, p_2-1} \otimes B_{k_3, p_3})\bm{e}_3: \; 2 \leq k_3 \leq n_3-1, \, 1 \leq k_l \leq n_l-1, \, l=1, 2 \right\} &,
\end{align*}

We clearly have
\begin{equation*}
\left\{\begin{array}{l}
V_{h}(\textbf{grad},\Omega) = span\left(\mathcal{B}(\bm{grad}) \right), \vspace{0.2cm}\\
V_{h,0}(\textbf{grad},\Omega) = span\left(\mathcal{B}_0(\bm{grad})\right), \vspace{0.2cm}\\
\bm{V}_{h}(\mathcal{D},\Omega) = span\left(\mathcal{B}(\mathcal{D})  \right), 
\end{array} \right.
\end{equation*}
where in the unified notation, as usual, we dropped the index $0$ in the case with boundary conditions.

We will make use of the following $L^2$-stability of spline basis functions (see \cite{kunoth2018foundations}).
 
\begin{theorem}\label{theroem-l2-stability}
Let $(B_{i,p})_{1 \leq i \leq n}$ and $(D_{j,p})_{1 \leq j \leq n-1}$ denote, respectively, the $B$-spline and the Curry–Schoenberg spline bases associated to knot vector $T$. Then, we have 
\begin{equation*}
h \sum_{i=1}^n b_i^2 \approx \left\| \sum_{i=1}^n b_i B_{i, p}  \right\|_{L^2(0, 1)}^2, \quad  \quad \sum_{j=1}^{n-1} d_j^2 \approx \left\| \sum_{j=1}^{n-1} d_j D_{j, p-1}  \right\|_{L^2(0, 1)}^2,
\end{equation*}
for all vectors $(b_1, \cdots, b_n)$ and $(d_1, \cdots, d_{n-1})$. 

In particular,
\begin{equation*}
 \left\| B_{i, p}  \right\|_{L^2(0, 1)}^2 \approx h, \quad \left\|  D_{j, p-1}  \right\|_{L^2(0, 1)}^2 \approx 1, \quad \forall 1 \leq i \leq n, \quad \forall 1 \leq j \leq n-1.
\end{equation*}
\end{theorem} 

A direct consequence of this theorem is the following stability result:

\begin{corollary}
The bases $\mathcal{B}(D)$ are  $L^2$-stables, i.e 
\begin{equation}\label{eq:l2-stability}
\left\| \sum_{\substack{r \\ \bm{v}_{r} \in \mathcal{B}(\mathcal{D})}} c_{r} \, \bm{v}_{r} \right\|^2_{\bm{L}^2(\Omega)} \approx \sum_{\substack{r \\ \bm{v}_{r} \in \mathcal{B}(\mathcal{D})}} c_{r}^2 \left\| \bm{v}_{r} \right\|^2_{\bm{L}^2(\Omega)}, 
\end{equation}
for any vector $(c_{r}) \in \mathbb{R}^{\# \mathcal{B}(\mathcal{D})}$.
\end{corollary}
\begin{proof}
From Theorem \ref{theroem-l2-stability} we deduce that for all vectors $(b_1, \cdots, b_n)$ and $(d_1, \cdots, d_{n-1})$ we have 
$$
\sum_{i=1}^n b_i^2 \left\| B_{i,p} \right\|_{L^2(0, 1)}^2 \approx h \sum_{i=1}^n b_i^2 \approx \left\| \sum_{i=1}^n b_i B_{i, p}  \right\|_{L^2(0, 1)}^2,
$$
and 
$$
\sum_{j=1}^{n-1} d_j^2 \left\| D_{j,p-1} \right\|_{L^2(0, 1)}^2 \approx \sum_{j=1}^{n-1} d_j^2 \approx \left\| \sum_{j=1}^{n-1} d_j D_{j, p-1}  \right\|_{L^2(0, 1)}^2.
$$
The bound \eqref{eq:l2-stability} follows by tensorization and applying last estimates on each coordinate.
\end{proof}

We prove the following stable decomposition result.

\begin{theorem}[Hitpmair-Xu decomposition]\label{prop:stable-Hitpmair-Xu-decomposition}
Let $\tau>0$. For each $\bm{u}_h \in \bm{V}_{h}(\mathcal{D},\Omega)$, there exit $\bm{w}_h \in \bm{V}_{h}(\mathcal{D},\Omega)$, $\bm{\varphi}_h \in  \bm{X}_h(\Omega)$ and $\bm{\phi}_h \in V_{h}(\mathcal{D}^-,\Omega)$ such that   
\begin{equation}\label{eq:stable-Hitpmair-Xu-decomposition}
\bm{u}_h=\bm{w}_h+\Pi_{h}^{\mathcal{D}} \bm{\varphi}_h + \mathcal{D}^- \bm{\phi}_h.
\end{equation}
In addition, expanding the component $\bm{w}_h$ on the basis  $\mathcal{B}(\mathcal{D})$ 
$$
\bm{w}_h =  \sum_{\substack{r \\ \bm{v}_{r} \in \mathcal{B}(\mathcal{D})}} c_{r} \, \bm{v}_{r},
$$
we have 
\begin{equation}\label{eq:stable-Hitpmair-Xu-decomposition-estimate-2}
\sum_{\substack{r \\ \bm{v}_{r} \in \mathcal{B}(\mathcal{D})}} c_{r}^2 \left\| \bm{v}_{r} \right\|^2_{A_{\mathcal{D}}} + \|\bm{\varphi}_h\|_{\bm{X}(\Omega)}^2 + \tau \|\bm{\varphi}_h\|_{\bm{L}^2(\Omega)}^2 + \|\mathcal{D}^- \phi_h \|_{A_{\mathcal{D}}}^2 \lesssim \|\bm{u}_h\|_{A_{\mathcal{D}}}^2.
\end{equation}
\end{theorem}
\begin{proof}
Let $\bm{w}_h \in \bm{V}_{h}(\mathcal{D},\Omega)$, $\bm{\varphi}_h \in  \bm{X}_h(\Omega)$, and $\bm{\phi}_h \in V_{h}(\mathcal{D}^-,\Omega)$ the components given by Proposition \ref{prop:Hitpmair-Xu decomposition}. By remaking that 
$$
\|\mathcal{D}^- \bm{\phi}_h \|_{A_{\mathcal{D}}}^2= \|\mathcal{D} \,\left( \mathcal{D}^- \bm{\phi}_h \right)\|_{\bm{L}^2(\Omega)}^2 + \tau \, \|\mathcal{D}^- \bm{\phi}_h \|_{\bm{L}^2(\Omega)}^2 =\tau \|\mathcal{D}^- \bm{\phi}_h \|_{\bm{L}^2(\Omega)}^2,
$$
we need only to estimate the first term in \eqref{eq:stable-Hitpmair-Xu-decomposition-estimate-2}. Using \eqref{eq:l2-stability} we obtain 

\begin{eqnarray}
\sum_{\substack{r \\ \bm{v}_{r} \in \mathcal{B}(\mathcal{D})}} c_{r}^2 \left\| \bm{v}_{r} \right\|^2_{A_{\mathcal{D}}} &=& \sum_{\substack{r \\ \bm{v}_{r} \in \mathcal{B}(\mathcal{D})}} c_{r}^2 \left\| \mathcal{D} \bm{v}_{r} \right\|^2_{L^2(\Omega)} + \tau \sum_{\substack{r \\ \bm{v}_{r} \in \mathcal{B}(\mathcal{D})}} c_{r}^2 \left\| \bm{v}_{r} \right\|^2_{L^2(\Omega)} \nonumber\\
& \lesssim & (h^{-2} + \tau)  \sum_{\substack{r \\ \bm{v}_{r} \in \mathcal{B}(\mathcal{D})}} c_{r}^2 \left\| \bm{v}_{r} \right\|^2_{L^2(\Omega)} \nonumber\\
& \approx & (h^{-2} + \tau) \left\| \sum_{\substack{r \\ \bm{v}_{r} \in \mathcal{B}(\mathcal{D})}} c_{r} \, \bm{v}_{r} \right\|^2_{\bm{L}^2(\Omega)} \nonumber\\
& = &   (h^{-2} + \tau) \|\bm{w}_h\|_{\bm{L}^2(\Omega)}, \nonumber
\end{eqnarray}
and we conclude the proof using \eqref{eq:Hitpmair-Xu-decomposition-estimate}.

\end{proof}

As a final step, we provide a slightly different decomposition for the case of $\mathcal{D} = div$. In fact using the splitting of Theorem \ref{prop:stable-Hitpmair-Xu-decomposition} in the case of the $div$ problem involves solving an $\bm{H}(\bm{curl}, \Omega)$ elliptic problem, which also has a large null space. To avoid this difficulty, we adopt the approach of \cite{hiptmair2007nodal} and we use both the decomposition presented in Theorem \ref{prop:stable-Hitpmair-Xu-decomposition} and the one outlined in Proposition \ref{prop:Hitpmair-Xu decomposition}. Specifically, we prove the following result:

\begin{corollary}[Hitpmair-Xu decomposition for $\bm{H}(div, \Omega)$]\label{prop:div-stable-Hitpmair-Xu-decomposition}
Let $\tau>0$. For each $\bm{u}_h \in \bm{V}_{h}(div,\Omega)$, there exit $\bm{w}_h \in \bm{V}_{h}(div,\Omega)$ $\bm{z}_h \in \bm{V}_{h}(\bm{curl},\Omega) $ and $(\bm{\varphi}_h, \bm{\psi}_h) \in  \bm{X}_h(\Omega)^2$  such that   
\begin{equation}\label{EEE1}
\bm{u}_h=\bm{w}_h+\Pi_{h}^{div} \bm{\varphi}_h + \bm{curl}\, \bm{z}_h + \bm{curl}\, \bm{\psi}_h .
\end{equation}
In addition, expanding the components $\bm{w}_h$ and $\bm{z}_h$ on the bases  $\mathcal{B}(div)$ and  $\mathcal{B}(\bm{curl})$, respectively, 
$$
\bm{w}_h =  \sum_{\substack{r \\ \bm{v}_{r} \in \mathcal{B}(div)}} c_{r} \, \bm{v}_{r}, \quad \bm{z}_h =  \sum_{\substack{r \\ \bm{v}_{r} \in \mathcal{B}(\bm{curl})}} d_{r} \, \bm{v}_{r},
$$
we have 
\begin{eqnarray}\label{EEE2} 
&& \sum_{\substack{r \\ \bm{v}_{r} \in \mathcal{B}(div)}} c_{r}^2 \left\| \bm{v}_{r} \right\|^2_{A_{div}} + \|\bm{\varphi}_h\|_{\bm{X}(\Omega)}^2 +\tau \|\bm{\varphi}_h\|_{\bm{L}^2(\Omega)}^2\\
&& {\color{white}-----} + \tau \sum_{\substack{r \\ \bm{v}_{r} \in \mathcal{B}(\bm{curl})}} d_{r}^2 \left\| \bm{curl} \, \bm{v}_{r}  \right\|_{\bm{L}^2(\Omega)}^2  + \tau  \left\| \bm{\psi}_h \right\|_{\bm{X}(\Omega)}^2 \lesssim \|\bm{u}_h\|_{A_{div}}^2, \nonumber 
\end{eqnarray}
\end{corollary}
 
\begin{proof}
Using Theorem \ref{prop:stable-Hitpmair-Xu-decomposition}, we can find  $\bm{w}_h \in \bm{V}_{h}(div,\Omega)$, $\bm{\varphi}_h \in  \bm{X}_h(\Omega)$ and $\bm{\phi}_h \in \bm{V}_{h}(\bm{curl},\Omega)$ such that 
\begin{equation*}
\bm{u}_h=\bm{w}_h+\Pi_{h}^{div} \bm{\varphi}_h + \bm{curl} \, \bm{\phi}_h,
\end{equation*}
and 
\begin{equation}\label{eq:stable-Hitpmair-Xu-decomposition-estimate-3}
\sum_{\substack{r \\ \bm{v}_{r} \in \mathcal{B}(div)}} c_{r}^2 \left\| \bm{v}_{r} \right\|^2_{A_{div}} + \|\bm{\varphi}_h\|_{\bm{X}(\Omega)}^2 + \tau \|\bm{\varphi}_h\|_{\bm{L}^2(\Omega)}^2 + \|\bm{curl}\, \phi_h \|_{A_{div}}^2 \lesssim \|\bm{u}_h\|_{A_{div}}^2.
\end{equation}
Using however Proposition \ref{prop:Hitpmair-Xu decomposition}, there exist 
$\bm{z}_h \in \bm{V}_{h}(\bm{curl},\Omega)$, $\bm{\psi}_h \in  \bm{X}_h(\Omega)$ and $\Psi_h \in V_{h}(\bm{grad},\Omega)$ such that 
\begin{equation*}
\bm{\phi}_h=\bm{z}_h+\Pi_{h}^{\bm{curl}} \bm{\psi}_h + \bm{grad} \, \Psi_h,
\end{equation*}
with estimate
\begin{equation}\label{eq:stable-Hitpmair-Xu-decomposition-estimate-5}
h^{-2} \|\bm{z}_h\|_{\bm{L}^2(\Omega)}^2 + \|\psi\|_{\bm{L}^2(\Omega)}^2 \lesssim \|\bm{curl} \, \bm{\phi}_h\|_{\bm{L}^2(\Omega)}^2.
\end{equation}
It follows that
\begin{eqnarray*}
\bm{u}_h &=& \bm{w}_h+\Pi_{h}^{div} \bm{\varphi}_h + \bm{curl} \, \bm{z}_h +  \bm{curl} \left(\Pi_{h}^{\bm{curl}} \bm{\psi}_h \right) + \bm{curl} \, \bm{grad} \, \Psi_h \\
&=& \bm{w}_h+\Pi_{h}^{div} \bm{\varphi}_h + \bm{curl} \, \bm{z}_h +  \bm{curl} \, \bm{\psi}_h, 
\end{eqnarray*}
where in the last equality we have used the fact that $\bm{curl} \left(\Pi_{h}^{\bm{curl}} \bm{\psi}_h \right) =  \bm{curl} \, \bm{\psi}_h$ and $\bm{curl} \, \bm{grad} \, \Psi_h =0$; this prove \eqref{EEE1}.

We now prove bound \eqref{EEE2}. Using estimate \eqref{eq:stable-Hitpmair-Xu-decomposition-estimate-3}, we obtain 
\begin{equation}\label{eq:stable-Hitpmair-Xu-decomposition-estimate-4}
\sum_{\substack{r \\ \bm{v}_{r} \in \mathcal{B}(div)}} c_{r}^2 \left\| \bm{v}_{r} \right\|^2_{A_{div}} + \|\bm{\varphi}_h\|_{\bm{X}(\Omega)}^2 + \tau \|\bm{\varphi}_h\|_{\bm{L}^2(\Omega)}^2 + \tau \|\bm{curl}\, \phi_h \|_{\bm{L}^2(\Omega)}^2 \lesssim \|\bm{u}_h\|_{A_{div}}^2.
\end{equation}
But, using \eqref{eq:stable-Hitpmair-Xu-decomposition-estimate-5}, we have
\begin{eqnarray*}
\sum_{\substack{r \\ \bm{v}_{r} \in \mathcal{B}(\bm{curl})}} d_{r}^2 \left\| \bm{curl} \, \bm{v}_{r}  \right\|_{\bm{L}^2(\Omega)}^2  + \left\| \bm{\psi}_h \right\|_{\bm{X}(\Omega)}^2 & \lesssim & h^{-2} \sum_{\substack{r \\ \bm{v}_{r} \in \mathcal{B}(\bm{curl})}} d_{r}^2 \left\|  \bm{v}_{r}  \right\|_{\bm{L}^2(\Omega)}^2  + \left\| \bm{\psi}_h \right\|_{\bm{X}(\Omega)}^2 \\
& \approx & h^{-2} \left\|\sum_{\substack{r \\ \bm{v}_{r} \in \mathcal{B}(\bm{curl})}} d_{r} \bm{v}_{r} \right\|_{\bm{L}^2(\Omega)}^2 + \left\| \bm{\psi}_h \right\|_{\bm{X}(\Omega)}^2 \\
& = & h^{-2} \|\bm{z}_h\|_{\bm{L}^2(\Omega)}^2 + \left\| \bm{\psi}_h \right\|_{\bm{X}(\Omega)}^2 \\
& \lesssim & \|\bm{curl} \, \bm{\phi}_h\|_{\bm{L}^2(\Omega)}^2, 
\end{eqnarray*}
and we conclude the proof combining this last estimate together with \eqref{eq:stable-Hitpmair-Xu-decomposition-estimate-4}.
\end{proof} 

\subsection{Auxiliary Space Preconditioners}\label{sec:construction-of-ASP}

We are now ready to apply the abstract ASP theory of Section \ref{subsec:Auxiliary-Space-Preconditioning-Method}. 

\subsubsection{ASP-preconditioner in the case $\mathcal{D} = \bm{curl}$} 
Following the same notations of Section \ref{subsec:Auxiliary-Space-Preconditioning-Method}, let us consider $V=\bm{V}_{h}(\bm{curl},\Omega)$ equipped with the bilinear form $a$ related to equation \eqref{introduction-variational-formulation}, namely $a(\bm{w}_h,\widetilde{\bm{w}_h})= \left(\bm{curl} \, \bm{w}_h,\bm{curl}\, \widetilde{\bm{w}_h} \right)_{\bm{L}^2(\Omega)} + \tau (\bm{w}_h,\widetilde{\bm{w}_h})_{\bm{L}^2(\Omega)}$ for $\bm{w}_h, \widetilde{\bm{w}_h} \in V$, and auxiliary spaces $W_1= \bm{X}_h(\Omega)$ and $W_2=V_{h}(\bm{grad},\Omega)$ equipped with the following inner products
\begin{equation*}
a_1(\bm{\varphi}_h,\widetilde{\bm{\varphi}_h})= (\bm{\varphi}_h,\widetilde{\bm{\varphi}_h})_{\bm{X}(\Omega)} + \tau (\bm{\varphi}_h,\widetilde{\bm{\varphi}_h})_{\bm{L}^2(\Omega)}, \quad \bm{\varphi}_h,\widetilde{\bm{\varphi}_h} \in W_1,
\end{equation*}
and
\begin{equation*}
a_2(\phi_h,\widetilde{\phi_h})= \tau \, (\bm{grad} \, \phi_h, \bm{grad} \, \widetilde{\phi_h})_{\bm{L}^2(\Omega)},  \quad \phi_h,\widetilde{\phi_h} \in W_2,
\end{equation*}
respectively. The corresponding transfer operators are $\pi_1 = \Pi_{h}^{\bm{curl}}\big|_{W_1}$ and $\pi_2=\bm{grad} \big|_{W_2}$. \\

Before we verify the validity of assumptions of Theorem \ref{th:ASP-lemma}, we transition to a matrix notation. So we write $\bm{H}$ for the matrix related to the restriction of $\bm{X}(\Omega)$ inner product to $\bm{X}_h(\Omega)$, that is the matrix representation of the bilinear form 
\begin{equation*}
(\bm{\varphi}_h,\widetilde{\bm{\varphi}_h}) \in \bm{X}_h(\Omega)\times \bm{X}_h(\Omega) \mapsto (\bm{\varphi}_h,\widetilde{\bm{\varphi}_h})_{\bm{X}(\Omega)}.
\end{equation*}
Similarly, we write $\bm{M}$ for the matrix related to the restriction of the $\bm{L}^2(\Omega)$ inner product to $\bm{X}_{h}(\Omega)$. Let $\bm{L}$ be the matrix related to the mapping  
\begin{equation*}
(\bm{\phi}_h,\widetilde{\bm{\phi}_h}) \in V_{h}(\bm{grad},\Omega) \times V_{h}(\bm{grad},\Omega)  \longmapsto \left(\bm{grad}\, \bm{\phi}_h, \bm{grad}\, \widetilde{\bm{\phi}_h} \right)_{\bm{L}^2(\Omega)}.
\end{equation*}
We also write $\bm{P}_{\bm{curl}}$ and $\bm{G}$ for matrices related to the transform operators $\pi_1$ and $\pi_2$ respectively while $\bm{S}_{\bm{curl}}$ further stands for the matrix related to the smoother. In the case of  Jacobi smoothing $\bm{S}_{\bm{curl}}$ is the matrix representation of the smoothing operator
\begin{equation}\label{eq:curl-jacobi-smoothing}
\bm{w}_h \in V \mapsto s( \bm{w}_h, \bm{w}_h) = \sum_{\substack{r \\ \bm{v}_{r} \in \mathcal{B}(\bm{curl})}} c_{r}^2 \, a(\bm{v}_{r}, \bm{v}_{r}), \quad \bm{w}_h =  \sum_{\substack{r \\ \bm{v}_{r} \in \mathcal{B}(\bm{curl})}} c_{r} \, \bm{v}_{r},
\end{equation}
and it coincides with the diagonal $\bm{D}_{\bm{A}_{\bm{curl}}}$ of $\bm{A}_{\bm{curl}}$ ($\bm{A}_{\bm{curl}}$ stands for the matrix representation of bilinear form $a$).

With these notations, a simple computation shows that ASP preconditioner  for problem \eqref{eq:curl-curl} reads 
\begin{equation}
\label{eq:asp-curl}
\bm{B}_{\bm{curl}} = \bm{S}_{\bm{curl}}^{-1} + \bm{P}_{\bm{curl}} \left(\bm{H} + \tau \bm{M}\right)^{-1} \bm{P}_{\bm{curl}}^T + \tau^{-1} \bm{G} \bm{L}^{-1} \bm{G}^T.
\end{equation}

\begin{remark}
In two dimensions, we have two distinct curl operators: the scalar curl operator, defined as  $curl \, \bm{u} = \frac{\partial u_1}{\partial x_2}- \frac{\partial u_2}{\partial x_1}$, and the vector $\bm{curl}$ operator, defined as $\bm{curl} \, u = \left(\frac{\partial u}{\partial x_2},-\frac{\partial u}{\partial x_1} \right)$. The present analysis refers to the scalar $curl$ operator. However, the vector $\bm{curl}$ operator is used in the preconditioning of the two-dimensional $\bm{H}(div, \Omega)$ problem addressed in the following subsection.
\end{remark}

The following fundamental result demonstrates the mesh-independence of the preconditioner $\bm{B}_{\bm{curl}}$, at least when a Jacobi smoothing scheme is used. This result is a direct consequence of Theorem \ref{prop:stable-Hitpmair-Xu-decomposition} and Theorem \ref{th:ASP-lemma}.

\begin{theorem}\label{mesh-independence-curl}
Let $\tau >0$ and suppose that the smoothing operator $s$ is given by \eqref{eq:curl-jacobi-smoothing}. Then, the spectral condition number $\kappa \left( \bm{B}_{\bm{curl}} \bm{A}_{\bm{curl}}\right)$ is bounded, with respect to $h$ and $\tau$.
\end{theorem}

\begin{proof}
We will verify the assumptions of Theorem \ref{th:ASP-lemma}. To do so, we use Theorem \ref{error_estimates} and the following estimate (found, for instance, in \cite{buffa2011isogeometric}):
$$
\|\Pi_h^{\bm{curl}}\bm{w}_h\|_{\bm{H}(\bm{curl}, \Omega)} \lesssim \|\bm{w}_h\|_{\bm{H}(\bm{curl}, \Omega)}, \quad w_h \in W_1.
$$
From this estimate, we can see that the first inequality in (i) holds with a constant $\beta_1$ that is independent of $h$. The second inequality in (i) holds with a constant $\beta_2=1$, which is a consequence of the relation $\bm{curl} \circ \bm{grad} = 0$. 

To prove the inequality in (ii), we express any $\bm{w}_h \in \bm{V}_h(\bm{curl}, \Omega)$ as 
$$
\bm{w}_h =  \sum_{\substack{\bm{r} \\ \bm{v}_{\bm{r}} \in \mathcal{B}(\bm{curl})}} c_{\bm{r}} \, \bm{v}_{\bm{r}}.
$$
We have
\begin{eqnarray*}
a(\bm{w}_h, \bm{w}_h) & = & \| \bm{curl} \, \bm{w}_h \|_{\bm{L}^2(\Omega)}^2 + \tau \| \bm{w}_h \|_{\bm{L}^2(\Omega)}^2 \\
& \lesssim & \left\| \sum_{\substack{\bm{r} \\ \bm{v}_{\bm{r}} \in \mathcal{B}(\bm{curl})}} c_{\bm{r}} \, \bm{curl} \, \bm{v}_{\bm{r}} \right\|_{\bm{L}^2(\Omega)}^2 + \tau \left\| \sum_{\substack{\bm{r} \\ \bm{v}_{\bm{r}} \in \mathcal{B}(\bm{curl})}} c_{\bm{r}} \, \bm{v}_{\bm{r}} \right\|_{\bm{L}^2(\Omega)}^2\\
& = & \sum_{\mathcal{Q} \in \mathcal{Q}_h} \left\| \sum_{k = 1}^K  c_{k} \, \bm{curl} \, \bm{v}_{k} \right\|_{\bm{L}^2(\Omega)}^2 + \tau \sum_{\mathcal{Q} \in \mathcal{Q}_h} \left\| \sum_{k = 1}^K  c_{k} \,  \bm{v}_{k} \right\|_{\bm{L}^2(\Omega)}^2 \\
& \leq & K \sum_{\mathcal{Q} \in \mathcal{Q}_h}  \sum_{k = 1}^K c_{k}^2 \| \bm{curl} \, \bm{v}_{k} \|_{\bm{L}^2(\Omega)}^2 + \tau K \sum_{\mathcal{Q} \in \mathcal{Q}_h}  \sum_{k = 1}^K c_{k}^2 \| \bm{v}_{k} \|_{\bm{L}^2(\Omega)}^2\\
& = & K \, s(\bm{w}_h, \bm{w}_h), 
\end{eqnarray*}
where $\mathcal{Q}_h$ denotes the parametric B\'ezier mesh, $Q$ a generic Bézier element and the constant $K$ is the number of basis functions whose support interacts with $Q$. The constant $K$ depends only on the degrees of the spline bases.

Finally, the last assumption (iii) follows from Theorem \ref{prop:stable-Hitpmair-Xu-decomposition}.

\end{proof}

\subsubsection{ASP-preconditioner in the case $\mathcal{D} = div$} 
To simplify the discussion, we will consider the two and three dimensional cases separately. In the two-dimensional setting, the de Rham diagram reduces to the following:
\begin{align*}
    \begin{array}{ccccc}
   H^1(\Omega) & \xrightarrow{\quad \bm{curl} \quad } & \bm{H}(div, \Omega) & \xrightarrow{\quad div \quad } & L^2(\Omega)\\
  \Pi_{h}^{\bm{grad}} \Bigg\downarrow &  & \Pi_{h}^{div} \Bigg\downarrow &  & \Pi_{h}^{L^2} \Bigg\downarrow \\
  V_{h}(\bm{grad},\Omega) & \xrightarrow{\quad \bm{curl} \quad } & \bm{V}_h(div, \Omega) & \xrightarrow{\quad div \quad } & V_h(L^2, \Omega)
  \end{array}
\end{align*}
Here, $\bm{curl}$ refers to the vector $\bm{curl}$ operator. Theorem \ref{prop:stable-Hitpmair-Xu-decomposition}, with $\mathcal{D} = div$, provides us with the starting point. We therefore choose $V = \bm{V}_h(div, \Omega)$ equipped with the bilinear form
\begin{equation}
\label{eq:bilinear-form-div}
(\bm{w}_h, \widetilde{\bm{w}_h}) \in V \times V \mapsto a(\bm{w}_h,\widetilde{\bm{w}_h})= \left(div \, \bm{w}_h, div\, \widetilde{\bm{w}_h} \right)_{L^2(\Omega)} + \tau (\bm{w}_h,\widetilde{\bm{w}_h})_{\bm{L}^2(\Omega)}.
\end{equation}
We choose the auxiliary spaces $W_1= \bm{X}_h(\Omega)$ and $W_2=V_{h}(\bm{grad},\Omega)$. $W_1$ and $W_2$ are equipped with the following inner products:
\begin{equation*}
a_1(\bm{\varphi}_h,\widetilde{\bm{\varphi}_h})= (\bm{\varphi}_h,\widetilde{\bm{\varphi}_h})_{\bm{X}(\Omega)} + \tau (\bm{\varphi}_h,\widetilde{\bm{\varphi}_h})_{\bm{L}^2(\Omega)}, \quad \bm{\varphi}_h,\widetilde{\bm{\varphi}_h} \in W_1,
\end{equation*}
and
\begin{equation*}
a_2(\phi_h,\widetilde{\phi_h})= \tau \, (\bm{curl} \, \phi_h, \bm{curl} \, \widetilde{\phi_h})_{\bm{L}^2(\Omega)},  \quad \phi_h,\widetilde{\phi_h} \in W_2,
\end{equation*}
respectively. We define the transfer operators as $\pi_1 = \Pi_{h}^{div}\big|_{W_1}$ and $\pi_2=\bm{curl} \big|_{W_2}$. Let $\bm{P}_{div}$ and $\bm{R}$ be the matrices related to the transfer operators $\pi_1$ and $\pi_2$, respectively, and let $\bm{S}_{div}$ be the matrix representation of the smoother. The ASP preconditioner for problem \eqref{eq:div-div} in the two-dimensional setting can be expressed as
\begin{equation}
\label{eq:asp-2d-div}
\bm{B}_{div} = \bm{S}_{div}^{-1} + \bm{P}_{div} \left(\bm{H} + \tau \bm{M}\right)^{-1} \bm{P}_{div}^T + \tau^{-1} \bm{R} \bm{L}^{-1} \bm{R}^T,
\end{equation} 
where $\bm{L}$ and $\bm{M}$ are the matrices defined in the case $\mathcal{D} = \bm{curl}$.

In the three-dimensional case, Corollary \ref{prop:div-stable-Hitpmair-Xu-decomposition} is used as the basis for constructing the preconditioner. Similar to the two-dimensional case, we have $V = \bm{V}_h(div, \Omega)$ equipped with the bilinear form \eqref{eq:bilinear-form-div}. We choose the transform operators as follows:
\begin{enumerate}
\item[1.] $W_1 = \bm{X}_h(\Omega)$ with inner product
$$
a_1(\bm{\varphi}_h, \widetilde{\bm{\varphi}_h}) = (\bm{\varphi}_h,\widetilde{\bm{\varphi}_h})_{\bm{X}(\Omega)} + \tau (\bm{\varphi}_h,\widetilde{\bm{\varphi}_h})_{\bm{L}^2(\Omega)}, \quad \bm{\varphi}_h,  \widetilde{\bm{\varphi}_h} \in W_1.
$$
\bigskip

\item[2.] $W_2 = \bm{V}_h(\bm{curl}, \Omega)$  equipped with inner product
\begin{equation*}\label{eq:jacobi-curl}
a_2(\bm{z}_h, \widetilde{\bm{z}_h}) = \tau \sum_{\substack{r \\ \bm{v}_{r} \in \mathcal{B}(\bm{curl})}} d_{r} \, \tilde{d}_{r} \left\|\bm{curl}\, \bm{v}_r \right\|_{\bm{L}^2(\Omega)}^2, \quad \bm{z}_h,  \widetilde{\bm{z}_h} \in W_2,
\end{equation*}
with 
$$
\bm{z}_h =  \sum_{\substack{\bm{r} \\ \bm{v}_{\bm{r}} \in \mathcal{B}(\bm{curl})}} d_{\bm{r}} \, \bm{v}_{\bm{r}}, \quad \widetilde{\bm{z}_h} =  \sum_{\substack{\bm{r} \\ \bm{v}_{\bm{r}} \in \mathcal{B}(\bm{curl})}} \tilde{d}_{\bm{r}} \, \bm{v}_{\bm{r}}.
$$
\bigskip

\item[3.] $W_3 = \bm{X}_h(\Omega)$ with inner product
$$
a_3(\bm{\psi}_h, \widetilde{\bm{\psi}_h}) = \tau (\bm{\psi}_h,\widetilde{\bm{\psi}_h})_{\bm{X}(\Omega)}, \quad \bm{\psi}_h,  \widetilde{\bm{\psi}_h} \in W_3.
$$
\end{enumerate}

The corresponding transfer operators are $\pi_1 = \Pi_{h}^{div}\big|_{W_1}$, $\pi_2=\bm{curl} \big|_{W_2}$ and $\pi_3 = \bm{curl} \big|_{W_3}$.

In matrix notation, the bilinear form $\tau^{-1} a_2$ is represented by the matrix $\bm{D}_{\bm{curl}}$, it coincides  with the diagonal of the matrix $\bm{Q}_{\bm{curl}} = (\bm{Q}_{\bm{r},\bm{q}})$ defined by
\begin{equation}
\label{eq:div-second-smoother}
\bm{Q}_{\bm{r},\bm{q}}  = \int_{\Omega} \bm{curl}\, \bm{v}_{\bm{q}} \cdot  \bm{curl}\, \bm{v}_{\bm{r}}, \quad   \bm{v}_{\bm{r}}, \bm{v}_{\bm{q}} \in \mathcal{B}(\bm{curl}).
\end{equation}

The matrix related to projection $\pi_2$ is denoted by $\bm{C}$. A straightforward calculation yields that
\begin{equation}
\label{eq:asp-div}
\begin{split}
\bm{B}_{div} = \bm{S}_{div}^{-1} & + \bm{P}_{div} \left(\bm{H} + \tau \bm{M}\right)^{-1} \bm{P}_{div}^T + \tau^{-1} \bm{C} \bm{D}_{\bm{curl}}^{-1} \bm{C}^T\\
 & + \tau^{-1} \bm{C} \bm{P}_{\bm{curl}} \bm{H}^{-1} \bm{P}_{\bm{curl}}^T \bm{C}^T,
\end{split}
\end{equation}
where $\bm{S}_{div}$ represents the smoother in matrix form.

\begin{theorem}
Suppose that the smoothing operator $s$ is given by the Jacobi relaxation scheme and let $\tau >0$. Then, the spectral condition number $\kappa \left( \bm{B}_{div} \bm{A}_{div}\right)$ is bounded with respect to $h$ and $\tau$.
\end{theorem}

\begin{proof}
The proof follows a similar approach to that of Theorem \ref{mesh-independence-curl} and is therefore omitted.
\end{proof}

\section{Numerical Results}\label{sec:numerical-results}

\begin{table}[H]

\begin{adjustbox}{width=1.\textwidth}
\begin{subtable}{1.\textwidth} 
    \centering 
    \hspace*{-2. cm}
\begin{tabular}{c p{0.cm} p{1.5cm} p{1.5cm} p{1.5cm} p{1.5cm}  c p{1.5cm} p{1.5cm} p{1.5cm} p{1.5cm}}
\hline
 &  & \multicolumn{4}{l}{$p=1$} &  & \multicolumn{4}{l}{$p=2$} \\ \cline{3-6} \cline{8-11} 
\diaghead{taun-----}{$\tau$}{$n$} &  & $8$   & $16$   & $32$   & $64$    & & $8$   & $16$   & $32$   & $64$    \\ \hline
$10^{-4}$    & & $1.37e+07$ & $5.97e+07$ & $2.44e+08$ & $9.81e+08$ & & $4.01e+07$ & $1.65e+08$ & $6.72e+08$ & $2.70e+09$ \\
$10^{-3}$    & & $1.37e+06$ & $5.97e+06$ & $2.44e+07$ & $9.81e+07$ & & $4.01e+06$ & $1.65e+07$ & $6.72e+07$ & $2.70e+08$ \\
$10^{-2}$    & & $1.37e+05$ & $5.97e+05$ & $2.44e+06$ & $9.81e+06$ & & $4.01e+05$ & $1.65e+06$ & $6.72e+06$ & $2.70e+07$ \\
$10^{-1}$    & & $1.37e+04$ & $5.97e+04$ & $2.44e+05$ & $9.81e+05$ & & $4.01e+04$ & $1.66e+05$ & $6.72e+05$ & $2.70e+06$ \\
$1$          & & $1.37e+03$ & $5.97e+03$ & $2.44e+04$ & $9.81e+04$ & & $4.02e+03$ & $1.66e+04$ & $6.72e+04$ & $2.70e+05$ \\
$10^{1}$     & & $1.38e+02$ & $5.98e+02$ & $2.44e+03$ & $9.81e+03$ & & $4.13e+02$ & $1.67e+03$ & $6.73e+03$ & $2.70e+04$ \\
$10^{2}$     & & $1.47e+01$ & $6.07e+01$ & $2.45e+02$ & $9.82e+02$ & & $5.46e+01$ & $1.78e+02$ & $6.84e+02$ & $2.72e+03$ \\
$10^{3}$     & & $2.77e+00$ & $6.97e+00$ & $2.54e+01$ & $9.91e+01$ & & $3.10e+01$ & $3.53e+01$ & $8.12e+01$ & $2.83e+02$ \\
$10^{4}$     & & $2.72e+00$ & $2.93e+00$ & $3.44e+00$ & $1.08e+01$ & & $3.08e+01$ & $3.23e+01$ & $3.29e+01$ & $4.36e+01$ 
\end{tabular}
\end{subtable}
\end{adjustbox}

\vspace{.5cm}

\begin{adjustbox}{width=1.\textwidth}
\begin{subtable}{\textwidth} 
    \centering 
    \hspace*{-2. cm}
\begin{tabular}{c c p{1.5cm} p{1.5cm} p{1.5cm} p{1.5cm}  c p{1.5cm} p{1.5cm} p{1.5cm} p{1.5cm}}
\hline
 &  & \multicolumn{4}{l}{$p=3$} &  & \multicolumn{4}{l}{$p=4$} \\ \cline{3-6} \cline{8-11} 
\diaghead{taun-----}{$\tau$}{$n$} &  & $8$   & $16$   & $32$   & $64$    & & $8$   & $16$   & $32$   & $64$    \\ \hline
$10^{-4}$    & & $4.40e+08$ & $1.71e+09$ & $6.74e+09$ & $2.68e+10$ & & $4.98e+09$ & $1.77e+10$ & $6.77e+10$ & $2.68e+11$ \\
$10^{-3}$    & & $4.40e+07$ & $1.71e+08$ & $6.74e+08$ & $2.68e+09$ & & $4.98e+08$ & $1.77e+09$ & $6.77e+09$ & $2.68e+10$ \\
$10^{-2}$    & & $4.40e+06$ & $1.71e+07$ & $6.74e+07$ & $2.68e+08$ & & $4.98e+07$ & $1.77e+08$ & $6.77e+08$ & $2.68e+09$ \\
$10^{-1}$    & & $4.40e+05$ & $1.71e+06$ & $6.74e+06$ & $2.68e+07$ & & $4.98e+06$ & $1.77e+07$ & $6.77e+07$ & $2.68e+08$ \\
$1$          & & $4.41e+04$ & $1.71e+05$ & $6.74e+05$ & $2.68e+06$ & & $4.99e+05$ & $1.77e+06$ & $6.77e+06$ & $2.68e+07$ \\
$10^{1}$     & & $4.48e+03$ & $1.71e+04$ & $6.75e+04$ & $2.68e+05$ & & $5.04e+04$ & $1.77e+05$ & $6.77e+05$ & $2.68e+06$ \\
$10^{2}$     & & $5.31e+02$ & $1.79e+03$ & $6.82e+03$ & $2.69e+04$ & & $5.65e+03$ & $1.82e+04$ & $6.82e+04$ & $2.69e+05$ \\
$10^{3}$     & & $2.95e+02$ & $2.95e+02$ & $7.56e+02$ & $2.76e+03$ & & $2.73e+03$ & $2.52e+03$ & $7.31e+03$ & $2.73e+04$ \\
$10^{4}$     & & $3.04e+02$ & $2.97e+02$ & $2.90e+02$ & $3.62e+02$ & & $2.90e+03$ & $2.61e+03$ & $2.44e+03$ & $3.26e+03$  
\end{tabular}
\end{subtable}
\end{adjustbox}

\vspace{.5cm}

\begin{adjustbox}{width=1.\textwidth}
\begin{subtable}{1.\textwidth} 
    \centering 
    \hspace*{-2. cm}
\begin{tabular}{c c p{1.5cm} p{1.5cm} p{1.5cm} p{1.5cm}  c p{1.5cm} p{1.5cm} p{1.5cm} p{1.5cm}}
\hline
 &  & \multicolumn{4}{l}{$p=5$} &  & \multicolumn{4}{l}{$p=6$} \\ \cline{3-6} \cline{8-11} 
\diaghead{taun-----}{$\tau$}{$n$} &  & $8$   & $16$   & $32$   & $64$    & & $8$   & $16$   & $32$   & $64$    \\ \hline
$10^{-4}$    & & $5.75e+10$ & $1.79e+11$ & $6.63e+11$ & $2.62e+12$ & & $6.98e+11$ & $1.80e+12$ & $6.33e+12$ & $2.49e+13$\\
$10^{-3}$    & & $5.75e+09$ & $1.79e+10$ & $6.63e+10$ & $2.62e+11$ & & $6.98e+10$ & $1.80e+11$ & $6.33e+11$ & $2.49e+12$ \\
$10^{-2}$    & & $5.75e+08$ & $1.79e+09$ & $6.63e+09$ & $2.62e+10$ & & $6.98e+09$ & $1.80e+10$ & $6.33e+10$ & $2.49e+11$ \\
$10^{-1}$    & & $5.75e+07$ & $1.79e+08$ & $6.63e+08$ & $2.62e+09$ & & $6.98e+08$ & $1.80e+09$ & $6.33e+09$ & $2.49e+10$ \\
$1$          & & $5.76e+06$ & $1.80e+07$ & $6.63e+07$ & $2.62e+08$ & & $6.99e+07$ & $1.80e+08$ & $6.33e+08$ & $2.49e+09$ \\
$10^{1}$     & & $5.80e+05$ & $1.80e+06$ & $6.63e+06$ & $2.62e+07$ & & $7.03e+06$ & $1.80e+07$ & $6.34e+07$ & $2.49e+08$ \\
$10^{2}$     & & $6.29e+04$ & $1.83e+05$ & $6.66e+05$ & $2.62e+06$ & & $7.46e+05$ & $1.83e+06$ & $6.36e+06$ & $2.49e+07$ \\
$10^{3}$     & & $2.59e+04$ & $2.25e+04$ & $7.00e+04$ & $2.65e+05$ & & $2.59e+05$ & $2.12e+05$ & $6.60e+05$ & $2.52e+06$ \\
$10^{4}$     & & $2.78e+04$ & $2.24e+04$ & $2.03e+04$ & $3.01e+04$ & & $2.80e+05$ & $1.93e+05$ & $1.67e+05$ & $2.76e+05$ 
\end{tabular}
\end{subtable}
\end{adjustbox}
\caption{The $2$-$d$ unpreconditioned $\bm{H}_0(\bm{curl}, \Omega)$: condition number $\kappa_2 \left( \bm{A}_{\bm{curl}} \right)$.}  
  \label{tab:curl-cn-none}  
\end{table}

The computational domain is defined as the unit square $\Omega = (0,1)^2$ subdivided into $n \times n$ sub-domains ($n \in \mathbb{N}^*$). First, we compute the condition number and then track the total number of iterations required for convergence of the Conjugate Gradient (\CG) method for different values of $n$, $\tau$, and $p$. We use $\bm{u}^\mathcal{D}$ ($\mathcal{D}= \bm{curl}$ or $\mathcal{D} = div$) to denote the solution of the linear system, \textit{i.e.}, $\bm{A}_\mathcal{D} \bm{u}^\mathcal{D} = \bm{b}$, where $\bm{b}$ represents the IgA discretization of the right-hand side function $\bm{f}$. In all experiments, we use the stopping criterion of
\begin{equation}\label{eq:stopping-criterion}
\frac{\|\bm{A}_\mathcal{D}\bm{u}^\mathcal{D}-\bm{b} \|_2}{\|\bm{b}\|_2} \leq 10^{-6}, 
\end{equation}
and the initial guess is always chosen to be the zero vector. 


\begin{table}[H]

\begin{adjustbox}{width=1.\textwidth}
\begin{subtable}{1.\textwidth} 
    \centering 
    \hspace*{-2. cm}
\begin{tabular}{c p{0.cm} p{1.cm} p{1.7cm} p{1.7cm} p{1.7cm}  c p{1.7cm} p{1.7cm} p{1.7cm} p{1.7cm}}
\hline
 &  & \multicolumn{4}{l}{$p=1$} &  & \multicolumn{4}{l}{$p=2$} \\ \cline{3-6} \cline{8-11} 
\diaghead{taun-----}{$\tau$}{$n$} &  & $8$   & $16$   & $32$   & $64$    & & $8$   & $16$   & $32$   & $64$    \\ \hline
$10^{-4}$    & & $63$  & $110$  & $186$  & $253$  & & $115$ & $287$  & $588$  & $1120$ \\
$10^{-3}$    & & $51$  & $99$   & $158$  & $185$  & & $86$  & $256$  & $510$  & $927$ \\
$10^{-2}$    & & $46$  & $82$   & $117$  & $160$  & & $72$  & $197$  & $426$  & $743$ \\
$10^{-1}$    & & $34$  & $44$   & $88$   & $149$  & & $55$  & $167$  & $317$  & $586$ \\
$1$          & & $18$  & $41$   & $80$   & $134$  & & $51$  & $123$  & $230$  & $435$ \\
$10^{1}$     & & $15$  & $36$   & $67$   & $115$  & & $34$  & $67$   & $120$  & $240$ \\
$10^{2}$     & & $12$  & $20$   & $34$   & $58$   & & $22$  & $37$   & $70$   & $125$  \\
$10^{3}$     & & $6$   & $5$    & $12$   & $28$   & & $19$  & $19$   & $23$   & $37$ \\
$10^{4}$     & & $5$   & $7$    & $4$    & $6$    & & $19$  & $21$   & $18$   & $16$ 
\end{tabular}
\end{subtable}
\end{adjustbox}

\vspace{.5cm}

\begin{adjustbox}{width=1.\textwidth}
\begin{subtable}{\textwidth} 
    \centering 
    \hspace*{-2. cm}
\begin{tabular}{c p{0.cm} p{1.cm} p{1.7cm} p{1.7cm} p{1.7cm}  c p{1.7cm} p{1.7cm} p{1.7cm} p{1.7cm}}
\hline
 &  & \multicolumn{4}{l}{$p=3$} &  & \multicolumn{4}{l}{$p=4$} \\ \cline{3-6} \cline{8-11} 
\diaghead{taun-----}{$\tau$}{$n$} &  & $8$   & $16$   & $32$   & $64$    & & $8$   & $16$   & $32$   & $64$    \\ \hline
$10^{-4}$    & & $323$ & $666$  & $1050$ & $1896$ & & $616$ & $1771$ & $2296$ & $-$ \\
$10^{-3}$    & & $272$ & $604$  & $861$  & $1545$ & & $437$ & $1394$ & $1878$ & $2931$ \\
$10^{-2}$    & & $216$ & $427$  & $697$  & $1133$ & & $329$ & $1032$ & $1518$ & $2404$ \\
$10^{-1}$    & & $175$ & $346$  & $558$  & $963$  & & $229$ & $794$  & $1143$ & $1849$ \\
$1$          & & $128$ & $275$  & $390$  & $685$  & & $150$ & $506$  & $785$  & $1217$ \\
$10^{1}$     & & $72$  & $147$  & $236$  & $428$  & & $82$  & $241$  & $400$  & $701$ \\
$10^{2}$     & & $35$  & $53$   & $97$   & $166$  & & $43$  & $73$   & $135$  & $237$ \\
$10^{3}$     & & $30$  & $33$   & $34$   & $53$   & & $32$  & $38$   & $47$   & $77$ \\
$10^{4}$     & & $32$  & $33$   & $26$   & $21$   & & $46$  & $44$   & $38$   & $33$ 
\end{tabular}
\end{subtable}
\end{adjustbox}

\vspace{.5cm}

\begin{adjustbox}{width=1.\textwidth}
\begin{subtable}{1.\textwidth} 
    \centering 
    \hspace*{-2. cm}
\begin{tabular}{c p{0.cm} p{1.cm} p{1.7cm} p{1.7cm} p{1.7cm}  c p{1.7cm} p{1.7cm} p{1.7cm} p{1.7cm}}
\hline
 &  & \multicolumn{4}{l}{$p=5$} &  & \multicolumn{4}{l}{$p=6$} \\ \cline{3-6} \cline{8-11} 
\diaghead{taun-----}{$\tau$}{$n$} &  & $8$   & $16$   & $32$   & $64$    & & $8$   & $16$   & $32$   & $64$    \\ \hline
$10^{-4}$    & & $1567$ & $-$ & $-$ & $-$ & & $2151$ & $-$ & $-$ & $-$ \\
$10^{-3}$    & & $1258$ & $2816$ & $-$ & $-$ & & $1565$ & $-$ & $-$ & $-$ \\
$10^{-2}$    & & $927$  & $2214$ & $-$ & $-$ & & $1029$ & $-$ & $-$ & $-$ \\
$10^{-1}$    & & $606$  & $1417$ & $-$ & $-$ & & $621$  & $2368$ & $-$ & $-$ \\
$1$          & & $307$  & $731$  & $1587$ & $2795$ &  & $292$  & $988$  & $2377$ & $-$ \\
$10^{1}$     & & $125$  & $286$  & $613$  & $1141$ & & $129$  & $355$  & $839$  & $1631$ \\
$10^{2}$     & & $55$   & $98$   & $173$  & $362$  & & $59$   & $138$  & $229$  & $495$ \\
$10^{3}$     & & $45$   & $52$   & $66$   & $115$  &  & $54$   & $73$   & $84$   & $165$ \\
$10^{4}$     & & $61$   & $56$   & $48$   & $50$ & & $64$   & $69$   & $64$   & $75$ 
\end{tabular}
\end{subtable}
\end{adjustbox}
\caption{The $2$-$d$ unpreconditioned $\bm{H}_0(\bm{curl}, \Omega)$: \CG  iterations. Exact solution is given by \eqref{eq:curl-exact-sol-1}. '$-$' means that \CG reaches the maximum number of iterations (set to $3000$) without convergence.}  
  \label{tab:curl-cgn-none}  
\end{table}


In this section, some sample simulations are developed to test the strategy proposed in this paper in view of further applications. The simulations are performed in two and three spatial dimensions, which are discussed in two separate subsections. The first subsection focuses on the two-dimensional case, while the second subsection is dedicated to the three-dimensional case.

\subsection{Two dimensional tests}

\begin{table}[H] 
\begin{adjustbox}{width=\textwidth}
  \begin{subtable}{0.6\textwidth} 
    \subcaption{$\bm{H}_0(\bm{curl}, \Omega)$ problem.}
    \centering
    \begin{tabular}{|c|c|c|c|}
      \hline
      $\tau$ & \CG Iter & Res. Error & $l^2$ Error\\ 
      \hline
      \hline
      $10^{-7}$ 	& $42$     & $1.22e-06$ & $2.29e-01$    \\
 	  $10^{-6}$ 	& $42$     & $1.22e-06$ & $2.29e-01$    \\
 	  $10^{-5}$ 	& $42$     & $1.36e-06$ & $2.29e-01$    \\
      $10^{-4}$ 	& $68$     & $9.86e-07$ & $1.54e-02$    \\
      $10^{-3}$ 	& $87$     & $1.83e-06$ & $7.78e-03$    \\
      $10^{-2}$ 	& $206$    & $1.22e-06$ & $7.16e-04$    \\
      $10^{-1}$ 	& $228$    & $1.41e-06$ & $1.16e-04$    \\
      $1$     		& $245$    & $1.32e-06$ & $7.87e-06$   	\\   
      \hline
    \end{tabular}
  \end{subtable}%
  \begin{subtable}{0.6\textwidth} 
    \subcaption{$\bm{H}_0(div, \Omega)$ problem.}
    \centering
    \begin{tabular}{|c|c|c|c|c|}
      \hline
      $n$ & \CG Iter & Res. Error & $l^2$ Error\\ 
      \hline
      \hline
   	  $10^{-7}$ 	& $42$     & $1.23e-06$ & $2.29e-01$    \\
 	  $10^{-6}$ 	& $42$     & $1.23e-06$ & $2.29e-01$    \\
 	  $10^{-5}$ 	& $42$     & $1.36e-06$ & $2.29e-01$    \\
      $10^{-4}$ 	& $68$     & $9.86e-07$ & $1.54e-02$    \\
      $10^{-3}$ 	& $89$     & $1.13e-06$ & $7.78e-03$    \\
      $10^{-2}$ 	& $207$    & $1.08e-06$ & $7.16e-04$    \\
      $10^{-1}$ 	& $249$    & $1.28e-06$ & $6.16e-05$    \\
      $1$     		& $245$    & $1.46e-06$ & $7.74e-06$   	\\  
      \hline
    \end{tabular}
  \end{subtable}
  \end{adjustbox}
  \caption{$2$-$d$ unpreconditioned problem: \CG iterations, residual and $l^2$ approximation errors. Exact solutions are given by \eqref{eq:curl-exact-sol-2}--\eqref{eq:div-exact-sol-2}. Parameter values are set to $n=32$, $p = 3$.}
  \label{tab:cgn-res-err}
\end{table}

In Subsection \ref{sec:none}, we investigate the unpreconditioned system, which allows us to evaluate the importance of the ASP preconditioner by comparing the obtained results with those of Subsection \ref{sec:jacobi}. In that subsection, we develop numerical tests related to the auxiliary space preconditioning method using both Jacobi and Gauss-Seidel smoothing schemes. Later in Subsection \ref{sec:glt}, we examine the behavior of the preconditioner with respect to $p$-refinement. We demonstrate that the resulting algorithm can be easily extended to a $p$-stable algorithm that exhibits excellent convergence behavior of the preconditioner with respect to the $B$-spline degree $p$.

\subsubsection{Test 1: the $2$-$d$ unpreconditioned system}\label{sec:none} 
We computed the condition number $\kappa_2(\bm{A}_{\bm{curl}})$ and the number of iterations required for convergence of the \CG solver for various choices of $p$, $n$, and $\tau$. The exact solutions were defined as small perturbations from the corresponding null spaces:
\begin{equation}\label{eq:curl-exact-sol-1}
\bm{u}_{\bm{curl}}(x_1, x_2) = \tau^{-1} \begin{pmatrix}
                          x_2 (x_2-1)(2x_1-1) \\
                          x_1 (x_1-1)(2x_2-1) 
                          \end{pmatrix} + 10^{-2} \bm{v}_{\bm{curl}}(x_1, x_2),
\end{equation}
and
\begin{equation}\label{eq:div-exact-sol-1}
\bm{u}_{div}(x_1, x_2) = \tau^{-1} \begin{pmatrix}
                          x_1 (x_1-1)(2x_2-1) \\
                          x_2 (x_2-1)(2x_1-1) 
                          \end{pmatrix} + 10^{-2} \bm{v}_{div}(x_1, x_2),
\end{equation}
where $\bm{v}_{\bm{curl}}$ and $\bm{v}_{div}$ are solutions of \eqref{eq:curl-curl} and \eqref{eq:div-div} respectively, with $f = \begin{pmatrix}
1\\
1
\end{pmatrix}$. Simple computation shows that 
\begin{equation}\label{eq:curl-exact-sol-2}
\bm{v}_{\bm{curl}}(x_1, x_2) = C_1 \begin{pmatrix}
 e^{-\sqrt{\tau} x_2 + \sqrt{\tau}/2} + e^{\sqrt{\tau} x_2 - \sqrt{\tau}/2} \\
 e^{-\sqrt{\tau} x_1 + \sqrt{\tau}/2} + e^{\sqrt{\tau} x_1 - \sqrt{\tau}/2}
\end{pmatrix} + \tau^{-1} \begin{pmatrix}
1\\
1
\end{pmatrix},
\end{equation}
and
\begin{equation}\label{eq:div-exact-sol-2}
\bm{v}_{div}(x_1, x_2) = C_2 \begin{pmatrix}
\cos\left( \sqrt{\tau} x_1 - \sqrt{\tau}/2\right)\\
\cos\left( \sqrt{\tau} x_2 - \sqrt{\tau}/2\right)
\end{pmatrix} + \tau^{-1} \begin{pmatrix}
1\\
1
\end{pmatrix},
\end{equation}
where $C_1$ and $C_2$ are given by
$$
C_1 =\frac{-\tau^{-1}}{e^{-\sqrt{\tau}/2} + e^{\sqrt{\tau}/2}}, \quad C_2 = \frac{-\tau^{-1}}{\cos \left( \sqrt{\tau}/2\right)}. 
$$


\begin{table}[H]

\begin{adjustbox}{width=1.\textwidth}
\begin{subtable}{1.\textwidth} 
    \centering 
    \hspace*{-1.5 cm}
\begin{tabular}{c p{0.cm} p{1.5cm} p{1.5cm} p{1.5cm} p{1.5cm}  c p{1.5cm} p{1.5cm} p{1.5cm} p{1.5cm}}
\hline
 &  & \multicolumn{4}{l}{$p=1$} &  & \multicolumn{4}{l}{$p=2$} \\ \cline{3-6} \cline{8-11} 
\diaghead{taun-----}{$\tau$}{$n$} &  & $8$   & $16$   & $32$   & $64$    & & $8$   & $16$   & $32$   & $64$    \\ \hline
$10^{-4}$    & & $9.96e+00$ & $1.43e+01$ & $1.83e+01$ & $2.19e+01$ & & $6.07e+00$ & $9.98e+00$ & $1.39e+01$ & $1.75e+01$ \\
$10^{-3}$    & & $9.96e+00$ & $1.43e+01$ & $1.83e+01$ & $2.19e+01$ & & $6.07e+00$ & $9.98e+00$ & $1.39e+01$ & $1.75e+01$ \\
$10^{-2}$    & & $9.95e+00$ & $1.43e+01$ & $1.83e+01$ & $2.19e+01$ & & $6.07e+00$ & $9.97e+00$ & $1.39e+01$ & $1.75e+01$ \\
$10^{-1}$    & & $9.89e+00$ & $1.43e+01$ & $1.82e+01$ & $2.18e+01$ & & $6.03e+00$ & $9.92e+00$ & $1.39e+01$ & $1.74e+01$ \\
$1$          & & $9.38e+00$ & $1.36e+01$ & $1.74e+01$ & $2.09e+01$ & & $5.70e+00$ & $9.41e+00$ & $1.32e+01$ & $1.67e+01$ \\
$10^{1}$     & & $6.95e+00$ & $1.05e+01$ & $1.36e+01$ & $1.68e+01$ & & $3.91e+00$ & $7.02e+00$ & $1.02e+01$ & $1.31e+01$ \\
$10^{2}$     & & $2.55e+00$ & $5.06e+00$ & $8.06e+00$ & $1.12e+01$ & & $3.59e+00$ & $3.41e+00$ & $5.31e+00$ & $7.98e+00$ \\
$10^{3}$     & & $3.07e+00$ & $2.16e+00$ & $3.25e+00$ & $6.04e+00$ & & $1.39e+01$ & $7.01e+00$ & $3.39e+00$ & $3.72e+00$ \\
$10^{4}$     & & $5.55e+00$ & $4.46e+00$ & $2.62e+00$ & $2.18e+00$ & & $2.62e+01$ & $2.29e+01$ & $1.32e+01$ & $5.53e+00$ 
\end{tabular}
\end{subtable}
\end{adjustbox}

\vspace{.5cm}

\begin{adjustbox}{width=1.\textwidth}
\begin{subtable}{\textwidth} 
    \centering 
    \hspace*{-1.5 cm}
\begin{tabular}{c c p{1.5cm} p{1.5cm} p{1.5cm} p{1.5cm}  c p{1.5cm} p{1.5cm} p{1.5cm} p{1.5cm}}
\hline
 &  & \multicolumn{4}{l}{$p=3$} &  & \multicolumn{4}{l}{$p=4$} \\ \cline{3-6} \cline{8-11} 
\diaghead{taun-----}{$\tau$}{$n$} &  & $8$   & $16$   & $32$   & $64$    & & $8$   & $16$   & $32$   & $64$    \\ \hline
$10^{-4}$    & & $1.01e+01$ & $1.08e+01$ & $1.50e+01$ & $1.96e+01$ & & $3.23e+01$ & $3.71e+01$ & $4.34e+01$ & $4.83e+01$ \\
$10^{-3}$    & & $1.01e+01$ & $1.08e+01$ & $1.50e+01$ & $1.96e+01$ & & $3.23e+01$ & $3.71e+01$ & $4.34e+01$ & $4.83e+01$ \\
$10^{-2}$    & & $1.01e+01$ & $1.08e+01$ & $1.50e+01$ & $1.96e+01$ & & $3.23e+01$ & $3.71e+01$ & $4.34e+01$ & $4.83e+01$ \\
$10^{-1}$    & & $1.01e+01$ & $1.08e+01$ & $1.49e+01$ & $1.95e+01$ & & $3.23e+01$ & $3.71e+01$ & $4.34e+01$ & $4.83e+01$ \\
$1$          & & $1.01e+01$ & $1.08e+01$ & $1.42e+01$ & $1.86e+01$ & & $3.22e+01$ & $3.71e+01$ & $4.34e+01$ & $4.83e+01$ \\
$10^{1}$     & & $1.02e+01$ & $1.09e+01$ & $1.12e+01$ & $1.44e+01$ & & $3.19e+01$ & $3.72e+01$ & $4.35e+01$ & $4.83e+01$ \\
$10^{2}$     & & $1.22e+01$ & $1.16e+01$ & $1.14e+01$ & $1.13e+01$ & & $3.48e+01$ & $3.79e+01$ & $4.44e+01$ & $4.87e+01$ \\
$10^{3}$     & & $4.25e+01$ & $2.75e+01$ & $1.43e+01$ & $1.21e+01$ & & $9.80e+01$ & $7.58e+01$ & $5.41e+01$ & $5.26e+01$ \\
$10^{4}$     & & $8.52e+01$ & $8.11e+01$ & $5.94e+01$ & $2.69e+01$ & & $2.05e+02$ & $1.84e+02$ & $1.79e+02$ & $1.23e+02$
\end{tabular}
\end{subtable}
\end{adjustbox}

\vspace{.5cm}

\begin{adjustbox}{width=1.\textwidth}
\begin{subtable}{1.\textwidth} 
    \centering 
    \hspace*{-1.5 cm}
\begin{tabular}{c c p{1.5cm} p{1.5cm} p{1.5cm} p{1.5cm}  c p{1.5cm} p{1.5cm} p{1.5cm} p{1.5cm}}
\hline
 &  & \multicolumn{4}{l}{$p=5$} &  & \multicolumn{4}{l}{$p=6$} \\ \cline{3-6} \cline{8-11} 
\diaghead{taun-----}{$\tau$}{$n$} &  & $8$   & $16$   & $32$   & $64$    & & $8$   & $16$   & $32$   & $64$    \\ \hline
$10^{-4}$    & & $1.05e+02$ & $1.15e+02$ & $1.42e+02$ & $1.86e+02$ & & $3.56e+02$ & $3.52e+02$ & $4.26e+02$ & $5.89e+02$ \\
$10^{-3}$    & & $1.05e+02$ & $1.15e+02$ & $1.42e+02$ & $1.86e+02$ & & $3.56e+02$ & $3.52e+02$ & $4.26e+02$ & $5.89e+02$ \\
$10^{-2}$    & & $1.05e+02$ & $1.15e+02$ & $1.42e+02$ & $1.86e+02$ & & $3.56e+02$ & $3.52e+02$ & $4.26e+02$ & $5.89e+02$ \\
$10^{-1}$    & & $1.05e+02$ & $1.15e+02$ & $1.42e+02$ & $1.86e+02$ & & $3.56e+02$ & $3.52e+02$ & $4.26e+02$ & $5.89e+02$ \\
$1$          & & $1.05e+02$ & $1.15e+02$ & $1.42e+02$ & $1.86e+02$ &  & $3.54e+02$ & $3.52e+02$ & $4.26e+02$ & $5.89e+02$ \\
$10^{1}$     & & $1.02e+02$ & $1.14e+02$ & $1.42e+02$ & $1.86e+02$ & & $3.43e+02$ & $3.48e+02$ & $4.25e+02$ & $5.89e+02$ \\
$10^{2}$     & & $1.02e+02$ & $1.11e+02$ & $1.43e+02$ & $1.87e+02$ & & $3.48e+02$ & $3.28e+02$ & $4.21e+02$ & $5.92e+02$ \\
$10^{3}$     & & $2.39e+02$ & $1.84e+02$ & $1.63e+02$ & $2.01e+02$ & & $6.68e+02$ & $4.76e+02$ & $4.32e+02$ & $6.24e+02$ \\
$10^{4}$     & & $5.14e+02$ & $3.86e+02$ & $4.10e+02$ & $4.12e+02$ & & $1.41e+03$ & $8.86e+02$ & $8.46e+02$ & $1.14e+03$
\end{tabular}
\end{subtable}
\end{adjustbox}
\caption{$2$-$d$ preconditioned $\bm{H}_0(\bm{curl}, \Omega)$: condition number $\kappa_2 \left( \bm{B}_{\bm{curl}} \bm{A}_{\bm{curl}} \right)$ in the case of Jacobi smoothing.}  
  \label{tab:curl-cn-jacobi}  
\end{table}


Note that the functions defined above, namely $\bm{u}_{\bm{curl}} \in \bm{H}_0(\bm{curl}, \Omega)$ and $\bm{u}_{div} \in \bm{H}_0(div, \Omega)$, are solutions to problems \eqref{eq:curl-curl}-\eqref{eq:div-div} with right-hand sides given by
$$
\bm{f}_{\bm{curl}} = 10^{-2} + \begin{pmatrix}
                          x_2 (x_2-1)(2x_1-1) \\
                          x_1 (x_1-1)(2x_2-1) 
                          \end{pmatrix}, \quad (x_1, x_2) \in (0, 1)^2,
$$
and 
$$
\bm{f}_{div} =  10^{-2} + \begin{pmatrix}
                          x_1 (x_1-1)(2x_2-1) \\
                          x_2 (x_2-1)(2x_1-1) 
                          \end{pmatrix}, \quad (x_1, x_2) \in (0, 1)^2,
$$
respectively. It is worth mentioning that both $\bm{f}_{\bm{curl}}$ and $\bm{f}_{div}$ are independent of the parameter $\tau$.


\begin{table}[H]

\begin{adjustbox}{width=1.\textwidth}
\begin{subtable}{1.\textwidth} 
    \centering 
    \hspace*{-1.5 cm}
\begin{tabular}{c p{0.cm} p{1.5cm} p{1.5cm} p{1.5cm} p{1.5cm}  c p{1.5cm} p{1.5cm} p{1.5cm} p{1.5cm}}
\hline
 &  & \multicolumn{4}{l}{$p=1$} &  & \multicolumn{4}{l}{$p=2$} \\ \cline{3-6} \cline{8-11} 
\diaghead{taun-----}{$\tau$}{$n$} &  & $8$   & $16$   & $32$   & $64$    & & $8$   & $16$   & $32$   & $64$    \\ \hline
$10^{-4}$    & & $4.52e+00$ & $5.94e+00$ & $8.22e+00$ & $1.04e+01$ & & $3.75e+00$ & $4.77e+00$ & $7.12e+00$ & $9.61e+00$ \\
$10^{-3}$    & & $4.52e+00$ & $5.94e+00$ & $8.21e+00$ & $1.04e+01$ & & $3.75e+00$ & $4.77e+00$ & $7.12e+00$ & $9.61e+00$ \\
$10^{-2}$    & & $4.52e+00$ & $5.94e+00$ & $8.21e+00$ & $1.04e+01$ & & $3.74e+00$ & $4.77e+00$ & $7.12e+00$ & $9.61e+00$ \\
$10^{-1}$    & & $4.52e+00$ & $5.90e+00$ & $8.17e+00$ & $1.03e+01$ & & $3.73e+00$ & $4.74e+00$ & $7.08e+00$ & $9.56e+00$ \\
$1$          & & $4.49e+00$ & $5.61e+00$ & $7.78e+00$ & $9.85e+00$ & & $3.63e+00$ & $4.49e+00$ & $6.72e+00$ & $9.10e+00$ \\
$10^{1}$     & & $4.23e+00$ & $4.56e+00$ & $6.02e+00$ & $7.75e+00$ & & $3.21e+00$ & $3.57e+00$ & $5.02e+00$ & $7.02e+00$ \\
$10^{2}$     & & $3.80e+00$ & $4.13e+00$ & $4.50e+00$ & $4.75e+00$ & & $3.07e+00$ & $3.34e+00$ & $3.39e+00$ & $3.87e+00$ \\
$10^{3}$     & & $2.97e+00$ & $3.67e+00$ & $4.75e+00$ & $4.61e+00$ & & $4.29e+00$ & $3.39e+00$ & $3.57e+00$ & $3.70e+00$ \\
$10^{4}$     & & $3.12e+00$ & $3.05e+00$ & $3.19e+00$ & $4.22e+00$ & & $6.45e+00$ & $5.52e+00$ & $3.98e+00$ & $3.38e+00$ 
\end{tabular}
\end{subtable}
\end{adjustbox}

\vspace{.5cm}

\begin{adjustbox}{width=1.\textwidth}
\begin{subtable}{\textwidth} 
    \centering 
    \hspace*{-1.5 cm}
\begin{tabular}{c c p{1.5cm} p{1.5cm} p{1.5cm} p{1.5cm}  c p{1.5cm} p{1.5cm} p{1.5cm} p{1.5cm}}
\hline
 &  & \multicolumn{4}{l}{$p=3$} &  & \multicolumn{4}{l}{$p=4$} \\ \cline{3-6} \cline{8-11} 
\diaghead{taun-----}{$\tau$}{$n$} &  & $8$   & $16$   & $32$   & $64$    & & $8$   & $16$   & $32$   & $64$    \\ \hline
$10^{-4}$    & & $4.86e+00$ & $4.90e+00$ & $6.93e+00$ & $9.75e+00$ & & $1.10e+01$ & $1.07e+01$ & $1.05e+01$ & $1.05e+01$ \\
$10^{-3}$    & & $4.86e+00$ & $4.90e+00$ & $6.93e+00$ & $9.75e+00$ & & $1.10e+01$ & $1.07e+01$ & $1.05e+01$ & $1.05e+01$ \\
$10^{-2}$    & & $4.86e+00$ & $4.90e+00$ & $6.92e+00$ & $9.74e+00$ & & $1.10e+01$ & $1.07e+01$ & $1.05e+01$ & $1.05e+01$ \\
$10^{-1}$    & & $4.83e+00$ & $4.89e+00$ & $6.88e+00$ & $9.69e+00$ & & $1.09e+01$ & $1.07e+01$ & $1.05e+01$ & $1.05e+01$ \\
$1$          & & $4.62e+00$ & $4.81e+00$ & $6.52e+00$ & $9.22e+00$ & & $1.07e+01$ & $1.06e+01$ & $1.05e+01$ & $1.05e+01$ \\
$10^{1}$     & & $3.89e+00$ & $4.36e+00$ & $4.73e+00$ & $7.02e+00$ & & $9.89e+00$ & $1.00e+01$ & $1.03e+01$ & $1.04e+01$ \\
$10^{2}$     & & $4.46e+00$ & $4.13e+00$ & $4.06e+00$ & $4.50e+00$ & & $1.10e+01$ & $1.01e+01$ & $9.52e+00$ & $1.01e+01$ \\
$10^{3}$     & & $9.43e+00$ & $6.33e+00$ & $4.80e+00$ & $4.54e+00$ & & $2.16e+01$ & $1.59e+01$ & $1.28e+01$ & $1.08e+01$ \\
$10^{4}$     & & $1.84e+01$ & $1.54e+01$ & $1.01e+01$ & $5.94e+00$ & & $4.94e+01$ & $3.80e+01$ & $3.06e+01$ & $1.92e+01$
\end{tabular}
\end{subtable}
\end{adjustbox}

\vspace{.5cm}

\begin{adjustbox}{width=1.\textwidth}
\begin{subtable}{1.\textwidth} 
    \centering 
    \hspace*{-1.5 cm}
\begin{tabular}{c c p{1.5cm} p{1.5cm} p{1.5cm} p{1.5cm}  c p{1.5cm} p{1.5cm} p{1.5cm} p{1.5cm}}
\hline
 &  & \multicolumn{4}{l}{$p=5$} &  & \multicolumn{4}{l}{$p=6$} \\ \cline{3-6} \cline{8-11} 
\diaghead{taun-----}{$\tau$}{$n$} &  & $8$   & $16$   & $32$   & $64$    & & $8$   & $16$   & $32$   & $64$    \\ \hline
$10^{-4}$    & & $4.14e+01$ & $3.92e+01$ & $3.84e+01$ & $3.82e+01$ & & $2.39e+02$ & $2.01e+02$ & $2.06e+02$ & $2.06e+02$ \\
$10^{-3}$    & & $4.14e+01$ & $3.92e+01$ & $3.84e+01$ & $3.82e+01$ & & $2.39e+02$ & $2.01e+02$ & $2.06e+02$ & $2.06e+02$ \\
$10^{-2}$    & & $4.14e+01$ & $3.92e+01$ & $3.84e+01$ & $3.82e+01$ & & $2.39e+02$ & $2.01e+02$ & $2.06e+02$ & $2.06e+02$ \\
$10^{-1}$    & & $4.14e+01$ & $3.92e+01$ & $3.84e+01$ & $3.82e+01$ & & $2.39e+02$ & $2.01e+02$ & $2.06e+02$ & $2.06e+02$ \\
$1$          & & $4.12e+01$ & $3.92e+01$ & $3.84e+01$ & $3.82e+01$ & & $2.39e+02$ & $2.01e+02$ & $2.06e+02$ & $2.06e+02$ \\
$10^{1}$     & & $4.12e+01$ & $3.92e+01$ & $3.84e+01$ & $3.82e+01$ & & $2.41e+02$ & $2.02e+02$ & $2.06e+02$ & $2.06e+02$ \\
$10^{2}$     & & $4.41e+01$ & $4.07e+01$ & $3.90e+01$ & $3.83e+01$ & & $2.57e+02$ & $2.08e+02$ & $2.10e+02$ & $2.08e+02$ \\
$10^{3}$     & & $5.86e+01$ & $4.96e+01$ & $4.61e+01$ & $4.13e+01$ & & $3.04e+02$ & $2.35e+02$ & $2.36e+02$ & $2.23e+02$ \\
$10^{4}$     & & $1.36e+02$ & $8.63e+01$ & $7.33e+01$ & $7.72e+01$ & & $4.52e+02$ & $2.96e+02$ & $3.03e+02$ & $3.33e+02$
\end{tabular}
\end{subtable}
\end{adjustbox}
\caption{$2$-$d$ preconditioned $\bm{H}_0(\bm{curl}, \Omega)$: condition number $\kappa_2 \left( \bm{B}_{\bm{curl}} \bm{A}_{\bm{curl}} \right)$ in the case of a Gauss-Seidel smoothing.}  
  \label{tab:curl-cn-Gauss-Seidel}  
\end{table}


The results are summarized in tables \ref{tab:curl-cn-none} and \ref{tab:curl-cgn-none}. As expected, we found that the spectral condition number is very large and increases with $n$. Furthermore, it becomes extremely large as $p$ increases and as $\tau$ approaches $0$. Similar observations apply to the number of \CG iterations. The results for the $\bm{H}_0(div, \Omega)$ problem are similar, and therefore, we do not report them here.


\begin{table}[H]

\begin{adjustbox}{width=1.\textwidth}
\begin{subtable}{1.\textwidth} 
    \centering 
    \hspace*{-2.5 cm}
\begin{tabular}{c p{0.cm} p{.5cm}p{.5cm} p{0.cm} p{.5cm}p{.5cm} p{0.cm} p{.5cm}p{.5cm} p{0.cm} p{.5cm}p{.5cm}p{0.cm} p{.5cm}p{.5cm} p{0.cm} p{.5cm}p{.5cm} p{0.cm} p{.5cm}p{.5cm} p{0.cm} p{.5cm}p{.5cm}}
\cline{2-25}
\multirow{2}{*}{} &                      & \multicolumn{11}{l}{$p=1$}                                                                                        &  & \multicolumn{11}{l}{$p=2$}   \\   
\cline{3-13} \cline{15-25} 
                 \multirow{2}{*}{\diagbox[innerwidth=0.8cm]{\hspace*{0.15cm}$\tau$}{$n$}}   & \multicolumn{1}{c}{} & \multicolumn{2}{l}{$8$} &  & \multicolumn{2}{l}{$16$} &  & \multicolumn{2}{l}{$32$} &  & \multicolumn{2}{l}{$64$} &  & \multicolumn{2}{l}{$8$} &  & \multicolumn{2}{l}{$16$} &  & \multicolumn{2}{l}{$32$} &  & \multicolumn{2}{l}{$64$} \\ \cline{3-4} \cline{6-7} \cline{9-10} \cline{12-13} \cline{15-16} \cline{18-19} \cline{21-22} \cline{24-25} 
           &                      & J         & GS           &  & J          & GS           &  & J          & GS           &  & J          & GS           &  & J         & GS           &  & J          & GS           &  & J          & GS           &  & J          & GS           \\ \hline

$10^{-4}$ &&     $13$ & $10$      &&      $17$ & $11$       &&        $20$ & $13$        &&       $23$ & $16$         
          &&     $11$ & $10$      &&      $13$ & $12$       &&        $17$ & $14$        &&       $20$ & $16$            \\
$10^{-3}$ &&     $13$ & $10$      &&      $17$ & $11$       &&        $20$ & $13$        &&       $23$ & $16$           
          &&     $11$ & $10$      &&      $13$ & $12$       &&        $17$ & $14$        &&       $20$ & $16$            \\
$10^{-2}$ &&     $13$ & $10$      &&      $17$ & $11$       &&        $20$ & $13$        &&       $23$ & $16$           
          &&     $11$ & $10$      &&      $14$ & $12$       &&        $17$ & $14$        &&       $20$ & $16$          \\
$10^{-1}$ &&     $13$ & $10$      &&      $17$ & $11$       &&        $20$ & $13$        &&       $23$ & $16$          
          &&     $11$ & $10$      &&      $14$ & $12$       &&        $17$ & $14$        &&       $20$ & $16$            \\
$1$       &&     $13$ & $10$      &&      $17$ & $11$       &&        $20$ & $13$        &&       $20$ & $16$           
          &&     $12$ & $10$      &&      $14$ & $12$       &&        $17$ & $13$        &&       $20$ & $15$           \\
$10$      &&     $13$ & $9$       &&      $16$ & $11$       &&        $20$ & $13$        &&       $22$ & $15$           
          &&     $11$ & $9$       &&      $13$ & $10$       &&        $16$ & $12$        &&       $19$ & $14$            \\
$10^2$    &&     $8$  & $9$       &&      $13$ & $9$        &&        $16$ & $10$        &&       $19$ & $13$          
          &&     $9$  & $9$       &&      $10$ & $9$        &&        $13$ & $9$         &&       $16$ & $11$            \\
$10^3$    &&     $8$  & $9$       &&      $8$  & $9$        &&        $13$ & $10$        &&       $13$ & $9$           
          &&     $13$ & $10$      &&      $11$ & $9$        &&        $8$  & $9$         &&       $11$ & $9$            \\
$10^4$    &&     $10$ & $9$       &&      $10$ & $9$        &&        $8$  & $9$         &&       $8$  & $9$  
          &&     $15$ & $12$      &&      $17$ & $10$       &&        $13$ & $9$         &&       $9$  & $9$          
\end{tabular}
\end{subtable}
\end{adjustbox}

\vspace{.5cm}

\begin{adjustbox}{width=1.\textwidth}
\begin{subtable}{1.\textwidth} 
    \centering 
    \hspace*{-2.5 cm}
\begin{tabular}{c p{0.cm} p{.5cm}p{.5cm} p{0.cm} p{.5cm}p{.5cm} p{0.cm} p{.5cm}p{.5cm} p{0.cm} p{.5cm}p{.5cm}p{0.cm} p{.5cm}p{.5cm} p{0.cm} p{.5cm}p{.5cm} p{0.cm} p{.5cm}p{.5cm} p{0.cm} p{.5cm}p{.5cm}}
\cline{2-25}
\multirow{2}{*}{} &                      & \multicolumn{11}{l}{$p=3$}                                                                                        &  & \multicolumn{11}{l}{$p=4$}   \\   
\cline{3-13} \cline{15-25} 
                 \multirow{2}{*}{\diagbox[innerwidth=0.8cm]{\hspace*{0.15cm}$\tau$}{$n$}}   & \multicolumn{1}{c}{} & \multicolumn{2}{l}{$8$} &  & \multicolumn{2}{l}{$16$} &  & \multicolumn{2}{l}{$32$} &  & \multicolumn{2}{l}{$64$} &  & \multicolumn{2}{l}{$8$} &  & \multicolumn{2}{l}{$16$} &  & \multicolumn{2}{l}{$32$} &  & \multicolumn{2}{l}{$64$} \\ \cline{3-4} \cline{6-7} \cline{9-10} \cline{12-13} \cline{15-16} \cline{18-19} \cline{21-22} \cline{24-25} 
           &                      & J         & GS           &  & J          & GS           &  & J          & GS           &  & J          & GS           &  & J         & GS           &  & J          & GS           &  & J          & GS           &  & J          & GS           \\ \hline

$10^{-4}$ &&     $12$ & $10$      &&      $15$ & $12$       &&        $17$ & $13$        &&       $21$ & $15$         
          &&     $15$ & $12$      &&      $19$ & $12$       &&        $22$ & $13$        &&       $25$ & $15$            \\
$10^{-3}$ &&     $13$ & $10$      &&      $15$ & $12$       &&        $17$ & $13$        &&       $21$ & $15$           
          &&     $15$ & $12$      &&      $19$ & $12$       &&        $22$ & $13$        &&       $25$ & $15$            \\
$10^{-2}$ &&     $13$ & $10$      &&      $15$ & $12$       &&        $18$ & $13$        &&       $21$ & $15$           
          &&     $15$ & $12$      &&      $19$ & $12$       &&        $23$ & $13$        &&       $25$ & $15$          \\
$10^{-1}$ &&     $13$ & $10$      &&      $15$ & $12$       &&        $18$ & $13$        &&       $21$ & $15$          
          &&     $15$ & $12$      &&      $19$ & $12$       &&        $23$ & $13$        &&       $25$ & $15$            \\
$1$       &&     $13$ & $10$      &&      $15$ & $11$       &&        $18$ & $12$        &&       $21$ & $15$           
          &&     $15$ & $12$      &&      $20$ & $12$       &&        $23$ & $13$        &&       $25$ & $15$            \\
$10$      &&     $12$ & $9$       &&      $15$ & $9$        &&        $17$ & $11$        &&       $20$ & $14$           
          &&     $15$ & $11$      &&      $19$ & $11$       &&        $21$ & $11$        &&       $23$ & $13$            \\
$10^2$    &&     $12$ & $9$       &&      $15$ & $10$       &&        $17$ & $9$         &&       $20$ & $10$          
          &&     $13$ & $11$      &&      $16$ & $11$       &&        $17$ & $11$        &&       $19$ & $10$            \\
$10^3$    &&     $12$ & $11$      &&      $13$ & $10$       &&        $14$ & $10$        &&       $17$ & $10$           
          &&     $16$ & $13$      &&      $17$ & $12$       &&        $13$ & $11$        &&       $14$ & $10$            \\
$10^4$    &&     $19$ & $15$      &&      $21$ & $12$       &&        $17$ & $10$        &&       $12$ & $10$  
          &&     $20$ & $16$      &&      $23$ & $14$       &&        $18$ & $11$        &&       $13$ & $10$       
\end{tabular}
\end{subtable}
\end{adjustbox}

\vspace{.5cm}

\begin{adjustbox}{width=1.\textwidth}
\begin{subtable}{1.\textwidth} 
    \centering 
    \hspace*{-2.5 cm}
\begin{tabular}{c p{0.cm} p{.5cm}p{.5cm} p{0.cm} p{.5cm}p{.5cm} p{0.cm} p{.5cm}p{.5cm} p{0.cm} p{.5cm}p{.5cm}p{0.cm} p{.5cm}p{.5cm} p{0.cm} p{.5cm}p{.5cm} p{0.cm} p{.5cm}p{.5cm} p{0.cm} p{.5cm}p{.5cm}}
\cline{2-25}
\multirow{2}{*}{} &                      & \multicolumn{11}{l}{$p=5$}                                                                                        &  & \multicolumn{11}{l}{$p=6$}   \\   
\cline{3-13} \cline{15-25} 
                 \multirow{2}{*}{\diagbox[innerwidth=0.8cm]{\hspace*{0.15cm}$\tau$}{$n$}}   & \multicolumn{1}{c}{} & \multicolumn{2}{l}{$8$} &  & \multicolumn{2}{l}{$16$} &  & \multicolumn{2}{l}{$32$} &  & \multicolumn{2}{l}{$64$} &  & \multicolumn{2}{l}{$8$} &  & \multicolumn{2}{l}{$16$} &  & \multicolumn{2}{l}{$32$} &  & \multicolumn{2}{l}{$64$} \\ \cline{3-4} \cline{6-7} \cline{9-10} \cline{12-13} \cline{15-16} \cline{18-19} \cline{21-22} \cline{24-25} 
           &                      & J         & GS           &  & J          & GS           &  & J          & GS           &  & J          & GS           &  & J         & GS           &  & J          & GS           &  & J          & GS           &  & J          & GS           \\ \hline

$10^{-4}$ &&     $18$ & $14$      &&      $21$ & $16$       &&        $27$ & $17$        &&       $29$ & $15$         
          &&     $19$ & $15$      &&      $24$ & $15$       &&        $30$ & $18$        &&       $33$ & $17$            \\
$10^{-3}$ &&     $18$ & $14$      &&      $21$ & $16$       &&        $27$ & $17$        &&       $29$ & $15$           
          &&     $19$ & $15$      &&      $24$ & $15$       &&        $29$ & $18$        &&       $33$ & $17$            \\
$10^{-2}$ &&     $18$ & $14$      &&      $21$ & $16$       &&        $27$ & $17$        &&       $29$ & $15$           
          &&     $19$ & $15$      &&      $24$ & $15$       &&        $29$ & $18$        &&       $33$ & $17$          \\
$10^{-1}$ &&     $18$ & $14$      &&      $21$ & $15$       &&        $27$ & $17$        &&       $29$ & $15$          
          &&     $19$ & $15$      &&      $24$ & $15$       &&        $29$ & $18$        &&       $33$ & $17$            \\
$1$       &&     $18$ & $14$      &&      $21$ & $15$       &&        $27$ & $16$        &&       $29$ & $14$           
          &&     $20$ & $15$      &&      $23$ & $15$       &&        $29$ & $18$        &&       $32$ & $16$            \\
$10$      &&     $18$ & $13$      &&      $21$ & $14$       &&        $24$ & $15$        &&       $28$ & $15$           
          &&     $19$ & $14$      &&      $22$ & $14$       &&        $28$ & $16$        &&       $30$ & $16$            \\
$10^2$    &&     $17$ & $11$      &&      $18$ & $13$       &&        $19$ & $13$        &&       $21$ & $11$          
          &&     $17$ & $13$      &&      $19$ & $13$       &&        $21$ & $13$        &&       $22$ & $12$            \\
$10^3$    &&     $19$ & $13$      &&      $19$ & $12$       &&        $15$ & $10$        &&       $15$ & $11$           
          &&     $19$ & $13$      &&      $20$ & $13$       &&        $17$ & $13$        &&       $16$ & $12$            \\
$10^4$    &&     $21$ & $16$      &&      $24$ & $14$       &&        $19$ & $11$        &&       $14$ & $9$   
          &&     $20$ & $17$      &&      $23$ & $14$       &&        $19$ & $12$        &&       $15$ & $11$        
\end{tabular}
\end{subtable}
\end{adjustbox}
\caption{$2$-$d$ preconditioned $\bm{H}_0(\bm{curl}, \Omega)$: \CG iterations  in the case of Jacobi (J) and the Gauss-Seidel (GS) smoothing. Exact solution is given by \eqref{eq:curl-exact-sol-1}.}  
  \label{tab:curl-cgn}   
\end{table}

It's important to note that, for certain types of problems, reaching the stopping criterion and having a decrease in residual error does not necessarily mean that the solution has converged to the exact one. In fact, it could lead to a non-physical solution. To test this scenario, we modified the analytic solutions \eqref{eq:curl-exact-sol-1}-\eqref{eq:div-exact-sol-1} by considering only the parts corresponding to right hand sides equal to $\begin{pmatrix}
1\\
1
\end{pmatrix}$, i.e $\bm{v}_{\bm{curl}}$ and $\bm{v}_{div}$ defined in \eqref{eq:curl-exact-sol-2} and \eqref{eq:div-exact-sol-2}.


\begin{table}[H]

\begin{adjustbox}{width=1.\textwidth}
\begin{subtable}{1.\textwidth} 
    \centering 
    \hspace*{-1.5 cm}
\begin{tabular}{c p{0.cm} p{1.5cm} p{1.5cm} p{1.5cm} p{1.5cm}  c p{1.5cm} p{1.5cm} p{1.5cm} p{1.5cm}}
\hline
 &  & \multicolumn{4}{l}{$p=1$} &  & \multicolumn{4}{l}{$p=2$} \\ \cline{3-6} \cline{8-11} 
\diaghead{taun--}{$\tau$}{$n$} &  & $8$   & $16$   & $32$   & $64$    & & $8$   & $16$   & $32$   & $64$    \\ \hline
 $10^{-4}$    & & $9.96e+00$ & $1.43e+01$ & $1.83e+01$ & $2.19e+01$ && $6.07e+00$ & $9.98e+00$ & $1.39e+01$ & $1.75e+01$ \\
 $10^{-3}$    & & $9.96e+00$ & $1.43e+01$ & $1.83e+01$ & $2.19e+01$ && $6.07e+00$ & $9.98e+00$ & $1.39e+01$ & $1.75e+01$ \\
 $10^{-2}$    & & $9.95e+00$ & $1.43e+01$ & $1.83e+01$ & $2.19e+01$ && $6.07e+00$ & $9.97e+00$ & $1.39e+01$ & $1.75e+01$ \\
 $10^{-1}$    & & $9.89e+00$ & $1.43e+01$ & $1.82e+01$ & $2.18e+01$ && $6.03e+00$ & $9.92e+00$ & $1.39e+01$ & $1.74e+01$ \\
 $1$          & & $9.38e+00$ & $1.36e+01$ & $1.74e+01$ & $2.09e+01$ && $5.70e+00$ & $9.41e+00$ & $1.32e+01$ & $1.67e+01$ \\
 $10^{1}$     & & $6.95e+00$ & $1.05e+01$ & $1.36e+01$ & $1.68e+01$ && $3.91e+00$ & $7.02e+00$ & $1.02e+01$ & $1.31e+01$ \\
 $10^{2}$     & & $2.55e+00$ & $5.06e+00$ & $8.06e+00$ & $1.12e+01$ && $3.59e+00$ & $3.41e+00$ & $5.31e+00$ & $7.98e+00$ \\
 $10^{3}$     & & $3.07e+00$ & $2.16e+00$ & $3.25e+00$ & $6.04e+00$ && $1.39e+01$ & $7.01e+00$ & $3.39e+00$ & $3.72e+00$ \\
 $10^{4}$     & & $5.55e+00$ & $4.46e+00$ & $2.62e+00$ & $2.18e+00$ && $2.62e+01$ & $2.29e+01$ & $1.32e+01$ & $5.53e+00$
\end{tabular}
\end{subtable}
\end{adjustbox}

\vspace{.5cm}

\begin{adjustbox}{width=1.\textwidth}
\begin{subtable}{\textwidth} 
    \centering 
    \hspace*{-1.5 cm}
\begin{tabular}{c c p{1.5cm} p{1.5cm} p{1.5cm} p{1.5cm}  c p{1.5cm} p{1.5cm} p{1.5cm} p{1.5cm}}
\hline
 &  & \multicolumn{4}{l}{$p=3$} &  & \multicolumn{4}{l}{$p=4$} \\ \cline{3-6} \cline{8-11} 
\diaghead{taun--}{$\tau$}{$n$} &  & $8$   & $16$   & $32$   & $64$    & & $8$   & $16$   & $32$   & $64$    \\ \hline
 $10^{-4}$    && $1.01e+01$ & $1.08e+01$ & $1.50e+01$ & $1.96e+01$ && $3.23e+01$ & $3.71e+01$ & $4.34e+01$ & $4.83e+01$ \\
 $10^{-3}$    && $1.01e+01$ & $1.08e+01$ & $1.50e+01$ & $1.96e+01$ && $3.23e+01$ & $3.71e+01$ & $4.34e+01$ & $4.83e+01$ \\
 $10^{-2}$    && $1.01e+01$ & $1.08e+01$ & $1.50e+01$ & $1.96e+01$ && $3.23e+01$ & $3.71e+01$ & $4.34e+01$ & $4.83e+01$ \\
 $10^{-1}$    && $1.01e+01$ & $1.08e+01$ & $1.49e+01$ & $1.95e+01$ && $3.23e+01$ & $3.71e+01$ & $4.34e+01$ & $4.83e+01$ \\
 $1$          && $1.01e+01$ & $1.08e+01$ & $1.42e+01$ & $1.86e+01$ && $3.22e+01$ & $3.71e+01$ & $4.34e+01$ & $4.83e+01$ \\
 $10^{1}$     && $1.02e+01$ & $1.09e+01$ & $1.12e+01$ & $1.44e+01$ && $3.19e+01$ & $3.72e+01$ & $4.35e+01$ & $4.83e+01$ \\
 $10^{2}$     && $1.22e+01$ & $1.16e+01$ & $1.14e+01$ & $1.13e+01$ && $3.48e+01$ & $3.79e+01$ & $4.44e+01$ & $4.87e+01$ \\
 $10^{3}$     && $4.25e+01$ & $2.75e+01$ & $1.43e+01$ & $1.21e+01$ && $9.80e+01$ & $7.58e+01$ & $5.41e+01$ & $5.26e+01$ \\
 $10^{4}$     && $8.52e+01$ & $8.11e+01$ & $5.94e+01$ & $2.69e+01$ && $2.05e+02$ & $1.84e+02$ & $1.79e+02$ & $1.23e+02$
\end{tabular}
\end{subtable}
\end{adjustbox}

\vspace{.5cm}

\begin{adjustbox}{width=1.\textwidth}
\begin{subtable}{1.\textwidth} 
    \centering 
    \hspace*{-1.5 cm}
\begin{tabular}{c c p{1.5cm} p{1.5cm} p{1.5cm} p{1.5cm}  c p{1.5cm} p{1.5cm} p{1.5cm} p{1.5cm}}
\hline
 &  & \multicolumn{4}{l}{$p=5$} &  & \multicolumn{4}{l}{$p=6$} \\ \cline{3-6} \cline{8-11} 
\diaghead{taun--}{$\tau$}{$n$} &  & $8$   & $16$   & $32$   & $64$    & & $8$   & $16$   & $32$   & $64$    \\ \hline
 $10^{-4}$    && $1.05e+02$ & $1.15e+02$ & $1.42e+02$ & $1.86e+02$ && $3.56e+02$ & $3.52e+02$ & $4.26e+02$ & $5.89e+02$ \\
 $10^{-3}$    && $1.05e+02$ & $1.15e+02$ & $1.42e+02$ & $1.86e+02$ && $3.56e+02$ & $3.52e+02$ & $4.26e+02$ & $5.89e+02$ \\
 $10^{-2}$    && $1.05e+02$ & $1.15e+02$ & $1.42e+02$ & $1.86e+02$ && $3.56e+02$ & $3.52e+02$ & $4.26e+02$ & $5.89e+02$ \\
 $10^{-1}$    && $1.05e+02$ & $1.15e+02$ & $1.42e+02$ & $1.86e+02$ && $3.56e+02$ & $3.52e+02$ & $4.26e+02$ & $5.89e+02$ \\
 $1$          && $1.05e+02$ & $1.15e+02$ & $1.42e+02$ & $1.86e+02$ && $3.54e+02$ & $3.52e+02$ & $4.26e+02$ & $5.89e+02$ \\
 $10^{1}$     && $1.02e+02$ & $1.14e+02$ & $1.42e+02$ & $1.86e+02$ && $3.43e+02$ & $3.48e+02$ & $4.25e+02$ & $5.89e+02$ \\
 $10^{2}$     && $1.02e+02$ & $1.11e+02$ & $1.43e+02$ & $1.87e+02$ && $3.48e+02$ & $3.28e+02$ & $4.21e+02$ & $5.92e+02$ \\
 $10^{3}$     && $2.39e+02$ & $1.84e+02$ & $1.63e+02$ & $2.01e+02$ && $6.68e+02$ & $4.76e+02$ & $4.32e+02$ & $6.24e+02$ \\
 $10^{4}$     && $5.14e+02$ & $3.86e+02$ & $4.10e+02$ & $4.12e+02$ && $1.41e+03$ & $8.86e+02$ & $8.46e+02$ & $1.14e+03$
\end{tabular}
\end{subtable}
\end{adjustbox}
\caption{$2$-$d$ preconditioned $\bm{H}_0(div, \Omega)$: condition number $\kappa_2 \left( \bm{B}_{div} \bm{A}_{div} \right)$ in the case of a Jacobi smoothing.}  
  \label{tab:div-cn-jacobi}  
\end{table}


 We evaluated the iteration counts, residual error, and relative $l^2$-error for different values of $\tau$ using a fixed number of elements ($n=32$) and a B-spline degree of $p=3$. The results are shown in Table \ref{tab:cgn-res-err}. As we observed, even when the \CG method reached the stopping criterion, the relative error was still very high for small values of $\tau$. This indicates that the approximated solution did not converge to the exact one. Next, we will show that this misleading convergence can be remedied using the ASP strategy.


\begin{table}[H]

\begin{adjustbox}{width=1.\textwidth}
\begin{subtable}{1.\textwidth} 
    \centering 
    \hspace*{-1.5 cm}
\begin{tabular}{c p{0.cm} p{1.5cm} p{1.5cm} p{1.5cm} p{1.5cm}  c p{1.5cm} p{1.5cm} p{1.5cm} p{1.5cm}}
\hline
 &  & \multicolumn{4}{l}{$p=1$} &  & \multicolumn{4}{l}{$p=2$} \\ \cline{3-6} \cline{8-11} 
\diaghead{taun-----}{$\tau$}{$n$} &  & $8$   & $16$   & $32$   & $64$    & & $8$   & $16$   & $32$   & $64$    \\ \hline
$10^{-4}$    & & $4.85e+00$ & $7.49e+00$ & $1.04e+01$ & $1.31e+01$ & & $3.49e+00$ & $5.40e+00$ & $8.04e+00$ & $1.08e+01$ \\
$10^{-3}$    & & $4.85e+00$ & $7.49e+00$ & $1.04e+01$ & $1.31e+01$ & & $3.49e+00$ & $5.40e+00$ & $8.04e+00$ & $1.08e+01$ \\
$10^{-2}$    & & $4.85e+00$ & $7.48e+00$ & $1.04e+01$ & $1.31e+01$ & & $3.48e+00$ & $5.39e+00$ & $8.03e+00$ & $1.08e+01$ \\
$10^{-1}$    & & $4.82e+00$ & $7.44e+00$ & $1.03e+01$ & $1.30e+01$ & & $3.46e+00$ & $5.36e+00$ & $7.99e+00$ & $1.08e+01$ \\
$1$          & & $4.56e+00$ & $7.06e+00$ & $9.81e+00$ & $1.24e+01$ & & $3.29e+00$ & $5.07e+00$ & $7.58e+00$ & $1.03e+01$ \\
$10^{1}$     & & $3.63e+00$ & $5.22e+00$ & $7.58e+00$ & $9.78e+00$ & & $2.93e+00$ & $3.56e+00$ & $5.63e+00$ & $7.91e+00$ \\
$10^{2}$     & & $3.33e+00$ & $3.62e+00$ & $4.06e+00$ & $5.99e+00$ & & $2.83e+00$ & $2.86e+00$ & $2.99e+00$ & $4.37e+00$ \\
$10^{3}$     & & $2.12e+00$ & $2.95e+00$ & $3.50e+00$ & $3.71e+00$ & & $3.31e+00$ & $2.84e+00$ & $2.85e+00$ & $2.90e+00$ \\
$10^{4}$     & & $2.14e+00$ & $2.06e+00$ & $2.37e+00$ & $3.27e+00$ & & $4.87e+00$ & $4.04e+00$ & $3.17e+00$ & $2.82e+00$ 
\end{tabular}
\end{subtable}
\end{adjustbox}

\vspace{.5cm}

\begin{adjustbox}{width=1.\textwidth}
\begin{subtable}{\textwidth} 
    \centering 
    \hspace*{-1.5 cm}
\begin{tabular}{c c p{1.5cm} p{1.5cm} p{1.5cm} p{1.5cm}  c p{1.5cm} p{1.5cm} p{1.5cm} p{1.5cm}}
\hline
 &  & \multicolumn{4}{l}{$p=3$} &  & \multicolumn{4}{l}{$p=4$} \\ \cline{3-6} \cline{8-11} 
\diaghead{taun-----}{$\tau$}{$n$} &  & $8$   & $16$   & $32$   & $64$    & & $8$   & $16$   & $32$   & $64$    \\ \hline
$10^{-4}$    & & $4.44e+00$ & $5.18e+00$ & $7.78e+00$ & $1.09e+01$ & & $1.19e+01$ & $1.15e+01$ & $1.13e+01$ & $1.12e+01$ \\
$10^{-3}$    & & $4.44e+00$ & $5.18e+00$ & $7.78e+00$ & $1.09e+01$ & & $1.19e+01$ & $1.15e+01$ & $1.13e+01$ & $1.12e+01$ \\
$10^{-2}$    & & $4.44e+00$ & $5.18e+00$ & $7.78e+00$ & $1.09e+01$ & & $1.19e+01$ & $1.15e+01$ & $1.13e+01$ & $1.12e+01$ \\
$10^{-1}$    & & $4.44e+00$ & $5.15e+00$ & $7.73e+00$ & $1.08e+01$ & & $1.19e+01$ & $1.15e+01$ & $1.13e+01$ & $1.12e+01$ \\
$1$          & & $4.43e+00$ & $4.87e+00$ & $7.32e+00$ & $1.03e+01$ & & $1.19e+01$ & $1.15e+01$ & $1.13e+01$ & $1.12e+01$ \\
$10^{1}$     & & $4.36e+00$ & $4.37e+00$ & $5.26e+00$ & $7.85e+00$ & & $1.18e+01$ & $1.15e+01$ & $1.13e+01$ & $1.12e+01$ \\
$10^{2}$     & & $4.83e+00$ & $4.35e+00$ & $4.33e+00$ & $4.37e+00$ & & $1.33e+01$ & $1.19e+01$ & $1.13e+01$ & $1.12e+01$ \\
$10^{3}$     & & $7.82e+00$ & $6.16e+00$ & $4.63e+00$ & $4.32e+00$ & & $2.03e+01$ & $1.81e+01$ & $1.37e+01$ & $1.15e+01$ \\
$10^{4}$     & & $1.64e+01$ & $1.20e+01$ & $8.84e+00$ & $5.78e+00$ & & $5.02e+01$ & $3.20e+01$ & $3.12e+01$ & $2.12e+01$
\end{tabular}
\end{subtable}
\end{adjustbox}

\vspace{.5cm}

\begin{adjustbox}{width=1.\textwidth}
\begin{subtable}{1.\textwidth} 
    \centering 
    \hspace*{-1.5 cm}
\begin{tabular}{c c p{1.5cm} p{1.5cm} p{1.5cm} p{1.5cm}  c p{1.5cm} p{1.5cm} p{1.5cm} p{1.5cm}}
\hline
 &  & \multicolumn{4}{l}{$p=5$} &  & \multicolumn{4}{l}{$p=6$} \\ \cline{3-6} \cline{8-11} 
\diaghead{taun-----}{$\tau$}{$n$} &  & $8$   & $16$   & $32$   & $64$    & & $8$   & $16$   & $32$   & $64$    \\ \hline
$10^{-4}$    & & $4.32e+01$ & $4.13e+01$ & $4.04e+01$ & $4.01e+01$ & & $2.40e+02$ & $2.00e+02$ & $2.03e+02$ & $2.03e+02$ \\
$10^{-3}$    & & $4.32e+01$ & $4.13e+01$ & $4.04e+01$ & $4.01e+01$ & & $2.40e+02$ & $2.00e+02$ & $2.03e+02$ & $2.03e+02$ \\
$10^{-2}$    & & $4.32e+01$ & $4.13e+01$ & $4.04e+01$ & $4.01e+01$ & & $2.40e+02$ & $2.00e+02$ & $2.03e+02$ & $2.03e+02$  \\
$10^{-1}$    & & $4.32e+01$ & $4.13e+01$ & $4.04e+01$ & $4.01e+01$ & & $2.40e+02$ & $2.00e+02$ & $2.03e+02$ & $2.03e+02$ \\
$1$          & & $4.29e+01$ & $4.13e+01$ & $4.04e+01$ & $4.01e+01$ & & $2.40e+02$ & $2.00e+02$ & $2.03e+02$ & $2.03e+02$ \\
$10^{1}$     & & $4.27e+01$ & $4.11e+01$ & $4.03e+01$ & $4.01e+01$ & & $2.42e+02$ & $2.00e+02$ & $2.03e+02$ & $2.03e+02$ \\
$10^{2}$     & & $4.56e+01$ & $4.20e+01$ & $4.05e+01$ & $4.00e+01$ & & $2.58e+02$ & $2.07e+02$ & $2.07e+02$ & $2.04e+02$ \\
$10^{3}$     & & $5.99e+01$ & $5.36e+01$ & $4.75e+01$ & $4.28e+01$ & & $3.08e+02$ & $2.35e+02$ & $2.35e+02$ & $2.20e+02$ \\
$10^{4}$     & & $1.38e+02$ & $8.25e+01$ & $8.15e+01$ & $8.02e+01$ & & $4.95e+02$ & $3.05e+02$ & $3.05e+02$ & $3.33e+02$
\end{tabular}
\end{subtable}
\end{adjustbox}
\caption{$2$-$d$ preconditioned $\bm{H}_0(div, \Omega)$: condition number $\kappa_2 \left( \bm{B}_{div} \bm{A}_{div} \right)$ in the case of a Gauss-Seidel smoothing.}  
  \label{tab:div-cn-Gauss-Seidel}  
\end{table}


\subsubsection{Test 2: convergence study of the ASP in the $2$-$d$ setting with Jacobi and Gauss-Seidel smoothing}\label{sec:jacobi}
The smoother is provided by Jacobi  and symmetric Gauss-Seidel relaxation schemes. We recall that this allows for an explicit form of the matrix related to the smother, more precisely we have
$$
\bm{S}_\mathcal{D}^{-1} = \bm{D}_{\bm{A}_\mathcal{D}}^{-1},
$$
in the case of Jacobi smoothing, while  when using Gauss-Seidel smoothing $\bm{D}_{\bm{A}_\mathcal{D}}^{-1}$ is replaced by
$$
\bm{S}_\mathcal{D}^{-1} = \bm{L}_{\bm{A}_\mathcal{D}}^{-1} - \bm{L}_{\bm{A}_\mathcal{D}}^{-1} \bm{A}_\mathcal{D} \bm{U}_{\bm{A}_\mathcal{D}}^{-1} + \bm{U}_{\bm{A}_\mathcal{D}}^{-1},
$$


\begin{table}[H]

\begin{adjustbox}{width=1.\textwidth}
\begin{subtable}{1.\textwidth} 
    \centering 
    \hspace*{-2. cm}
\begin{tabular}{c p{0.cm} p{.5cm}p{.5cm} p{0.cm} p{.5cm}p{.5cm} p{0.cm} p{.5cm}p{.5cm} p{0.cm} p{.5cm}p{.5cm}p{0.cm} p{.5cm}p{.5cm} p{0.cm} p{.5cm}p{.5cm} p{0.cm} p{.5cm}p{.5cm} p{0.cm} p{.5cm}p{.5cm}}
\cline{2-25}
\multirow{2}{*}{} &                      & \multicolumn{11}{l}{$p=1$}                                                                                        &  & \multicolumn{11}{l}{$p=2$}   \\   
\cline{3-13} \cline{15-25} 
                 \multirow{2}{*}{\diagbox[innerwidth=0.8cm]{\hspace*{0.15cm}$\tau$}{$n$}}   & \multicolumn{1}{c}{} & \multicolumn{2}{l}{$8$} &  & \multicolumn{2}{l}{$16$} &  & \multicolumn{2}{l}{$32$} &  & \multicolumn{2}{l}{$64$} &  & \multicolumn{2}{l}{$8$} &  & \multicolumn{2}{l}{$16$} &  & \multicolumn{2}{l}{$32$} &  & \multicolumn{2}{l}{$64$} \\ \cline{3-4} \cline{6-7} \cline{9-10} \cline{12-13} \cline{15-16} \cline{18-19} \cline{21-22} \cline{24-25} 
           &                      & J         & GS           &  & J          & GS           &  & J          & GS           &  & J          & GS           &  & J         & GS           &  & J          & GS           &  & J          & GS           &  & J          & GS           \\ \hline

$10^{-4}$ &&     $13$ & $11$      &&      $17$ & $13$       &&        $20$ & $16$        &&       $23$ & $17$         
          &&     $11$ & $10$      &&      $13$ & $12$       &&        $17$ & $14$        &&       $20$ & $16$           \\
$10^{-3}$ &&     $13$ & $11$      &&      $17$ & $13$       &&        $20$ & $16$        &&       $23$ & $17$           
          &&     $11$ & $10$      &&      $13$ & $12$       &&        $17$ & $14$        &&       $20$ & $16$           \\
$10^{-2}$ &&     $13$ & $11$      &&      $17$ & $13$       &&        $20$ & $16$        &&       $23$ & $17$           
          &&     $11$ & $10$      &&      $14$ & $12$       &&        $17$ & $14$        &&       $20$ & $16$           \\
$10^{-1}$ &&     $13$ & $11$      &&      $17$ & $13$       &&        $20$ & $16$        &&       $23$ & $17$          
          &&     $11$ & $10$      &&      $14$ & $12$       &&        $17$ & $14$        &&       $20$ & $16$           \\
$1$       &&     $13$ & $10$      &&      $17$ & $13$       &&        $20$ & $16$        &&       $23$ & $17$           
          &&     $12$ & $10$      &&      $14$ & $12$       &&        $17$ & $13$        &&       $20$ & $16$           \\
$10$      &&     $12$ & $10$      &&      $16$ & $12$       &&        $20$ & $15$        &&       $23$ & $17$           
          &&     $11$ & $9$       &&      $13$ & $10$       &&        $17$ & $13$        &&       $20$ & $16$           \\
$10^2$    &&     $9$  & $10$      &&      $13$ & $9$        &&        $16$ & $11$        &&       $19$ & $13$          
          &&     $8$  & $9$       &&      $10$ & $9$        &&        $14$ & $9$         &&       $16$ & $11$           \\
$10^3$    &&     $8$  & $9$       &&      $8$  & $10$       &&        $10$ & $9$         &&       $13$ & $9$           
          &&     $14$ & $10$      &&      $12$ & $9$        &&        $8$  & $9$         &&       $11$ & $9$            \\
$10^4$    &&     $9$  & $8$       &&      $10$ & $9$        &&        $9$  & $10$        &&       $8$  & $9$  
          &&     $15$ & $11$      &&      $17$ & $10$       &&        $16$ & $10$        &&       $10$ & $9$           
\end{tabular}
\end{subtable}
\end{adjustbox}

\vspace{.5cm}

\begin{adjustbox}{width=1.\textwidth}
\begin{subtable}{1.\textwidth} 
    \centering 
    \hspace*{-2. cm}
\begin{tabular}{c p{0.cm} p{.5cm}p{.5cm} p{0.cm} p{.5cm}p{.5cm} p{0.cm} p{.5cm}p{.5cm} p{0.cm} p{.5cm}p{.5cm}p{0.cm} p{.5cm}p{.5cm} p{0.cm} p{.5cm}p{.5cm} p{0.cm} p{.5cm}p{.5cm} p{0.cm} p{.5cm}p{.5cm}}
\cline{2-25}
\multirow{2}{*}{} &                      & \multicolumn{11}{l}{$p=3$}                                                                                        &  & \multicolumn{11}{l}{$p=4$}   \\   
\cline{3-13} \cline{15-25} 
                 \multirow{2}{*}{\diagbox[innerwidth=0.8cm]{\hspace*{0.15cm}$\tau$}{$n$}}   & \multicolumn{1}{c}{} & \multicolumn{2}{l}{$8$} &  & \multicolumn{2}{l}{$16$} &  & \multicolumn{2}{l}{$32$} &  & \multicolumn{2}{l}{$64$} &  & \multicolumn{2}{l}{$8$} &  & \multicolumn{2}{l}{$16$} &  & \multicolumn{2}{l}{$32$} &  & \multicolumn{2}{l}{$64$} \\ \cline{3-4} \cline{6-7} \cline{9-10} \cline{12-13} \cline{15-16} \cline{18-19} \cline{21-22} \cline{24-25} 
           &                      & J         & GS           &  & J          & GS           &  & J          & GS           &  & J          & GS           &  & J         & GS           &  & J          & GS           &  & J          & GS           &  & J          & GS           \\ \hline

$10^{-4}$ &&     $12$ & $10$      &&      $15$ & $12$       &&        $17$ & $13$        &&       $21$ & $16$         
          &&     $15$ & $12$      &&      $19$ & $13$       &&        $22$ & $13$        &&       $25$ & $15$           \\
$10^{-3}$ &&     $13$ & $10$      &&      $15$ & $12$       &&        $17$ & $13$        &&       $21$ & $16$           
          &&     $15$ & $12$      &&      $19$ & $13$       &&        $22$ & $13$        &&       $25$ & $15$           \\
$10^{-2}$ &&     $13$ & $10$      &&      $15$ & $12$       &&        $18$ & $13$        &&       $21$ & $16$           
          &&     $15$ & $12$      &&      $19$ & $13$       &&        $23$ & $13$        &&       $25$ & $15$           \\
$10^{-1}$ &&     $13$ & $10$      &&      $15$ & $12$       &&        $18$ & $13$        &&       $21$ & $16$          
          &&     $15$ & $12$      &&      $19$ & $13$       &&        $23$ & $13$        &&       $25$ & $15$           \\
$1$       &&     $13$ & $10$      &&      $16$ & $11$       &&        $18$ & $12$        &&       $21$ & $16$           
          &&     $15$ & $12$      &&      $20$ & $13$       &&        $23$ & $13$        &&       $25$ & $15$           \\
$10$      &&     $13$ & $9$       &&      $17$ & $10$       &&        $19$ & $12$        &&       $21$ & $15$           
          &&     $16$ & $11$      &&      $21$ & $11$       &&        $26$ & $12$        &&       $28$ & $14$           \\
$10^2$    &&     $13$ & $10$      &&      $14$ & $10$       &&        $15$ & $10$        &&       $17$ & $10$          
          &&     $15$ & $11$      &&      $17$ & $11$       &&        $19$ & $11$        &&       $21$ & $10$           \\
$10^3$    &&     $17$ & $11$      &&      $16$ & $10$       &&        $11$ & $9$         &&       $11$ & $10$           
          &&     $19$ & $14$      &&      $20$ & $12$       &&        $13$ & $11$        &&       $11$ & $10$           \\
$10^4$    &&     $19$ & $14$      &&      $23$ & $13$       &&        $27$ & $14$        &&       $12$ & $9$  
          &&     $20$ & $15$      &&      $25$ & $15$       &&        $32$ & $16$        &&       $13$ & $9$     
\end{tabular}
\end{subtable}
\end{adjustbox}

\vspace{.5cm}

\begin{adjustbox}{width=1.\textwidth}
\begin{subtable}{1.\textwidth} 
    \centering 
    \hspace*{-2. cm}
\begin{tabular}{c p{0.cm} p{.5cm}p{.5cm} p{0.cm} p{.5cm}p{.5cm} p{0.cm} p{.5cm}p{.5cm} p{0.cm} p{.5cm}p{.5cm}p{0.cm} p{.5cm}p{.5cm} p{0.cm} p{.5cm}p{.5cm} p{0.cm} p{.5cm}p{.5cm} p{0.cm} p{.5cm}p{.5cm}}
\cline{2-25}
\multirow{2}{*}{} &                      & \multicolumn{11}{l}{$p=5$}                                                                                        &  & \multicolumn{11}{l}{$p=6$}   \\   
\cline{3-13} \cline{15-25} 
                 \multirow{2}{*}{\diagbox[innerwidth=0.8cm]{\hspace*{0.15cm}$\tau$}{$n$}}   & \multicolumn{1}{c}{} & \multicolumn{2}{l}{$8$} &  & \multicolumn{2}{l}{$16$} &  & \multicolumn{2}{l}{$32$} &  & \multicolumn{2}{l}{$64$} &  & \multicolumn{2}{l}{$8$} &  & \multicolumn{2}{l}{$16$} &  & \multicolumn{2}{l}{$32$} &  & \multicolumn{2}{l}{$64$} \\ \cline{3-4} \cline{6-7} \cline{9-10} \cline{12-13} \cline{15-16} \cline{18-19} \cline{21-22} \cline{24-25} 
           &                      & J         & GS           &  & J          & GS           &  & J          & GS           &  & J          & GS           &  & J         & GS           &  & J          & GS           &  & J          & GS           &  & J          & GS           \\ \hline

$10^{-4}$ &&     $18$ & $15$      &&      $21$ & $16$       &&        $27$ & $17$        &&       $29$ & $16$         
          &&     $19$ & $15$      &&      $24$ & $17$       &&        $30$ & $19$        &&       $33$ & $18$           \\
$10^{-3}$ &&     $18$ & $15$      &&      $21$ & $16$       &&        $27$ & $17$        &&       $29$ & $16$           
          &&     $19$ & $15$      &&      $24$ & $17$       &&        $29$ & $19$        &&       $33$ & $18$           \\
$10^{-2}$ &&     $18$ & $14$      &&      $21$ & $16$       &&        $27$ & $17$        &&       $29$ & $16$           
          &&     $19$ & $15$      &&      $24$ & $17$       &&        $29$ & $19$        &&       $33$ & $18$           \\
$10^{-1}$ &&     $18$ & $14$      &&      $21$ & $16$       &&        $27$ & $17$        &&       $29$ & $16$          
          &&     $19$ & $15$      &&      $24$ & $17$       &&        $29$ & $21$        &&       $35$ & $18$           \\
$1$       &&     $19$ & $14$      &&      $21$ & $16$       &&        $27$ & $17$        &&       $30$ & $16$           
          &&     $20$ & $15$      &&      $23$ & $17$       &&        $29$ & $20$        &&       $44$ & $18$           \\
$10$      &&     $19$ & $13$      &&      $24$ & $14$       &&        $31$ & $16$        &&       $36$ & $16$           
          &&     $21$ & $13$      &&      $25$ & $14$       &&        $33$ & $17$        &&       $28$ & $17$           \\
$10^2$    &&     $17$ & $12$      &&      $19$ & $13$       &&        $24$ & $13$        &&       $26$ & $12$          
          &&     $18$ & $14$      &&      $21$ & $14$       &&        $26$ & $14$        &&       $28$ & $13$           \\
$10^3$    &&     $21$ & $15$      &&      $22$ & $12$       &&        $13$ & $11$        &&       $13$ & $12$           
          &&     $21$ & $16$      &&      $22$ & $13$       &&        $15$ & $13$        &&       $15$ & $12$           \\
$10^4$    &&     $21$ & $14$      &&      $25$ & $14$       &&        $36$ & $20$        &&       $15$ & $10$  
          &&     $22$ & $17$      &&      $28$ & $15$       &&        $36$ & $19$        &&       $15$ & $11$        
\end{tabular}
\end{subtable}
\end{adjustbox}
\caption{$2$-$d$ preconditioned $\bm{H}_0(div, \Omega)$: \CG iterations  in the case of Jacobi (J) and the Gauss-Seidel (GS) smoothing. Exact solution is given by \eqref{eq:curl-exact-sol-1}.}  
  \label{tab:div-cgn}     
\end{table}


where $\bm{D}_{\bm{A}_\mathcal{D}}$, $\bm{L}_{\bm{A}_\mathcal{D}}$ and $\bm{U}_{\bm{A}_\mathcal{D}}$ stand for the diagonal, the lower and the upper  parts of the matrix $\bm{A}_\mathcal{D}$, respectively. 

Following the approach of the previous subsection, we compute the spectral condition number $\kappa_2 \left( \bm{B}_{\mathcal{D}} \bm{A}_{\mathcal{D}} \right)$ and the number of conjugate gradient iterations required for the preconditioned system to converge, while varying the $B$-spline degree $p$, the number of elements $n$, and the parameter $\tau$. Tables \ref{tab:curl-cn-jacobi}--\ref{tab:curl-cgn} show the results for the $\bm{curl}$ problem, while tables \ref{tab:div-cn-jacobi}--\ref{tab:div-cgn} show the results for the $div$ problem.
\begin{table}[H] 
\begin{adjustbox}{width=\textwidth}
  \begin{subtable}{0.6\textwidth} 
    \subcaption{$\bm{H}_0(\bm{curl}, \Omega)$ problem.}
    \centering
    \begin{tabular}{|c|c|c|c|}
      \hline
      $\tau$ & \CG Iter & Res. Error & $l^2$ Error\\ 
      \hline
      \hline
      $10^{-7}$ 	& $20$     & $1.38e-06$ & $1.52e-06$    \\
 	  $10^{-6}$ 	& $20$     & $1.38e-06$ & $4.43e-07$    \\
 	  $10^{-5}$ 	& $20$     & $1.38e-06$ & $4.25e-07$    \\
      $10^{-4}$ 	& $20$     & $1.38e-06$ & $4.24e-07$    \\
      $10^{-3}$ 	& $20$     & $1.38e-06$ & $4.23e-07$    \\
      $10^{-2}$ 	& $20$     & $1.37e-06$ & $4.13e-07$    \\
      $10^{-1}$ 	& $20$     & $1.30e-06$ & $3.32e-07$    \\
      $1$     		& $20$     & $1.12e-06$ & $1.52e-07$   	\\   
      \hline
    \end{tabular}
  \end{subtable}%
  \begin{subtable}{0.6\textwidth} 
    \subcaption{$\bm{H}_0(div, \Omega)$ problem.}
    \centering
    \begin{tabular}{|c|c|c|c|c|}
      \hline
      $n$ & \CG Iter & Res. Error & $l^2$ Error\\ 
      \hline
      \hline
      $10^{-7}$ 	& $20$     & $1.38e-06$ & $1.58e-06$    \\
 	  $10^{-6}$ 	& $20$     & $1.38e-06$ & $4.47e-07$    \\
 	  $10^{-5}$ 	& $20$     & $1.38e-06$ & $4.25e-07$    \\
      $10^{-4}$ 	& $20$     & $1.38e-06$ & $4.24e-07$    \\
      $10^{-3}$ 	& $20$     & $1.38e-06$ & $4.23e-07$    \\
      $10^{-2}$ 	& $20$     & $1.37e-06$ & $4.13e-07$    \\
      $10^{-1}$ 	& $20$     & $1.32e-06$ & $3.37e-07$    \\
      $1$     		& $20$     & $1.24e-06$ & $1.71e-07$   	\\   
      \hline
    \end{tabular}
  \end{subtable}
  \end{adjustbox}
  \caption{$2$-$d$ ASP reconditioning with Jacobi smoothing: \CG iterations, residual and $l^2$ approximation errors. Exact solutions are given by \eqref{eq:curl-exact-sol-2}--\eqref{eq:div-exact-sol-2}. Parameter values are set to $n=32$, $p = 3$.}
  \label{tab:cgn-res-err-J}
\end{table}
\begin{table}[H] 
\begin{adjustbox}{width=\textwidth}
  \begin{subtable}{0.6\textwidth} 
    \subcaption{$\bm{H}_0(\bm{curl}, \Omega)$ problem.}
    \centering
    \begin{tabular}{|c|c|c|c|}
      \hline
      $p$ & \CG Iter & Res. Error & $l^2$ Error\\ 
      \hline
   	  \hline
      $10^{-7}$ 	& $14$     & $4.09e-07$ & $1.39e-06$    \\
 	  $10^{-6}$ 	& $14$     & $4.09e-07$ & $3.61e-07$    \\
 	  $10^{-5}$ 	& $14$     & $4.09e-07$ & $3.32e-07$    \\
      $10^{-4}$ 	& $14$     & $4.09e-07$ & $3.31e-07$    \\
      $10^{-3}$ 	& $14$     & $4.09e-07$ & $3.30e-07$    \\
      $10^{-2}$ 	& $14$     & $4.09e-07$ & $3.12e-07$    \\
      $10^{-1}$ 	& $14$     & $3.95e-07$ & $2.05e-07$    \\
      $1$     		& $13$     & $1.04e-06$ & $1.33e-07$   	\\   
      \hline
    \end{tabular}
  \end{subtable}%
  \begin{subtable}{0.6\textwidth} 
    \subcaption{$\bm{H}_0(div, \Omega)$ problem.}
    \centering
    \begin{tabular}{|c|c|c|c|c|}
      \hline
      $n$ & \CG Iter & Res. Error & $l^2$ Error\\ 
      \hline
      \hline
      $10^{-7}$ 	& $14$     & $9.20e-07$ & $1.51e-06$    \\
 	  $10^{-6}$ 	& $14$     & $9.20e-07$ & $3.95e-07$    \\
 	  $10^{-5}$ 	& $14$     & $9.20e-07$ & $3.74e-07$    \\
      $10^{-4}$ 	& $14$     & $9.20e-07$ & $3.74e-07$    \\
      $10^{-3}$ 	& $14$     & $9.20e-07$ & $3.73e-07$    \\
      $10^{-2}$ 	& $14$     & $9.14e-07$ & $3.63e-07$    \\
      $10^{-1}$ 	& $14$     & $8.79e-07$ & $2.86e-07$    \\
      $1$     		& $13$     & $1.09e-06$ & $1.50e-07$   	\\   
      \hline
    \end{tabular}
  \end{subtable}
  \end{adjustbox}
  \caption{$2$-$d$ ASP reconditioning with Gauss-Seidel smoothing: \CG iterations, residual and $l^2$ approximation errors. Exact solutions are given by \eqref{eq:curl-exact-sol-2}--\eqref{eq:div-exact-sol-2}. Parameter values are set to $n=32$, $p = 3$.}
  \label{tab:cgn-res-err-GS}
\end{table}
\begin{table}[H]
\begin{adjustbox}{width=1.\textwidth}
\begin{subtable}{0.4\textwidth} 
    \centering
\begin{tabular}{|c|c|c|c|}
\hline
 $p$   & NP & J   & GS     \\
\hline
\hline
 $1$   & $253$ 	& $23 \, (8)$       & $16 \, (7)$                        \\
 $2$   & $1120$ & $20 \, (6)$     	& $16 \, (6)$                         \\
 $3$   & $1896$ & $21 \, (6)$     	& $15 \, (5)$                         \\
 $4$   & $-$ 	& $25 \, (6)$     	& $15 \, (5)$                      \\
 $5$   & $-$ 	& $29 \, (6)$     	& $15 \, (5)$                       \\
 $6$   & $-$ 	& $33 \, (6)$     	& $17 \, (5)$                      \\
 $7$   & $-$ 	& $35 \, (6)$     	& $20 \, (4)$                      \\
 $8$   & $-$ 	& $37 \, (6)$     	& $22 \, (4)$                   \\
 $9$   & $-$ 	& $36 \, (6)$     	& $24 \, (4)$                      \\
 $10$  & $-$ 	& $37 \, (6)$     	& $23 \, (4)$                    \\
\hline
\end{tabular}
\subcaption{$\bm{H}_0(\bm{curl}, \Omega)$ problem.}
\end{subtable}%
\begin{subtable}{0.4\textwidth} 
    \centering
\begin{tabular}{|c|c|c|c|c|c|}
\hline
 $p$   & NP & J   & GS    \\
\hline
\hline
 $1$   & $245$  & $23 \, (8)$     & $17 \, (7)$                      \\
 $2$   & $1120$ & $20 \, (6)$     & $16 \, (6)$                      \\
 $3$   & $1595$ & $21 \, (6)$     & $16 \, (5)$                       \\
 $4$   & $-$ 	& $25 \, (6)$     & $15 \, (5)$                    \\
 $5$   & $-$ 	& $29 \, (6)$     & $16 \, (5)$                     \\
 $6$   & $-$ 	& $33 \, (6)$     & $18 \, (5)$                    \\
 $7$   & $-$ 	& $35 \, (6)$     & $22 \, (4)$                     \\
 $8$   & $-$ 	& $37 \, (6)$     & $24 \, (4)$                     \\
 $9$   & $-$ 	& $36 \, (6)$     & $24 \, (4)$                      \\
 $10$  & $-$ 	& $37 \, (6)$     & $24 \, (4)$                      \\
\hline
\end{tabular}
\subcaption{$\bm{H}_0(div, \Omega)$ problem.}
\end{subtable}
\end{adjustbox}
\caption{The $2$-$d$ problem: \CG iterations in the case of the  unpreconditioned problem (NP), ASP preconditioning with Jacobi smoothing (J), ASP preconditioning with Gauss-Seidel smoothing (GS), and the optimal ASP algorithm with GLT smoothing (ASP-GLT). The number of iterations required for ASP-GLT is indicated in parentheses. The exact solutions are defined by \eqref{eq:curl-exact-sol-1}-\eqref{eq:div-exact-sol-1}. Parameter values are set to $\tau=10^{-4}$, $n=64$, $\nu_1=1$, $\nu_2=p^2$, and $\nu_{asp}=3$. ($-$) indicates that \CG reached the maximum number of iterations (set to $3000$) without achieving convergence.}
\label{tab:cgn-glt}  
\end{table}
\begin{table}[H]
\begin{adjustbox}{width=1.\textwidth}
\begin{subtable}{1.\textwidth} 
    \centering 
    \hspace*{-2. cm}
\begin{tabular}{c p{0.cm} p{1.cm} p{1.7cm} p{1.7cm} p{1.7cm}  c p{1.7cm} p{1.7cm} p{1.7cm} p{1.7cm}}
\hline
 &  & \multicolumn{4}{l}{$p=1$} &  & \multicolumn{4}{l}{$p=2$}  \\ \cline{3-6} \cline{8-11}
\diaghead{taun-----}{$\tau$}{$n$} &  & $8$   & $16$   & $32$   & $64$    & & $8$   & $16$   & $32$   & $64$ \\ \hline
$10^{-4}$    & & $151$  & $328$   & $511$    & $879$   
             & & $520$  & $975$   & $1313$   & $1962$ \\
$10^{-3}$    & & $127$  & $271$   & $439$    & $749$   
             & & $452$  & $856$   & $1092$   & $1550$ \\
$10^{-2}$    & & $114$  & $192$   & $350$    & $610$   
             & & $366$  & $679$   & $892$    & $1324$ \\
$10^{-1}$    & & $81$   & $174$   & $276$    & $486$   
             & & $279$  & $533$   & $703$    & $1029$\\
$1$          & & $59$   & $109$   & $188$    & $295$   
             & & $188$  & $367$   & $518$    & $717$\\
$10^{1}$     & & $37$   & $70$    & $133$    & $244$   
             & & $112$  & $226$   & $308$    & $432$\\
$10^{2}$     & & $21$   & $38$    & $72$     & $131$   
             & & $48$   & $78$    & $122$    & $182$\\
$10^{3}$     & & $11$   & $12$    & $21$     & $41$    
             & & $44$   & $41$    & $39$     & $52$\\
$10^{4}$     & & $11$   & $11$    & $7$      & $10$    
             & & $47$   & $52$    & $41$     & $30$     
\end{tabular}
\end{subtable}
\end{adjustbox}

\vspace{.5cm}

\begin{adjustbox}{width=1.\textwidth}
\begin{subtable}{1.\textwidth} 
    \centering 
    \hspace*{-2. cm}
\begin{tabular}{c p{0.cm} p{1.cm} p{1.7cm} p{1.7cm} p{1.7cm}  c p{1.7cm} p{1.7cm} p{1.7cm} p{1.7cm}}
\hline
 &  & \multicolumn{4}{l}{$p=3$} &  & \multicolumn{4}{l}{$p=4$}  \\ \cline{3-6} \cline{8-11}
\diaghead{taun-----}{$\tau$}{$n$} &  & $8$   & $16$   & $32$   & $64$    & & $8$   & $16$   & $32$   & $64$ \\ \hline

$10^{-4}$    & & $-$    & $-$     & $-$      & $-$ 
			 & & $-$    & $-$     & $-$      & $-$ \\
$10^{-3}$    & & $-$    & $-$     & $-$      & $-$ 
             & & $-$    & $-$     & $-$      & $-$ \\
$10^{-2}$    & & $2301$ & $-$     & $-$      & $-$ 
             & & $-$    & $-$     & $-$      & $-$ \\
$10^{-1}$    & & $1579$ & $2763$  & $-$      & $-$ 
             & & $-$    & $-$     & $-$      & $-$ \\
$1$          & & $991$  & $1570$  & $1840$   & $2277$
             & & $1471$ & $-$     & $-$      & $-$ \\
$10^{1}$     & & $303$  & $534$   & $764$    & $1149$ 
             & & $510$  & $867$   & $1463$   & $2091$ \\
$10^{2}$     & & $113$  & $150$   & $219$    & $362$ 
             & & $154$  & $277$   & $432$    & $548$ \\
$10^{3}$     & & $88$   & $60$    & $64$     & $102$ 
             & & $116$  & $110$   & $122$    & $172$\\
$10^{4}$     & & $129$  & $109$   & $61$     & $39$ 
             & & $195$  & $206$   & $128$    & $74$     
\end{tabular}
\end{subtable}
\end{adjustbox}

\vspace{.5cm}

\begin{adjustbox}{width=1.\textwidth}
\begin{subtable}{1.\textwidth} 
    \centering 
    \hspace*{-2. cm}
\begin{tabular}{c p{0.cm} p{1.cm} p{1.7cm} p{1.7cm} p{1.7cm}  c p{1.7cm} p{1.7cm} p{1.7cm} p{1.7cm}}
\hline
 &  & \multicolumn{4}{l}{$p=5$} &  & \multicolumn{4}{l}{$p=6$}  \\ \cline{3-6} \cline{8-11}
\diaghead{taun-----}{$\tau$}{$n$} &  & $8$   & $16$   & $32$   & $64$    & & $8$   & $16$   & $32$   & $64$ \\ \hline

$10^{-4}$    & & $-$  	 & $-$   	& $-$   	& $-$   
             & & $-$  	 & $-$   	& $-$   	& $-$ \\
$10^{-3}$    & & $-$  	 & $-$   	& $-$   	& $-$   
             & & $-$  	 & $-$   	& $-$   	& $-$ \\
$10^{-2}$    & & $-$  	 & $-$   	& $-$   	& $-$   
             & & $-$  	 & $-$   	& $-$   	& $-$ \\
$10^{-1}$    & & $-$  	 & $-$   	& $-$   	& $-$   
             & & $-$  	 & $-$   	& $-$   	& $-$ \\
$1$          & & $-$     & $-$   	& $-$   	& $-$   
             & & $-$     & $-$   	& $-$   	& $-$ \\
$10^{1}$     & & $1187$  & $1692$   & $2200$   	& $-$   
             & & $1325$  & $2252$   & $2838$   	& $-$ \\
$10^{2}$     & & $359$   & $414$    & $625$     & $786$   
             & & $384$   & $562$    & $771$     & $1013$ \\
$10^{3}$     & & $300$   & $206$    & $184$     & $228$    
             & & $327$   & $363$    & $267$     & $348$ \\
$10^{4}$     & & $542$   & $431$    & $260$     & $139$    
             & & $625$   & $764$    & $403$     & $201$    
\end{tabular}
\end{subtable}
\end{adjustbox}
\caption{$3$-$d$ unpreconditioned $\bm{H}_0(\bm{curl}, \Omega)$: \CG iterations. The right-hand side function is defined by \eqref{eq:3d-curl-right-hand-side}. '$-$' indicates that \CG reached the maximum number of iterations (set to $3000$) without convergence.}  
  \label{tab:3d-curl-cgn-none}  
\end{table}

To test the algorithm's convergence to exact solutions, we consider analytic solutions \eqref{eq:curl-exact-sol-2}--\eqref{eq:div-exact-sol-2}, as we did for the unpreconditioned problem. We track the number of iterations, residual error, and $l^2$ relative error after the \CG method converges for different choices of $\tau$ with fixed values of $n=32$ and $p=3$. We present the results in Tables \ref{tab:cgn-res-err-J}--\ref{tab:cgn-res-err-GS}.

\begin{table}[H]

\begin{adjustbox}{width=1.\textwidth}
\begin{subtable}{1.\textwidth} 
    \centering 
    \hspace*{-1.5 cm}
\begin{tabular}{c p{0.cm} p{.4cm}p{.4cm} p{0.cm} p{.4cm}p{.4cm} p{0.cm} p{.4cm}p{.4cm} p{0.cm} p{.4cm}p{.4cm}p{0.cm} p{.4cm}p{.4cm} p{0.cm} p{.4cm}p{.4cm} p{0.cm} p{.4cm}p{.4cm} p{0.cm} p{.4cm}p{.4cm}c p{0.cm} cc p{0.cm} cc p{0.cm} cc p{0.cm} ccp{0.cm} cc p{0.cm} cc p{0.cm} cc p{0.cm} cc}
\cline{2-25}
\multirow{2}{*}{} &                      & \multicolumn{11}{l}{$p=1$}                                                                                        &  & \multicolumn{11}{l}{$p=2$}   \\   
\cline{3-13} \cline{15-25} 
                 \multirow{2}{*}{\diagbox[innerwidth=0.8cm]{\hspace*{0.15cm}$\tau$}{$n$}}   & \multicolumn{1}{c}{} & \multicolumn{2}{l}{$8$} &  & \multicolumn{2}{l}{$16$} &  & \multicolumn{2}{l}{$32$} &  & \multicolumn{2}{l}{$64$} &  & \multicolumn{2}{l}{$8$} &  & \multicolumn{2}{l}{$16$} &  & \multicolumn{2}{l}{$32$} &  & \multicolumn{2}{l}{$64$} \\ \cline{3-4} \cline{6-7} \cline{9-10} \cline{12-13} \cline{15-16} \cline{18-19} \cline{21-22} \cline{24-25} 
           &                      & J         & GS           &  & J          & GS           &  & J          & GS           &  & J          & GS           &  & J         & GS           &  & J          & GS           &  & J          & GS           &  & J          & GS           \\ \hline

$10^{-4}$ &&    $4$ & $3$     &&     $6$ & $5$      &&       $7$ & $6$       &&      $8$ & $7$        
          &&    $2$ & $2$     &&     $4$ & $3$      &&       $5$ & $4$       &&      $6$ & $6$          \\
$10^{-3}$ &&    $4$ & $3$     &&     $6$ & $5$      &&       $7$ & $6$       &&      $8$ & $7$          
          &&    $2$ & $2$     &&     $4$ & $3$      &&       $5$ & $4$       &&      $6$ & $6$          \\
$10^{-2}$ &&    $4$ & $3$     &&     $6$ & $5$      &&       $7$ & $6$       &&      $8$ & $7$          
          &&    $2$ & $2$     &&     $4$ & $3$      &&       $5$ & $4$       &&      $6$ & $5$          \\
$10^{-1}$ &&    $4$ & $3$     &&     $6$ & $5$      &&       $7$ & $6$       &&      $8$ & $7$         
          &&    $2$ & $2$     &&     $3$ & $3$      &&       $4$ & $4$       &&      $5$ & $5$          \\
$1$       &&    $4$ & $3$     &&     $6$ & $5$      &&       $7$ & $6$       &&      $8$ & $7$          
          &&    $2$ & $2$     &&     $3$ & $3$      &&       $4$ & $4$       &&      $6$ & $5$          \\
$10$      &&    $4$ & $3$     &&     $5$ & $4$      &&       $7$ & $6$       &&      $7$ & $7$          
          &&    $2$ & $2$     &&     $3$ & $2$      &&       $4$ & $3$       &&      $5$ & $5$          \\
$10^2$    &&    $3$ & $2$     &&     $5$ & $3$      &&       $6$ & $5$       &&      $7$ & $6$         
          &&    $2$ & $1$     &&     $2$ & $2$      &&       $3$ & $3$       &&      $5$ & $4$          \\
$10^3$    &&    $2$ & $1$     &&     $2$ & $1$      &&       $5$ & $3$       &&      $7$ & $4$          
          &&    $1$ & $1$     &&     $2$ & $2$      &&       $2$ & $2$       &&      $2$ & $2$          \\
$10^4$    &&    $1$ & $1$     &&     $2$ & $1$      &&       $2$ & $1$       &&      $3$ & $2$ 
          &&    $1$ & $1$     &&     $1$ & $1$      &&       $1$ & $1$       &&      $2$ & $1$         
\end{tabular}
\end{subtable}
\end{adjustbox}

\vspace{.5cm}

\begin{adjustbox}{width=1.\textwidth}
\begin{subtable}{1.\textwidth} 
    \centering 
    \hspace*{-1.5 cm}
\begin{tabular}{c p{0.cm} p{.4cm}p{.4cm} p{0.cm} p{.4cm}p{.4cm} p{0.cm} p{.4cm}p{.4cm} p{0.cm} p{.4cm}p{.4cm}p{0.cm} p{.4cm}p{.4cm} p{0.cm} p{.4cm}p{.4cm} p{0.cm} p{.4cm}p{.4cm} p{0.cm} p{.4cm}p{.4cm}c p{0.cm} cc p{0.cm} cc p{0.cm} cc p{0.cm} ccp{0.cm} cc p{0.cm} cc p{0.cm} cc p{0.cm} cc}
\cline{2-25}
\multirow{2}{*}{} &                      & \multicolumn{11}{l}{$p=3$}                                                                                        &  & \multicolumn{11}{l}{$p=4$}   \\   
\cline{3-13} \cline{15-25} 
                 \multirow{2}{*}{\diagbox[innerwidth=0.8cm]{\hspace*{0.15cm}$\tau$}{$n$}}   & \multicolumn{1}{c}{} & \multicolumn{2}{l}{$8$} &  & \multicolumn{2}{l}{$16$} &  & \multicolumn{2}{l}{$32$} &  & \multicolumn{2}{l}{$64$} &  & \multicolumn{2}{l}{$8$} &  & \multicolumn{2}{l}{$16$} &  & \multicolumn{2}{l}{$32$} &  & \multicolumn{2}{l}{$64$} \\ \cline{3-4} \cline{6-7} \cline{9-10} \cline{12-13} \cline{15-16} \cline{18-19} \cline{21-22} \cline{24-25} 
           &                      & J         & GS           &  & J          & GS           &  & J          & GS           &  & J          & GS           &  & J         & GS           &  & J          & GS           &  & J          & GS           &  & J          & GS           \\ \hline

$10^{-4}$ &&    $2$ & $2$     &&     $3$ & $2$      &&       $4$ & $3$       &&      $5$ & $5$        
          &&    $3$ & $2$     &&     $3$ & $2$      &&       $3$ & $3$       &&      $5$ & $4$          \\
$10^{-3}$ &&    $2$ & $2$     &&     $3$ & $2$      &&       $4$ & $3$       &&      $5$ & $5$          
          &&    $3$ & $2$     &&     $3$ & $2$      &&       $3$ & $3$       &&      $4$ & $4$           \\
$10^{-2}$ &&    $2$ & $2$     &&     $3$ & $2$      &&       $4$ & $3$       &&      $5$ & $5$          
          &&    $3$ & $2$     &&     $3$ & $2$      &&       $3$ & $3$       &&      $4$ & $4$           \\
$10^{-1}$ &&    $2$ & $2$     &&     $3$ & $2$      &&       $4$ & $3$       &&      $5$ & $4$         
          &&    $3$ & $2$     &&     $3$ & $2$      &&       $3$ & $3$       &&      $5$ & $4$           \\
$1$       &&    $2$ & $2$     &&     $3$ & $2$      &&       $4$ & $4$       &&      $6$ & $5$          
          &&    $2$ & $2$     &&     $3$ & $2$      &&       $3$ & $3$       &&      $5$ & $4$           \\
$10$      &&    $2$ & $2$     &&     $2$ & $2$      &&       $3$ & $3$       &&      $5$ & $5$          
          &&    $2$ & $2$     &&     $2$ & $2$      &&       $3$ & $3$       &&      $4$ & $3$           \\
$10^2$    &&    $2$ & $1$     &&     $2$ & $2$      &&       $2$ & $2$       &&      $3$ & $3$         
          &&    $2$ & $2$     &&     $2$ & $2$      &&       $2$ & $2$       &&      $3$ & $3$           \\
$10^3$    &&    $1$ & $1$     &&     $2$ & $2$      &&       $2$ & $2$       &&      $2$ & $2$          
          &&    $1$ & $2$     &&     $2$ & $2$      &&       $2$ & $2$       &&      $2$ & $2$           \\
$10^4$    &&    $1$ & $1$     &&     $1$ & $1$      &&       $1$ & $1$       &&      $2$ & $2$ 
          &&    $1$ & $1$     &&     $1$ & $1$      &&       $2$ & $2$       &&      $2$ & $2$     
\end{tabular}
\end{subtable}
\end{adjustbox}

\vspace{.5cm}

\begin{adjustbox}{width=1.\textwidth}
\begin{subtable}{1.\textwidth} 
    \centering 
    \hspace*{-1.5 cm}
\begin{tabular}{c p{0.cm} p{.4cm}p{.4cm} p{0.cm} p{.4cm}p{.4cm} p{0.cm} p{.4cm}p{.4cm} p{0.cm} p{.4cm}p{.4cm}p{0.cm} p{.4cm}p{.4cm} p{0.cm} p{.4cm}p{.4cm} p{0.cm} p{.4cm}p{.4cm} p{0.cm} p{.4cm}p{.4cm}c p{0.cm} cc p{0.cm} cc p{0.cm} cc p{0.cm} ccp{0.cm} cc p{0.cm} cc p{0.cm} cc p{0.cm} cc}
\cline{2-25}
\multirow{2}{*}{} &                      & \multicolumn{11}{l}{$p=5$}                                                                                        &  & \multicolumn{11}{l}{$p=6$}   \\   
\cline{3-13} \cline{15-25} 
                 \multirow{2}{*}{\diagbox[innerwidth=0.8cm]{\hspace*{0.15cm}$\tau$}{$n$}}   & \multicolumn{1}{c}{} & \multicolumn{2}{l}{$8$} &  & \multicolumn{2}{l}{$16$} &  & \multicolumn{2}{l}{$32$} &  & \multicolumn{2}{l}{$64$} &  & \multicolumn{2}{l}{$8$} &  & \multicolumn{2}{l}{$16$} &  & \multicolumn{2}{l}{$32$} &  & \multicolumn{2}{l}{$64$} \\ \cline{3-4} \cline{6-7} \cline{9-10} \cline{12-13} \cline{15-16} \cline{18-19} \cline{21-22} \cline{24-25} 
           &                      & J         & GS           &  & J          & GS           &  & J          & GS           &  & J          & GS           &  & J         & GS           &  & J          & GS           &  & J          & GS           &  & J          & GS           \\ \hline

$10^{-4}$ &&    $3$ & $2$     &&     $3$ & $2$      &&       $3$ & $3$       &&      $5$ & $4$        
          &&    $4$ & $2$     &&     $5$ & $2$      &&       $4$ & $3$       &&      $5$ & $4$           \\
$10^{-3}$ &&    $4$ & $2$     &&     $3$ & $2$      &&       $3$ & $3$       &&      $5$ & $4$          
          &&    $4$ & $2$     &&     $4$ & $2$      &&       $4$ & $3$       &&      $5$ & $4$           \\
$10^{-2}$ &&    $4$ & $2$     &&     $4$ & $2$      &&       $3$ & $3$       &&      $5$ & $4$          
          &&    $4$ & $2$     &&     $4$ & $3$      &&       $4$ & $3$       &&      $5$ & $4$         \\
$10^{-1}$ &&    $3$ & $2$     &&     $4$ & $2$      &&       $4$ & $3$       &&      $5$ & $4$         
          &&    $4$ & $2$     &&     $4$ & $3$      &&       $4$ & $3$       &&      $5$ & $4$           \\
$1$       &&    $3$ & $2$     &&     $3$ & $2$      &&       $4$ & $3$       &&      $4$ & $4$          
          &&    $3$ & $2$     &&     $4$ & $3$      &&       $4$ & $3$       &&      $5$ & $4$           \\
$10$      &&    $2$ & $2$     &&     $3$ & $2$      &&       $3$ & $3$       &&      $4$ & $3$          
          &&    $3$ & $2$     &&     $3$ & $2$      &&       $4$ & $3$       &&      $5$ & $4$           \\
$10^2$    &&    $2$ & $2$     &&     $2$ & $2$      &&       $3$ & $2$       &&      $3$ & $3$         
          &&    $2$ & $2$     &&     $3$ & $2$      &&       $3$ & $2$       &&      $4$ & $3$           \\
$10^3$    &&    $2$ & $2$     &&     $2$ & $2$      &&       $2$ & $2$       &&      $3$ & $2$          
          &&    $2$ & $2$     &&     $2$ & $2$      &&       $3$ & $2$       &&      $3$ & $2$           \\
$10^4$    &&    $1$ & $1$     &&     $1$ & $2$      &&       $2$ & $2$       &&      $2$ & $2$ 
          &&    $1$ & $1$     &&     $1$ & $2$      &&       $2$ & $2$       &&      $3$ & $2$       
\end{tabular}
\end{subtable}
\end{adjustbox}
\caption{$3$-$d$ ASP-GLT preconditioning for $\bm{H}_0(\bm{curl}, \Omega)$: \CG iterations with Jacobi (J) and Gauss-Seidel (GS) smoothing. The right-hand side function is defined by \eqref{eq:3d-curl-right-hand-side}. Parameter values are set to $\nu_1=1$, $\nu_2=p^3$, and $\nu_{asp}=3$.}  
  \label{tab:3d-curl-cgn-glt}  
\end{table}

Comparing the above results to those of the previous subsection, several  observations can be made. 
\begin{itemize}
\item[-] the spectral condition numbers and the number of conjugate gradient iterations required for convergence are relatively small and not heavily dependent on the mesh parameter $h$.

\bigskip

\item[-] both the number of \CG iterations and the spectral condition numbers appear to be independent of $\tau$, indicating that the ASP methodology can handle small values of $\tau$ effectively

\bigskip

\item[-] the relative errors are now sufficiently small and are of the same order as the corresponding residual errors, demonstrating that the solution obtained converges to the corresponding exact solution.
\bigskip

\item[-] it can be observed that the overall performance obtained with Gauss-Seidel smoothing is slightly better than that obtained with Jacobi smoothing scheme.
\end{itemize}

\begin{table}[H]

\begin{adjustbox}{width=1.\textwidth}
\begin{subtable}{1.\textwidth} 
    \centering 
    \hspace*{-2. cm}
\begin{tabular}{c p{0.cm} p{.5cm}p{.5cm} p{0.cm} p{.5cm}p{.5cm} p{0.cm} p{.5cm}p{.5cm} p{0.cm} p{.5cm}p{.5cm}p{0.cm} p{.5cm}p{.5cm} p{0.cm} p{.5cm}p{.5cm} p{0.cm} p{.5cm}p{.5cm} p{0.cm} p{.5cm}p{.5cm}}
\cline{2-25}
\multirow{2}{*}{} &                      & \multicolumn{11}{l}{$p=1$}                                                                                        &  & \multicolumn{11}{l}{$p=2$}   \\   
\cline{3-13} \cline{15-25} 
                 \multirow{2}{*}{\diagbox[innerwidth=0.8cm]{\hspace*{0.15cm}$\tau$}{$n$}}   & \multicolumn{1}{c}{} & \multicolumn{2}{l}{$8$} &  & \multicolumn{2}{l}{$16$} &  & \multicolumn{2}{l}{$32$} &  & \multicolumn{2}{l}{$64$} &  & \multicolumn{2}{l}{$8$} &  & \multicolumn{2}{l}{$16$} &  & \multicolumn{2}{l}{$32$} &  & \multicolumn{2}{l}{$64$} \\ \cline{3-4} \cline{6-7} \cline{9-10} \cline{12-13} \cline{15-16} \cline{18-19} \cline{21-22} \cline{24-25} 
           &                      & J         & GS           &  & J          & GS           &  & J          & GS           &  & J          & GS           &  & J         & GS           &  & J          & GS           &  & J          & GS           &  & J          & GS           \\ \hline

$10^{-4}$ &&     $22$ & $12$      &&      $26$ & $17$       &&        $28$ & $19$        &&       $28$ & $22$         
          &&     $10$ & $13$      &&      $16$ & $13$       &&        $24$ & $20$        &&       $27$ & $23$           \\
$10^{-3}$ &&     $21$ & $12$      &&      $25$ & $18$       &&        $28$ & $19$        &&       $29$ & $23$           
          &&     $11$ & $13$      &&      $18$ & $16$       &&        $26$ & $21$        &&       $29$ & $23$           \\
$10^{-2}$ &&     $21$ & $13$      &&      $25$ & $18$       &&        $29$ & $20$        &&       $30$ & $22$           
          &&     $12$ & $14$      &&      $20$ & $18$       &&        $27$ & $20$        &&       $31$ & $24$           \\
$10^{-1}$ &&     $20$ & $13$      &&      $25$ & $19$       &&        $28$ & $22$        &&       $29$ & $23$          
          &&     $14$ & $15$      &&      $20$ & $20$       &&        $28$ & $22$        &&       $29$ & $25$           \\
$1$       &&     $20$ & $14$      &&      $25$ & $19$       &&        $26$ & $23$        &&       $28$ & $24$           
          &&     $14$ & $15$      &&      $20$ & $21$       &&        $27$ & $24$        &&       $31$ & $27$           \\
$10$      &&     $16$ & $11$      &&      $20$ & $15$       &&        $24$ & $19$        &&       $26$ & $23$           
          &&     $4$  & $5$       &&      $13$ & $14$       &&        $22$ & $21$        &&       $27$ & $25$           \\
$10^2$    &&     $6$  & $4$       &&      $13$ & $7$        &&        $18$ & $12$        &&       $21$ & $17$          
          &&     $2$  & $2$       &&      $3$  & $2$        &&        $7$  & $7$         &&       $15$ & $15$           \\
$10^3$    &&     $2$  & $1$       &&      $3$  & $2$        &&        $8$  & $5$         &&       $14$ & $8$           
          &&     $1$  & $1$       &&      $1$  & $1$        &&        $2$  & $2$         &&       $3$  & $3$           \\
$10^4$    &&     $1$  & $1$       &&      $1$  & $1$        &&        $2$  & $1$         &&       $4$  & $2$  
          &&     $1$  & $1$       &&      $1$  & $1$        &&        $1$  & $1$         &&       $1$  & $1$          
\end{tabular}
\end{subtable}
\end{adjustbox}

\vspace{.5cm}

\begin{adjustbox}{width=1.\textwidth}
\begin{subtable}{1.\textwidth} 
    \centering 
    \hspace*{-2. cm}
\begin{tabular}{c p{0.cm} p{.5cm}p{.5cm} p{0.cm} p{.5cm}p{.5cm} p{0.cm} p{.5cm}p{.5cm} p{0.cm} p{.5cm}p{.5cm}p{0.cm} p{.5cm}p{.5cm} p{0.cm} p{.5cm}p{.5cm} p{0.cm} p{.5cm}p{.5cm} p{0.cm} p{.5cm}p{.5cm}}
\cline{2-25}
\multirow{2}{*}{} &                      & \multicolumn{11}{l}{$p=3$}                                                                                        &  & \multicolumn{11}{l}{$p=4$}   \\   
\cline{3-13} \cline{15-25} 
                 \multirow{2}{*}{\diagbox[innerwidth=0.8cm]{\hspace*{0.15cm}$\tau$}{$n$}}   & \multicolumn{1}{c}{} & \multicolumn{2}{l}{$8$} &  & \multicolumn{2}{l}{$16$} &  & \multicolumn{2}{l}{$32$} &  & \multicolumn{2}{l}{$64$} &  & \multicolumn{2}{l}{$8$} &  & \multicolumn{2}{l}{$16$} &  & \multicolumn{2}{l}{$32$} &  & \multicolumn{2}{l}{$64$} \\ \cline{3-4} \cline{6-7} \cline{9-10} \cline{12-13} \cline{15-16} \cline{18-19} \cline{21-22} \cline{24-25} 
           &                      & J         & GS           &  & J          & GS           &  & J          & GS           &  & J          & GS           &  & J         & GS           &  & J          & GS           &  & J          & GS           &  & J          & GS           \\ \hline

$10^{-4}$ &&     $16$ & $21$      &&      $25$ & $24$       &&        $31$ & $31$        &&       $43$ & $36$         
          &&     $26$ & $41$      &&      $38$ & $34$       &&        $44$ & $40$        &&       $70$ & $55$           \\
$10^{-3}$ &&     $18$ & $23$      &&      $22$ & $28$       &&        $33$ & $29$        &&       $47$ & $40$           
          &&     $31$ & $42$      &&      $42$ & $40$       &&        $51$ & $44$        &&       $75$ & $60$            \\
$10^{-2}$ &&     $23$ & $24$      &&      $23$ & $23$       &&        $31$ & $33$        &&       $43$ & $43$           
          &&     $48$ & $48$      &&      $33$ & $49$       &&        $54$ & $51$        &&       $72$ & $67$            \\
$10^{-1}$ &&     $21$ & $32$      &&      $30$ & $28$       &&        $35$ & $37$        &&       $47$ & $40$          
          &&     $4$  & $12$      &&      $41$ & $61$       &&        $48$ & $57$        &&       $71$ & $67$            \\
$1$       &&     $5$  & $15$      &&      $26$ & $28$       &&        $35$ & $37$        &&       $43$ & $44$           
          &&     $2$  & $3$       &&      $7$  & $12$       &&        $46$ & $57$        &&       $79$ & $65$            \\
$10$      &&     $2$  & $2$       &&      $3$  & $3$        &&        $9$  & $10$        &&       $35$ & $34$           
          &&     $2$  & $2$       &&      $3$  & $3$        &&        $4$  & $4$         &&       $26$ & $14$            \\
$10^2$    &&     $2$  & $2$       &&      $2$  & $2$        &&        $2$  & $2$         &&       $4$  & $3$          
          &&     $2$  & $2$       &&      $2$  & $2$        &&        $2$  & $2$         &&       $3$  & $3$            \\
$10^3$    &&     $1$  & $1$       &&      $1$  & $1$        &&        $2$  & $2$         &&       $2$  & $2$           
          &&     $1$  & $1$       &&      $2$  & $2$        &&        $2$  & $2$         &&       $2$  & $2$            \\
$10^4$    &&     $1$  & $1$       &&      $1$  & $1$        &&        $1$  & $1$         &&       $1$  & $1$  
          &&     $1$  & $1$       &&      $1$  & $1$        &&        $1$  & $1$         &&       $2$  & $1$      
\end{tabular}
\end{subtable}
\end{adjustbox}

\vspace{.5cm}

\begin{adjustbox}{width=1.\textwidth}
\begin{subtable}{1.\textwidth} 
    \centering 
    \hspace*{-2. cm}
\begin{tabular}{c p{0.cm} p{.5cm}p{.5cm} p{0.cm} p{.5cm}p{.5cm} p{0.cm} p{.5cm}p{.5cm} p{0.cm} p{.5cm}p{.5cm}p{0.cm} p{.5cm}p{.5cm} p{0.cm} p{.5cm}p{.5cm} p{0.cm} p{.5cm}p{.5cm} p{0.cm} p{.5cm}p{.5cm}}
\cline{2-25}
\multirow{2}{*}{} &                      & \multicolumn{11}{l}{$p=5$}                                                                                        &  & \multicolumn{11}{l}{$p=6$}   \\   
\cline{3-13} \cline{15-25} 
                 \multirow{2}{*}{\diagbox[innerwidth=0.8cm]{\hspace*{0.15cm}$\tau$}{$n$}}   & \multicolumn{1}{c}{} & \multicolumn{2}{l}{$8$} &  & \multicolumn{2}{l}{$16$} &  & \multicolumn{2}{l}{$32$} &  & \multicolumn{2}{l}{$64$} &  & \multicolumn{2}{l}{$8$} &  & \multicolumn{2}{l}{$16$} &  & \multicolumn{2}{l}{$32$} &  & \multicolumn{2}{l}{$64$} \\ \cline{3-4} \cline{6-7} \cline{9-10} \cline{12-13} \cline{15-16} \cline{18-19} \cline{21-22} \cline{24-25} 
           &                      & J         & GS           &  & J          & GS           &  & J          & GS           &  & J          & GS           &  & J         & GS           &  & J          & GS           &  & J          & GS           &  & J          & GS           \\ \hline

$10^{-4}$ &&     $40$  & $50$      &&      $55$  & $65$       &&        $79$  & $72$        &&       $97$  & $72$        
          &&     $99$  & $63$      &&      $90$  & $72$       &&        $96$  & $82$        &&       $146$ & $116$           \\
$10^{-3}$ &&     $68$  & $49$      &&      $72$  & $68$       &&        $88$  & $83$        &&       $103$ & $81$           
          &&     $101$ & $96$      &&      $103$ & $73$       &&        $107$ & $88$        &&       $153$ & $125$            \\
$10^{-2}$ &&     $59$  & $77$      &&      $65$  & $67$       &&        $90$  & $90$        &&       $116$ & $92$           
          &&     $164$ & $99$      &&      $113$ & $121$      &&        $127$ & $106$       &&       $147$ & $127$          \\
$10^{-1}$ &&     $5$   & $5$       &&      $13$  & $22$       &&        $79$  & $84$        &&       $97$  & $105$          
          &&     $7$   & $6$       &&      $11$  & $10$       &&        $123$ & $141$       &&       $152$ & $129$            \\
$1$       &&     $4$   & $4$       &&      $5$   & $5$        &&        $16$  & $26$        &&       $112$ & $88$           
          &&     $4$   & $4$       &&      $6$   & $5$        &&        $10$  & $12$        &&       $367$ & $126$            \\
$10$      &&     $2$   & $3$       &&      $3$   & $3$        &&        $3$   & $4$         &&       $9$   & $7$           
          &&     $3$   & $3$       &&      $4$   & $4$        &&        $5$   & $5$         &&       $9$   & $7$            \\
$10^2$    &&     $2$   & $2$       &&      $2$   & $2$        &&        $3$   & $3$         &&       $3$   & $3$          
          &&     $2$   & $2$       &&      $2$   & $2$        &&        $3$   & $3$         &&       $4$   & $3$            \\
$10^3$    &&     $2$   & $1$       &&      $2$   & $2$        &&        $2$   & $2$         &&       $2$   & $2$           
          &&     $2$   & $2$       &&      $2$   & $2$        &&        $2$   & $2$         &&       $2$   & $3$            \\
$10^4$    &&     $1$   & $1$       &&      $1$   & $1$        &&        $1$   & $1$         &&       $2$   & $2$  
          &&     $1$   & $1$       &&      $1$   & $1$        &&        $2$   & $1$         &&       $2$   & $2$        
\end{tabular}
\end{subtable}
\end{adjustbox}
\caption{$3$-$d$ ASP-GLT preconditioning for $\bm{H}_0(div, \Omega)$: \CG iterations with Jacobi (J) and Gauss-Seidel (GS) smoothing. The right-hand side function is defined by \eqref{eq:3d-div-right-hand-side}. Parameter values are set to $\nu_1=1$, $\nu_2=p^3$, and $\nu_{asp}=3$.}
\label{tab:3d-div-cgn-glt}  
\end{table}

In conclusion, Tables \ref{tab:curl-cn-jacobi}--\ref{tab:cgn-res-err-GS} present results that compare favorably with those of Subsection \ref{sec:none}, indicating that the preconditioning method outlined in Section \ref{sec:asp} has been an effective solution to the ill-preconditioning of the approximate problem. However, it is important to note that the numerical results deteriorate when $p$ is large, specially in the case of $\bm{H}_0(div, \Omega)$ problem. Although this $p$-dependency is not addressed in the theoretical results presented in this paper, in the next subsection we propose a numerical investigation.

\begin{table}[H]

\begin{adjustbox}{width=1.\textwidth}
\begin{subtable}{1.\textwidth} 
    \centering 
    \hspace*{-2. cm}
\begin{tabular}{c p{0.cm} p{.5cm}p{.5cm} p{0.cm} p{.5cm}p{.5cm} p{0.cm} p{.5cm}p{.5cm} p{0.cm} p{.5cm}p{.5cm}p{0.cm} p{.5cm}p{.5cm} p{0.cm} p{.5cm}p{.5cm} p{0.cm} p{.5cm}p{.5cm} p{0.cm} p{.5cm}p{.5cm}}
\cline{2-25}
\multirow{2}{*}{} &                      & \multicolumn{11}{l}{$p=1$}                                                                                        &  & \multicolumn{11}{l}{$p=2$}   \\   
\cline{3-13} \cline{15-25} 
                 \multirow{2}{*}{\diagbox[innerwidth=0.8cm]{\hspace*{0.15cm}$\tau$}{$n$}}   & \multicolumn{1}{c}{} & \multicolumn{2}{l}{$8$} &  & \multicolumn{2}{l}{$16$} &  & \multicolumn{2}{l}{$32$} &  & \multicolumn{2}{l}{$64$} &  & \multicolumn{2}{l}{$8$} &  & \multicolumn{2}{l}{$16$} &  & \multicolumn{2}{l}{$32$} &  & \multicolumn{2}{l}{$64$} \\ \cline{3-4} \cline{6-7} \cline{9-10} \cline{12-13} \cline{15-16} \cline{18-19} \cline{21-22} \cline{24-25} 
           &                      & J         & GS           &  & J          & GS           &  & J          & GS           &  & J          & GS           &  & J         & GS           &  & J          & GS           &  & J          & GS           &  & J          & GS           \\ \hline

$10^{-4}$ &&     $15$ & $7$      &&      $18$ & $9$        &&        $19$ & $11$        &&       $21$ & $12$         
          &&     $4$  & $5$      &&      $7$  & $6$        &&        $10$ & $8$         &&       $12$ & $9$           \\
$10^{-3}$ &&     $15$ & $7$      &&      $18$ & $10$       &&        $20$ & $11$        &&       $21$ & $12$           
          &&     $4$  & $5$      &&      $7$  & $6$        &&        $10$ & $8$         &&       $11$ & $10$           \\
$10^{-2}$ &&     $15$ & $7$      &&      $17$ & $10$       &&        $18$ & $11$        &&       $26$ & $13$           
          &&     $5$  & $5$      &&      $8$  & $7$        &&        $11$ & $9$         &&       $12$ & $10$           \\
$10^{-1}$ &&     $14$ & $7$      &&      $17$ & $10$       &&        $19$ & $11$        &&       $13$ & $13$          
          &&     $5$  & $5$      &&      $8$  & $7$        &&        $11$ & $9$         &&       $10$ & $10$           \\
$1$       &&     $14$ & $8$      &&      $16$ & $10$       &&        $19$ & $12$        &&       $24$ & $13$           
          &&     $6$  & $5$      &&      $8$  & $7$        &&        $11$ & $9$         &&       $12$ & $10$           \\
$10$      &&     $11$ & $7$      &&      $18$ & $8$        &&        $19$ & $10$        &&       $20$ & $12$           
          &&     $3$  & $3$      &&      $6$  & $5$        &&        $9$  & $8$         &&       $11$ & $9$           \\
$10^2$    &&     $5$  & $3$      &&      $10$ & $5$        &&        $15$ & $7$         &&       $15$ & $9$          
          &&     $2$  & $1$      &&      $2$  & $2$        &&        $4$  & $4$         &&       $7$  & $6$           \\
$10^3$    &&     $2$  & $1$      &&      $3$  & $2$        &&        $6$  & $4$         &&       $11$ & $6$           
          &&     $1$  & $1$      &&      $1$  & $1$        &&        $2$  & $1$         &&       $3$  & $2$           \\
$10^4$    &&     $1$  & $1$      &&      $1$  & $1$        &&        $2$  & $1$         &&       $3$  & $2$  
          &&     $1$  & $1$      &&      $1$  & $1$        &&        $1$  & $1$         &&       $1$  & $1$          
\end{tabular}
\end{subtable}
\end{adjustbox}

\vspace{.5cm}

\begin{adjustbox}{width=1.\textwidth}
\begin{subtable}{1.\textwidth} 
    \centering 
    \hspace*{-2. cm}
\begin{tabular}{c p{0.cm} p{.5cm}p{.5cm} p{0.cm} p{.5cm}p{.5cm} p{0.cm} p{.5cm}p{.5cm} p{0.cm} p{.5cm}p{.5cm}p{0.cm} p{.5cm}p{.5cm} p{0.cm} p{.5cm}p{.5cm} p{0.cm} p{.5cm}p{.5cm} p{0.cm} p{.5cm}p{.5cm}}
\cline{2-25}
\multirow{2}{*}{} &                      & \multicolumn{11}{l}{$p=3$}                                                                                        &  & \multicolumn{11}{l}{$p=4$}   \\   
\cline{3-13} \cline{15-25} 
                 \multirow{2}{*}{\diagbox[innerwidth=0.8cm]{\hspace*{0.15cm}$\tau$}{$n$}}   & \multicolumn{1}{c}{} & \multicolumn{2}{l}{$8$} &  & \multicolumn{2}{l}{$16$} &  & \multicolumn{2}{l}{$32$} &  & \multicolumn{2}{l}{$64$} &  & \multicolumn{2}{l}{$8$} &  & \multicolumn{2}{l}{$16$} &  & \multicolumn{2}{l}{$32$} &  & \multicolumn{2}{l}{$64$} \\ \cline{3-4} \cline{6-7} \cline{9-10} \cline{12-13} \cline{15-16} \cline{18-19} \cline{21-22} \cline{24-25} 
           &                      & J         & GS           &  & J          & GS           &  & J          & GS           &  & J          & GS           &  & J         & GS           &  & J          & GS           &  & J          & GS           &  & J          & GS           \\ \hline

$10^{-4}$ &&     $4$ & $4$      &&      $5$ & $4$       &&        $6$ & $6$        &&       $9$  & $8$         
          &&     $4$ & $3$      &&      $3$ & $4$       &&        $4$ & $5$        &&       $7$  & $6$           \\
$10^{-3}$ &&     $4$ & $5$      &&      $5$ & $4$       &&        $6$ & $6$        &&       $9$  & $8$           
          &&     $5$ & $3$      &&      $3$ & $4$       &&        $5$ & $5$        &&       $7$  & $6$            \\
$10^{-2}$ &&     $5$ & $4$      &&      $5$ & $4$       &&        $6$ & $7$        &&       $9$  & $9$           
          &&     $5$ & $5$      &&      $4$ & $5$       &&        $5$ & $5$        &&       $7$  & $6$            \\
$10^{-1}$ &&     $5$ & $5$      &&      $5$ & $5$       &&        $7$ & $7$        &&       $8$  & $8$          
          &&     $3$ & $3$      &&      $5$ & $6$       &&        $6$ & $5$        &&       $7$  & $7$            \\
$1$       &&     $3$ & $3$      &&      $6$ & $5$       &&        $7$ & $7$        &&       $10$ & $9$           
          &&     $2$ & $2$      &&      $4$ & $4$       &&        $6$ & $5$        &&       $7$  & $7$            \\
$10$      &&     $2$ & $2$      &&      $3$ & $2$       &&        $5$ & $5$        &&       $7$  & $7$           
          &&     $2$ & $2$      &&      $2$ & $2$       &&        $3$ & $3$        &&       $5$  & $5$            \\
$10^2$    &&     $2$ & $1$      &&      $2$ & $2$       &&        $2$ & $2$        &&       $3$  & $3$          
          &&     $2$ & $2$      &&      $2$ & $2$       &&        $2$ & $2$        &&       $2$  & $2$            \\
$10^3$    &&     $1$ & $1$      &&      $1$ & $1$       &&        $2$ & $1$        &&       $2$  & $2$           
          &&     $1$ & $1$      &&      $1$ & $1$       &&        $2$ & $1$        &&       $2$  & $2$            \\
$10^4$    &&     $1$ & $1$      &&      $1$ & $1$       &&        $1$ & $1$        &&       $1$  & $1$  
          &&     $1$ & $1$      &&      $1$ & $1$       &&        $1$ & $1$        &&       $1$  & $1$      
\end{tabular}
\end{subtable}
\end{adjustbox}

\vspace{.5cm}

\begin{adjustbox}{width=1.\textwidth}
\begin{subtable}{1.\textwidth} 
    \centering 
    \hspace*{-2. cm}
\begin{tabular}{c p{0.cm} p{.5cm}p{.5cm} p{0.cm} p{.5cm}p{.5cm} p{0.cm} p{.5cm}p{.5cm} p{0.cm} p{.5cm}p{.5cm}p{0.cm} p{.5cm}p{.5cm} p{0.cm} p{.5cm}p{.5cm} p{0.cm} p{.5cm}p{.5cm} p{0.cm} p{.5cm}p{.5cm}}
\cline{2-25}
\multirow{2}{*}{} &                      & \multicolumn{11}{l}{$p=5$}                                                                                        &  & \multicolumn{11}{l}{$p=6$}   \\   
\cline{3-13} \cline{15-25} 
                 \multirow{2}{*}{\diagbox[innerwidth=0.8cm]{\hspace*{0.15cm}$\tau$}{$n$}}   & \multicolumn{1}{c}{} & \multicolumn{2}{l}{$8$} &  & \multicolumn{2}{l}{$16$} &  & \multicolumn{2}{l}{$32$} &  & \multicolumn{2}{l}{$64$} &  & \multicolumn{2}{l}{$8$} &  & \multicolumn{2}{l}{$16$} &  & \multicolumn{2}{l}{$32$} &  & \multicolumn{2}{l}{$64$} \\ \cline{3-4} \cline{6-7} \cline{9-10} \cline{12-13} \cline{15-16} \cline{18-19} \cline{21-22} \cline{24-25} 
           &                      & J         & GS           &  & J          & GS           &  & J          & GS           &  & J          & GS           &  & J         & GS           &  & J          & GS           &  & J          & GS           &  & J          & GS           \\ \hline

$10^{-4}$ &&     $4$ & $4$      &&      $3$ & $5$       &&        $4$ & $5$        &&       $5$ & $6$         
          &&     $4$ & $4$      &&      $4$ & $4$       &&        $4$ & $4$        &&       $5$ & $5$            \\
$10^{-3}$ &&     $5$ & $4$      &&      $4$ & $5$       &&        $4$ & $5$        &&       $6$ & $6$           
          &&     $5$ & $4$      &&      $4$ & $5$       &&        $4$ & $4$        &&       $6$ & $6$            \\
$10^{-2}$ &&     $5$ & $4$      &&      $5$ & $5$       &&        $4$ & $5$        &&       $6$ & $6$           
          &&     $5$ & $4$      &&      $5$ & $5$       &&        $5$ & $5$        &&       $6$ & $5$          \\
$10^{-1}$ &&     $3$ & $2$      &&      $5$ & $5$       &&        $4$ & $5$        &&       $6$ & $6$          
          &&     $5$ & $3$      &&      $5$ & $5$       &&        $5$ & $5$        &&       $6$ & $6$            \\
$1$       &&     $2$ & $2$      &&      $3$ & $3$       &&        $4$ & $4$        &&       $6$ & $6$           
          &&     $2$ & $2$      &&      $3$ & $3$       &&        $4$ & $5$        &&       $6$ & $6$            \\
$10$      &&     $2$ & $2$      &&      $2$ & $2$       &&        $3$ & $3$        &&       $4$ & $3$           
          &&     $2$ & $2$      &&      $2$ & $2$       &&        $3$ & $3$        &&       $4$ & $3$            \\
$10^2$    &&     $2$ & $2$      &&      $2$ & $2$       &&        $2$ & $2$        &&       $2$ & $2$          
          &&     $2$ & $2$      &&      $2$ & $2$       &&        $2$ & $2$        &&       $2$ & $2$            \\
$10^3$    &&     $1$ & $1$      &&      $2$ & $1$       &&        $2$ & $2$        &&       $2$ & $2$           
          &&     $1$ & $1$      &&      $2$ & $1$       &&        $2$ & $2$        &&       $2$ & $2$            \\
$10^4$    &&     $1$ & $1$      &&      $1$ & $1$       &&        $1$ & $1$        &&       $2$ & $1$  
          &&     $1$ & $1$      &&      $1$ & $1$       &&        $1$ & $1$        &&       $2$ & $1$        
\end{tabular}
\end{subtable}
\end{adjustbox}
\caption{$3$-$d$ ASP-GLT preconditioning for $\bm{H}_0(div, \Omega)$: \CG iterations with Jacobi (J) and Gauss-Seidel (GS) smoothing. The ASP preconditioner is defined by \eqref{eq:asp-div}, with $\bm{D}_{\bm{curl}}^{-1}$ replaced by \eqref{eq:div-second-smoother-GS}. The right-hand side function is defined by \eqref{eq:3d-div-right-hand-side}. Parameter values are set to $\nu_1=1$, $\nu_2=p^3$, and $\nu_{asp}=3$.}
\label{tab:3d-div-cgn-GS-glt}  
\end{table}
\subsubsection{Test 3: ASP and p-dependency}\label{sec:glt}
This section introduces a modified version of our Auxiliary Space Preconditioning method that addresses the issue related to the dependency with respect to $B$-Splines degree. Specifically, an additional smoother is applied to control the $p$-dependency of the preconditioner. The construction of the smoother is based on the theory of Generalized Locally Toeplitz (GLT) and utilizes the spectral information of the  involved matrices,  as discussed in \cite{mazza2019isogeometric}. For this purpose, we decompose the ASP preconditioner as follows:
$$
\bm{B}_{\mathcal{D}} = \bm{S}^{-1}_{\mathcal{D}} + \bm{K}_{\mathcal{D}}.
$$  

The suggested algorithm is as follows: 

\begin{minipage}{\textwidth}
\begin{algorithm}[H]
\DontPrintSemicolon
\SetAlgoLined
\SetKwInOut{Input}{Input}\SetKwInOut{Output}{Output}
\Input{$\bm{A}$: the matrix related to the IgA discretization of \eqref{introduction-variational-formulation}; $\bm{b}$: a given vector; $\bm{x}$: a starting point; $\nu_1$: the number of Jacobi (J) or Gauss-Seidel (GS) iterations; $\nu_2$: the number of GLT iterations; $\nu_{ASP}$: the number of ASP iterations.}
\Output{$\bm{x}$: the approximate solution of $\bm{A}_{\mathcal{D}} \bm{x} = \bm{b}$.}
\bigskip

$k \gets 0$;

\While{$k \leq \nu_{ASP}$ \textbf{and} not converged}{
$\bm{x} \gets \mbox{\texttt{smoother}}_1(\bm{A}_{\mathcal{D}},\bm{b}, \bm{x},\nu_1)$ \tcp*{Apply J or GS smoother.}
$\bm{x} \gets \mbox{\texttt{smoother}}_2(\bm{A}_{\mathcal{D}},\bm{b},\bm{x}, \nu_2)$ \tcp*{Apply GLT smoother.}
$\bm{d} \gets \bm{b} - \bm{A}_{\mathcal{D}} \bm{x}$ \tcp*{Compute the defect.}
$\bm{x}_c \gets \bm{K}_{\mathcal{D}} \bm{d}$ \tcp*{Compute the ASP correction.}
$\bm{x} \gets \bm{x} + \bm{x}_c$ \tcp*{Update the solution.}
$k \gets k + 1$;
}
\caption{{\sc ASP-GLT}: Preconditioning for $\bm{V}_h(\mathcal{D},\Omega)$.}
\label{algo:asp-opt}
\end{algorithm}
\end{minipage}

In the simplest case, we can select the inverse of the mass matrix as the GLT smoother, and we use this approach in the numerical tests developed in this subsection. 

To investigate the impact of $p$-refinement on the convergence of Algorithm \ref{algo:asp-opt}, we report in Table \ref{tab:cgn-glt} the number of \CG iterations as a function of the $B$-spline degree, with fixed parameter values $n=64$, $\tau = 10^{-4}$, $\nu_1=1$, and $\nu_{asp}=3$. In the GLT smoother step, we employ the MINRES solver with $\nu_2=p^2$ iterations; the numerical value of $\nu_2$ is motivated by analytical results for the Poisson equation \cite{hofreither2015mass}. We consider four cases: unpreconditioned problem, ASP preconditioning with Jacobi smoothing, ASP preconditioning with Gauss-Seidel smoothing, and the optimal ASP algorithm (ASP-GLT). As expected, the number of iterations in the ASP-GLT case appears to be well-behaved with respect to $p$, meaning that it remains bounded as $p$ increases. In contrast, the other cases show a much stronger dependence on $p$, leading to a higher number of iterations as $p$ increases. In the case of the unpreconditioned problem, this increase is particularly dramatic.

\subsection{Three dimensional tests.} We will now consider the three-dimensional case, where our computational domain is a unit cube $\Omega = (0,1)$ that has been subdivided into $n \times n \times n$ sub-domains. To test our algorithm's effectiveness, we will be using the optimal Algorithm \ref{algo:asp-opt}. Similar to the two-dimensional case, we will be using the Conjugate Gradient (CG) method to solve the IgA discrete system related to \eqref{introduction-variational-formulation}. Both un-preconditioned and preconditioned systems will be tested, with the stopping criteria given by \eqref{eq:stopping-criterion} and the initial guess set to the zero vector. We will employ the MINRES solver with $\nu_2=p^3$ iterations in the GLT smoother step.

For clarity, we will consider the $\bm{curl}$ and $div$ problems separately.

\subsubsection{Test 4: the $3$-$d$ $\bm{H}_0(\bm{curl}, \Omega)$ problem} In this test we consider problem \eqref{eq:curl-curl} subjects to Dirichlet boundary conditions with a right-hand side function given by 
\begin{equation}\label{eq:3d-curl-right-hand-side}
\bm{f}(x_1,x_2,x_3) = (x_1, x_2, x_3), \quad (x_1,x_2,x_3) \in (0,1)^3.
\end{equation}
Table \ref{tab:3d-curl-cgn-none} shows the number of CG iterations for the unpreconditioned system, which allows for comparing the performance of the ASP-GLT algorithm for $3$-$d$ $\bm{curl}-\bm{curl}$ problems. In contrast, Table \ref{tab:3d-curl-cgn-glt} provides the results for the preconditioned system with both Jacobi and Gauss-Seidel smoothing, and for varying values of $\tau$, $n$, and $p$.

By examining the results presented in these tables, one can evaluate the efficiency and effectiveness of the ASP-GLT algorithm for solving $3$-$d$ $\bm{curl}-\bm{curl}$ problems. Specifically, the preconditioned system leads to a significantly lower number of iterations compared to the unpreconditioned system. Moreover, the number of iterations is largely independent of the system parameters, such as $\tau$, $n$, and $p$, which highlights the robustness of the algorithm. These results also allow for choosing the most suitable smoother, which, as in the $2$-$d$ case, is the Gauss-Seidel smoother.

\subsubsection{Test 5: the $3$-$d$ $\bm{H}_0(div, \Omega)$ problem} We now consider the $3$-$d$ $\bm{H}_0(div, \Omega)$ problem. Following a similar approach as the previous test, we study problem \eqref{eq:div-div} subject to Dirichlet boundary conditions with a right-hand side function given by
\begin{equation}\label{eq:3d-div-right-hand-side}
\bm{f}(x_1,x_2,x_3) = (x_2 x_3, x_1x_3, x_1 x_2), \quad (x_1,x_2,x_3) \in (0,1)^3.
\end{equation}
We focus on the \CG iterations of the preconditioned system, as the \CG iteration count in the case of the unpreconditioned system is similar to that of the $\bm{H}_0(\bm{curl}, \Omega)$ problem, and hence, we omit it. The results are shown in Table \ref{tab:3d-div-cgn-glt}. 

As it can be observed, the ASP-GLT algorithm shows different behavior compared to the $\bm{H}_0(\bm{curl}, \Omega)$ problem. In fact, Table \ref{tab:3d-div-cgn-glt} shows that although the ASP algorithm significantly reduces the number of iterations required, this number increases with the degree $p$. One solution to this issue is to replace the matrix $\bm{D}_{\bm{curl}}^{-1}$ in \eqref{eq:asp-div} with the symmetric Gauss-Seidel matrix associated with the matrix defined in \eqref{eq:div-second-smoother}. Specifically, we can replace it with:
\begin{equation}\label{eq:div-second-smoother-GS}
\bm{L}_{\bm{Q}_{\bm{curl}}}^{-1} - \bm{L}_{\bm{Q}_{\bm{curl}}}^{-1} \bm{Q}_{\bm{curl}} \bm{U}_{\bm{Q}_{\bm{curl}}}^{-1} + \bm{U}_{\bm{Q}_{\bm{curl}}}^{-1}.
\end{equation}
As usual, $\bm{L}_{\bm{Q}{\bm{curl}}}$ and $\bm{U}_{\bm{Q}{\bm{curl}}}$ represent the lower and upper parts of the matrix $\bm{Q}_{\bm{curl}}$, respectively.

We evaluated the effectiveness of our new strategy by conducting experiments, the results of which are summarized in Table \ref{tab:3d-div-cgn-GS-glt}. The table presents the number of \CG iterations required for various combinations of parameters $\tau$, $n$, and $p$, using both Jacobi and Gauss-Seidel smoothing schemes. 

The results presented in Table \ref{tab:3d-div-cgn-GS-glt} demonstrate that our proposed approach yields significant improvements compared to the results presented in Table \ref{tab:3d-div-cgn-glt}. Specifically, by replacing the diagonal matrix $\bm{D}_{\bm{curl}}^{-1}$ with \eqref{eq:div-second-smoother-GS}, we were able to improve the performance of our solver and control the $p$-dependency.

It is worth noting that in all of the developed tests, we used a MINRES solver in the GLT smoothing step. We tried several solvers, including GMRS and BICGSTAB, but the MINRES solver provided the best performance. However, there is another solver that yielded even more satisfactory results and can serve as an alternative to using the smoother matrix \eqref{eq:div-second-smoother-GS}: the flexible GCROT solver (see \cite{de1999truncation,hicken2010simplified}). In the following subsection, we will demonstrate the performance of using the flexible GCROT solver in the GLT smoothing step for both the $\bm{H}_0(\bm{curl}, \Omega)$ and $\bm{H}_0(div, \Omega)$ problems.

\subsubsection{Test 5:  evaluation of the $3$-$d$ ASP-GLT algorithm using a GCROT solver in the GLT smoothing step}
In this test, we consider the case studies of {\em Test 4} and {\em Test 5}. We consider only the case of a Jacobi smoother. In the case of $div$ problem the ASP preconditioner is the one given by \eqref{eq:asp-div}. The numerical results are shown in Table \ref{tab:3d-GCROT}. 

The results indicate that using the flexible GCROT algorithm during the GLT smoothing step significantly improves the performance of the ASP-GLT algorithm, particularly in the case of the $\bm{curl}$ problem. In fact, in this case, the algorithm behaves like a direct method; converging in just one iteration step.

\begin{table}[H]

\begin{adjustbox}{width=1.\textwidth}
\begin{subtable}{1.\textwidth} 
    \centering 
    \hspace*{-2.5 cm}
\begin{tabular}{c p{0.cm} cc p{0.cm} cc p{0.cm} cc p{0.cm} ccp{0.cm} cc p{0.cm} cc p{0.cm} cc p{0.cm} cc}
\cline{2-25}
\multirow{2}{*}{} &                      & \multicolumn{11}{l}{$p=3$}                                                                                        &  & \multicolumn{11}{l}{$p=4$}   \\   
\cline{3-13} \cline{15-25} 
                 \multirow{2}{*}{\diagbox[innerwidth=0.8cm]{\hspace*{0.15cm}$\tau$}{$n$}}   & \multicolumn{1}{c}{} & \multicolumn{2}{l}{$8$} &  & \multicolumn{2}{l}{$16$} &  & \multicolumn{2}{l}{$32$} &  & \multicolumn{2}{l}{$64$} &  & \multicolumn{2}{l}{$8$} &  & \multicolumn{2}{l}{$16$} &  & \multicolumn{2}{l}{$32$} &  & \multicolumn{2}{l}{$64$} \\ \cline{3-4} \cline{6-7} \cline{9-10} \cline{12-13} \cline{15-16} \cline{18-19} \cline{21-22} \cline{24-25} 
            &                      & curl         & div           &  & curl          & div           &  & curl          & div           &  & curl          & div           &  & curl         & div           &  & curl          & div           &  & curl          & div           &  & curl          & div           \\ \hline

$10^{-4}$ &&       $1$ & $1$       &&        $1$ & $1$         &&         $1$ & $1$          &&        $1$ & $1$          
          &&       $1$ & $3$       &&        $1$ & $1$         &&         $1$ & $1$          &&        $1$ & $1$            \\
$10^{-3}$ &&       $1$ & $2$       &&        $1$ & $2$         &&         $1$ & $1$          &&        $1$ & $3$            
          &&       $1$ & $1$       &&        $1$ & $2$         &&         $1$ & $2$          &&        $1$ & $5$            \\
$10^{-2}$ &&       $1$ & $1$       &&        $1$ & $1$         &&         $1$ & $2$          &&        $1$ & $4$            
          &&       $1$ & $1$       &&        $1$ & $1$         &&         $1$ & $2$          &&        $1$ & $2$            \\
$10^{-1}$ &&       $1$ & $2$       &&        $1$ & $1$         &&         $1$ & $1$          &&        $1$ & $2$           
          &&       $1$ & $1$       &&        $1$ & $1$         &&         $1$ & $1$          &&        $1$ & $2$            \\
$1$       &&       $1$ & $1$       &&        $1$ & $1$         &&         $1$ & $1$          &&        $1$ & $2$            
          &&       $1$ & $1$       &&        $1$ & $1$         &&         $1$ & $1$          &&        $1$ & $1$            \\
$10$      &&       $1$ & $2$       &&        $1$ & $1$         &&         $1$ & $1$          &&        $1$ & $1$            
          &&       $1$ & $1$       &&        $1$ & $1$         &&         $1$ & $1$          &&        $1$ & $2$            \\
$10^2$    &&       $1$ & $1$       &&        $1$ & $2$         &&         $1$ & $2$          &&        $1$ & $1$           
          &&       $1$ & $2$       &&        $1$ & $2$         &&         $1$ & $2$          &&        $1$ & $2$            \\
$10^3$    &&       $1$ & $1$       &&        $1$ & $1$         &&         $1$ & $1$          &&        $1$ & $1$            
          &&       $1$ & $1$       &&        $1$ & $1$         &&         $1$ & $1$          &&        $1$ & $1$            \\
$10^4$    &&       $1$ & $1$       &&        $1$ & $1$         &&         $1$ & $1$          &&        $1$ & $1$   
          &&       $1$ & $1$       &&        $1$ & $1$         &&         $1$ & $1$          &&        $1$ & $1$      
\end{tabular}
\end{subtable}
\end{adjustbox}

\vspace{.5cm}

\begin{adjustbox}{width=1.\textwidth}
\begin{subtable}{1.\textwidth} 
    \centering 
    \hspace*{-2.5 cm}
\begin{tabular}{c p{0.cm} cc p{0.cm} cc p{0.cm} cc p{0.cm} ccp{0.cm} cc p{0.cm} cc p{0.cm} cc p{0.cm} cc}
\cline{2-25}
\multirow{2}{*}{} &                      & \multicolumn{11}{l}{$p=5$}                                                                                        &  & \multicolumn{11}{l}{$p=6$}   \\   
\cline{3-13} \cline{15-25} 
                 \multirow{2}{*}{\diagbox[innerwidth=0.8cm]{\hspace*{0.15cm}$\tau$}{$n$}}   & \multicolumn{1}{c}{} & \multicolumn{2}{l}{$8$} &  & \multicolumn{2}{l}{$16$} &  & \multicolumn{2}{l}{$32$} &  & \multicolumn{2}{l}{$64$} &  & \multicolumn{2}{l}{$8$} &  & \multicolumn{2}{l}{$16$} &  & \multicolumn{2}{l}{$32$} &  & \multicolumn{2}{l}{$64$} \\ \cline{3-4} \cline{6-7} \cline{9-10} \cline{12-13} \cline{15-16} \cline{18-19} \cline{21-22} \cline{24-25} 
            &                      & curl         & div           &  & curl          & div           &  & curl          & div           &  & curl          & div           &  & curl         & div           &  & curl          & div           &  & curl          & div           &  & curl          & div           \\ \hline

$10^{-4}$ &&       $1$ & $3$       &&        $1$ & $2$         &&         $1$ & $1$          &&        $1$ & $1$          
          &&       $1$ & $3$       &&        $1$ & $3$         &&         $1$ & $2$          &&        $1$ & $4$            \\
$10^{-3}$ &&       $1$ & $1$       &&        $1$ & $2$         &&         $1$ & $2$          &&        $1$ & $2$            
          &&       $1$ & $1$       &&        $1$ & $2$         &&         $1$ & $1$          &&        $1$ & $2$            \\
$10^{-2}$ &&       $1$ & $1$       &&        $1$ & $1$         &&         $1$ & $1$          &&        $1$ & $3$            
          &&       $1$ & $1$       &&        $1$ & $1$         &&         $1$ & $1$          &&        $1$ & $7$            \\
$10^{-1}$ &&       $1$ & $2$       &&        $1$ & $1$         &&         $1$ & $1$          &&        $1$ & $1$           
          &&       $1$ & $1$       &&        $1$ & $1$         &&         $1$ & $1$          &&        $1$ & $2$            \\
$1$       &&       $1$ & $1$       &&        $1$ & $1$         &&         $1$ & $1$          &&        $1$ & $1$            
          &&       $1$ & $2$       &&        $1$ & $1$         &&         $1$ & $1$          &&        $1$ & $2$            \\
$10$      &&       $1$ & $2$       &&        $1$ & $1$         &&         $1$ & $1$          &&        $1$ & $2$            
          &&       $1$ & $2$       &&        $1$ & $2$         &&         $1$ & $1$          &&        $1$ & $2$            \\
$10^2$    &&       $1$ & $2$       &&        $1$ & $2$         &&         $1$ & $2$          &&        $1$ & $2$           
          &&       $1$ & $2$       &&        $1$ & $2$         &&         $1$ & $2$          &&        $1$ & $2$            \\
$10^3$    &&       $1$ & $1$       &&        $1$ & $1$         &&         $1$ & $1$          &&        $1$ & $1$            
          &&       $1$ & $1$       &&        $1$ & $1$         &&         $1$ & $2$          &&        $1$ & $2$           \\
$10^4$    &&       $1$ & $1$       &&        $1$ & $1$         &&         $1$ & $1$          &&        $1$ & $1$  
          &&       $1$ & $1$       &&        $1$ & $1$         &&         $1$ & $1$          &&        $1$ & $1$         
\end{tabular}
\end{subtable}
\end{adjustbox}
\caption{$3$-$d$ ASP-GLT preconditioning: \CG iterations with Jacobi smoothing are used for $\bm{H}_0(\bm{curl}, \Omega)$ (curl) and $\bm{H}_0(div, \Omega)$ (div) cases. The GLT smoothing step uses the flexible GCROT solver. The right-hand side functions are defined by \eqref{eq:3d-curl-right-hand-side} and \eqref{eq:3d-div-right-hand-side}. Parameter values are set to $\nu_1=1$, $\nu_2=p^2$, and $\nu_{asp}=1$.}  
\label{tab:3d-GCROT}  
\end{table}

\section{Conclusions}\label{conclusions}
In this work, we have proposed the use of the Auxiliary Space Preconditioning method (ASP) as a preconditioning strategy for linear $\bm{H}(\bm{curl}, \Omega)$ and $\bm{H}(div, \Omega)$ elliptic problems in Isogeometric Analysis (IgA). Our contributions include a uniform, stable, and regular decomposition of the IgA discrete spaces, and a proof of the mesh-independent effectiveness of the preconditioners using abstract theory by R. Hiptmair and J. Xu \cite{hiptmair2007nodal}. Our numerical simulations, conducted in two and three spatial dimensions, have confirmed the practical usefulness of our approach. Specifically, we have shown that our preconditioner significantly reduces the spectral condition number and that the number of conjugate gradient iterations required for convergence is independent of the discretization parameter. Additionally, we have demonstrated that the resulting algorithm can be easily extended to a $p$-stable algorithm. Our results suggest that the proposed preconditioning method is a promising candidate for future applications in isogeometric analysis.

We conclude this paper by presenting some future research perspectives selected from several possible ones. First, our results only apply to a parametric domain. Therefore, extending our method to more general physical domains is currently under development and will be the subject of future work. Another research direction involves extending our results to the case of variant coefficients, which is a wide-open subject and of great interest to investigate. Finally, our numerical tests indicate that the ASP-GLT algorithm has been highly effective. Therefore, a theoretical study of this algorithm would be beneficial, particularly with respect to the choice of solver in the GLT smoothing step, which needs more investigation.



\bibliographystyle{amsplain}
\providecommand{\bysame}{\leavevmode\hbox to3em{\hrulefill}\thinspace}
\providecommand{\MR}{\relax\ifhmode\unskip\space\fi MR }
\providecommand{\MRhref}[2]{%
  \href{http://www.ams.org/mathscinet-getitem?mr=#1}{#2}
}
\providecommand{\href}[2]{#2}

\end{document}